\documentclass[11pt,reqno]{amsart}
\usepackage{amsaddr}

\usepackage{algorithm}
\usepackage{algpseudocode}

\makeatletter
\newenvironment{breakablealgorithm}
  {% \begin{breakablealgorithm}
   \begin{center}
     \refstepcounter{algorithm}% New algorithm
     \hrule height.8pt depth0pt \kern2pt% \@fs@pre for \@fs@ruled
     \renewcommand{\caption}[2][\relax]{% Make a new \caption
       {\raggedright\textbf{\ALG@name~\thealgorithm} ##2\par}%
       \ifx\relax##1\relax % #1 is \relax
         \addcontentsline{loa}{algorithm}{\protect\numberline{\thealgorithm}##2}%
       \else % #1 is not \relax
         \addcontentsline{loa}{algorithm}{\protect\numberline{\thealgorithm}##1}%
       \fi
       \kern2pt\hrule\kern2pt
     }
  }{% \end{breakablealgorithm}
     \kern2pt\hrule\relax% \@fs@post for \@fs@ruled
   \end{center}
  }
\makeatother

\allowdisplaybreaks[1]

\title[Detecting Abrupt changes in Point Processes]{Detecting Abrupt changes in Point Processes:\\  Fundamental Limits and Applications}
\author{Anna Brandenberger, Elchanan Mossel and Anirudh Sridhar}
\address{Department of Mathematics, Massachusetts Institute of Technology}
\email{abrande@mit.edu, elmos@mit.edu, anisri@mit.edu}

\usepackage{setup}

\usepackage[normalem]{ulem}

\newcommand{\cascade}{\mathcal{I}}
\newcommand{\cut}{\mathsf{cut}}
\newcommand{\treeroot}{\mathsf{root}}

\newcommand{\Eabrupt}{\cE_{\mathrm{abrupt}}}

\newcommand{\Edisc}{\cE_{\mathrm{disc}}}
\newcommand{\Eapprox}{\cE_{\mathrm{approx}}}

\DeclareMathOperator{\err}{\mathsf{err}}
\newcommand{\spacing}{\kappa}
\newcommand{\derivative}{\Delta_{\delta}^{(\ell + 1)}}
\DeclareMathOperator{\dist}{\mathrm{dist}}
\DeclareMathOperator{\sep}{\mathsf{sep}}

\newcommand{\cadlag}{c\`adl\`ag }
\renewcommand{\emptyset}{\varnothing}
\newcommand{\highdeg}{\mathrm{highdeg}}

\DeclareMathOperator*{\argmax}{arg\,max}
\newcommand{\cpdist}{d_{\max}}
\def\draft{1}
\newcommand{\anna}[1]{\ifnum\draft=1\textcolor{teal}{[\textbf{Anna:} #1]}\fi}
\newcommand{\ani}[1]{\ifnum\draft=1\textcolor{red}{[\textbf{Ani:} #1]}\fi}

\begin{document}

\begin{abstract}
We consider the problem of detecting abrupt changes (i.e., large jump discontinuities) in the rate function of a point process. 
The rate function is assumed to be fully unknown, non-stationary, and may itself be a random process that depends on the history of event times. We show that abrupt changes can be accurately identified from observations of the point process, provided the changes are sharper than the ``smoothness'' of the rate function before the abrupt change. This condition is also shown to be necessary from an information-theoretic point of view.
We then apply our theory to several special cases of interest, including the detection of significant changes in piecewise smooth rate functions and detecting super-spreading events in epidemic models on graphs.
Finally, we confirm the effectiveness of our methods through a detailed empirical analysis of both synthetic and real datasets. 
\end{abstract}

\maketitle

\section{Introduction}

In the problem of change-point detection, a stochastic process is observed, and the goal is to detect, often as quickly as possible, when abrupt changes occur in the underlying data-generating mechanism. 
Motivated by applications in the sciences, operations, and engineering, we address the problem of change-point detection in point processes. 
In such situations, the observed data corresponds to the times of significant events that are generated by some physical law, and our goal is to detect abrupt changes in the rate of event occurrences.
For instance, in epidemiology, the times of reported infection events in a population are caused by disease-spreading interactions between individuals. 
Abrupt changes in the rate of infections can be caused by super-spreading events where many individuals are suddenly exposed to a pathogen.
Similarly, in geophysics, significant geological events such as high-magnitude earthquakes and volcanic eruptions can sharply change the rate of earthquake occurrences in localized regions. 
In both cases, timely detection of these changes is essential for effective response and damage mitigation. Additional applications include detecting abrupt shifts in financial market activity (see e.g.~\cite{pepelyshev2017real} and references therein), spikes in content virality in social media \cite{pinto2015trend}, and many others. 

Traditional approaches to change-point detection for point processes, such as the celebrated CUSUM procedure \cite{page_cusum, lorden1971procedures, moustakides1986optimal}, assume that the behavior of the observed time series is fully known before and after the change-point; only the location of the change-point itself is unknown. However, this is a highly unrealistic assumption in complex, non-stationary systems (e.g., those found in epidemics and geophysics), where baseline behaviors may evolve and post-change behaviors may be unpredictable.
Moreover, in the face of unknown or unpredictable systems, it can be challenging to distinguish genuine anomalies from normal fluctuations in event rates.
To overcome these issues, existing work often assumes that the rate function has a specific form (e.g., piecewise constant or is part of a parametric family), or that abrupt changes persist for a long period of time \cite{zhang2023sequential, wang2023detecting, brown2008bayesian, jansen2007multiscale, dion-blanc2024multiple, brown2018detecting, fromont2023minimax}.

In this work, we consider the change-point detection problem for point processes with a fully unknown rate function that may experience abrupt, transient changes. 
This is a setting that has not previously been studied in such generality, to the best of our knowledge.
Despite the complex issues that arise in this case, we design novel algorithms for change-point detection that are surprisingly simple, computationally efficient and statistically optimal.

\subsection{Summary of contributions}
We assume that the observed time series can be modeled by an inhomogeneous Poisson point process, governed by a (possibly random) time-varying rate function. At a high level, the rate function is assumed to be smooth, apart from large jump discontinuities that occur infrequently. Our goal is to estimate when these abrupt changes occur. We consider the most information-limited setting where the rate function is fully unknown \emph{a priori}, apart from mild conditions on its smoothness outside of abrupt changes.

The fundamental idea behind our solution is as follows: to detect whether a change-point occurred at time $t$, one needs to compare the observed point process after time $t$ to the \emph{hypothetical} behavior of the point process if there had been no change-point. We call this latter quantity the \emph{counterfactual trajectory} of the process. If the difference between the observed and counterfactual trajectories exceeds a threshold, we declare that a change-point has occurred. 
Since the counterfactual trajectory is of course unobservable, the main role of our change-point detection algorithm is to estimate it in a reasonable manner. 

To estimate the counterfactual trajectory, a key insight is that smooth functions can be approximated locally by polynomials (e.g., a Taylor expansion). 
In a bit more detail, to understand how the point process may evolve after time $t$, we can fit a polynomial approximation to the event data in a short window $[t- \delta, t]$, and use this approximation to forecast the process behavior in a short window $[t, t + \delta]$. 
Consequently, the behavior of the polynomial in $[t, t + \delta]$ serves as an estimate for the counterfactual trajectory, with improved accuracy as the polynomial degree increases.
In light of this idea, we test whether a change-point occurs at time $t$ by (1) computing the polynomial estimate of the counterfactual trajectory based on events in $[t  - \delta, t]$, and (2) comparing the difference between the observed and counterfactual processes in $[t, t + \delta]$ to a threshold. 
Using techniques from numerical analysis, we show that this strategy can be implemented extremely efficiently. In particular, the difference between the observed process and a degree $d$ polynomial estimate of the counterfactual is shown to be equal to the order $d+1$ discrete derivative of the point process with discretization $\delta$, which is a simple linear function of local behavior around $t$. As a result, our algorithm can be readily implemented in an online fashion, and is as efficient as possible.

We rigorously analyze the performance of this algorithm for detecting abrupt changes in a general class of point processes. In particular, we show that if the magnitude of the abrupt change exceeds a critical size, our algorithm successfully detects it with small delay. This critical size is fundamentally linked to the ``smoothness" of the rate function: relatively smooth functions cannot produce abrupt changes, making it easier to attribute significant shifts in behavior to an anomaly. Conversely, less smooth functions may naturally exhibit wild fluctuations, requiring detectable anomalies to be all the more significant.
Moreover, we show that our methods are \emph{statistically optimal}, in the sense that if an abrupt change is substantially smaller than the critical size, there exist cases where no algorithm can detect a change. In other words, our algorithms can detect changes whenever it is information-theoretically possible to do so. 

We remark that such a critical size arises because abrupt changes may be {transient}. Prior work on change-point detection largely assumes that the impact of abrupt changes \emph{persist} for a long time, hence more information can be leveraged to detect potentially small changes. However, such scenarios necessitate describing the long-term effects of changes, which can be particularly challenging in non-stationary, complex systems. Under the assumption that changes are transient, one must instead rely on \emph{local} characteristics of the observed process, which in particular eliminates the need to specify the full physical system in our analysis and algorithms.

As an application of our results, we characterize the fundamental limits of detecting abrupt changes that arise naturally in point processes on graphs. 
The connection between graph structure and change-point detection was recently explored by Mossel and Sridhar \cite{mossel2024finding} in the context of network epidemic models (i.e., the Susceptible-Infected process). In particular, they pointed out that high-degree vertices in the network in which a virus spreads cause an abrupt increase in the overall infection rate of an epidemic, similar to how a super-spreader event suddenly increases the infection rate of a pandemic. 
Our results applied to the detection of high-degree vertices in network epidemic models show that in an $n$-vertex network, vertices of degree larger than $\sqrt{n}$ can be detected (in the sense that we can conclude, with short delay, that such a vertex has been infected) with high probability. 
Together, our result along with the prior results of Mossel and Sridhar \cite[Theorem 4]{mossel2024finding} show that $D = \sqrt{n}$ marks a sharp transition in the detectability of vertices of degree at least $D$. This settles a main open question of Mossel and Sridhar. We also settle additional questions related to the sample complexity of exactly estimating high-degree vertices from multiple infection processes on the same graph; this is discussed in more detail in Section~\ref{sec:SIprocess}.

Finally, we complement our theoretical results with empirical analyses of synthetic and real datasets. Through simulations, we investigate the impact of our algorithm parameters (i.e., the polynomial degree and the window size $\delta$) in detecting abrupt changes of a certain size. We study change-point detection in two settings: (1) a point process with a fully unknown rate function that is smooth, except for a single abrupt change, and (2) the infection process corresponding to a Susceptible-Infected process on an unknown graph with a single high-degree vertex. 
Motivated by our algorithm's effectiveness in identifying abrupt changes in epidemic models, we provide some evidence, based on the impact of the 2020 Sturgis Motorcyle rally, that our algorithm can identify super-spreading events in real epidemic data, even when they may appear hidden within broader trends. 
The success of our methods in detecting abrupt changes in spite of noise and quantization demonstrates their robustness and highlights their potential for real-world applications.

\subsection{Related works} 

\subsubsection*{Change-point detection.} 
The problem of detecting a change in the distribution of a time series has been studied at least since the early 1900s; we defer the reader to the texts of Shiryaev \cite{shiryaev2008optimalstopping} and Poor and Hadjiliadis \cite{poor2008quickest} for general references on the topic. To address scenarios where observations are assumed to form an i.i.d.\ sequence of random variables before and after a change-point, the CUSUM procedure was proposed by Page to detect change-points in an online manner~\cite{page_cusum}; this method was later shown to be optimal \cite{lai1998information, lorden1971procedures, moustakides1986optimal, ritov1990decision}.
The CUSUM procedure forms the basis of change-point detectors for more complex, modern applications as well; we defer the reader to the recent review \cite{xie2021sequential_review} for more information. 

Change-point detection in Poisson processes -- commonly known as the Poisson disorder problem in the literature -- has been a topic of significant interest for decades. In the classical setting, first studied by Gal'chuk and Rozovskii \cite{galchuk1971disorder} and Davis \cite{davis1976poisson} in the 1970s, the goal is to estimate (in an online fashion) with respect to some loss function the point at which a Poisson process changes its rate from $\lambda_0$ to $\lambda_1$.
This problem was fully solved for a simple loss function (the expected detection delay) by Peskir and Shiryaev \cite{peskir2002solving} in 2002, and subsequent work studied generalizations to other loss functions, various priors for the change-point, and possibly random post-change rates \cite{bayraktar2005standard, bayraktar2006adaptive, brown2006optimal, bayraktar2006exponential, burnaev2009disorder, herberts2004optimal, bayraktar2009online}.
A focus of contemporary work is to understand change-point detection in more complex point processes, including compound point processes \cite{dayanik2006compound, dayanik2010compound, dayanik2008multisource, gapeev2005disorder, uru2023compound} and self-exciting processes (e.g., Hawkes processes) \cite{ludkovski2012bayesian, pinto2015trend, li2017detecting, wang2022sequential, zhang2023sequential, wang2023detecting}.

It is advantageous to develop algorithms that are adaptive to unknown rate functions that can undergo multiple change-points; however, this has received limited attention in the context of point processes.
In existing work addressing this angle, rate functions are assumed to be part of a parametric family (so that generalized likelihood ratios can be used, for instance), and that changes in behavior persist for a significant amount of time \cite{zhang2023sequential, bayraktar2006adaptive, wang2023detecting}.
The detection of abrupt but transient changes in point processes has been largely unexplored to the best of our knowledge, except for the recent work of Fromont, Grela and Le Gu\'{e}vel \cite{fromont2023minimax} which assumes that the rate function is a deterministic, piecewise constant function. 
Our work contributes to this important line of work by developing algorithms that provably detect multiple potential abrupt and transient changes, under fully unknown or unspecified rate functions.

\subsubsection*{Learning structure from dynamics}
As noted in \cite{mossel2024finding}, algorithms for change-point detection can be used to learn about the structure of certain dynamical systems on graphs, such as (stochastic) network epidemic models. In such settings, one observes a sequence of infections caused by the spread of an epidemic (more generally, a stochastic cascading process) on a graph. 
The \emph{structure learning} problem asks whether the graph structure can be recovered from the observed infection times, possibly using the outcomes from multiple independent epidemic processes on the same graph. Most of the prior work on this topic has developed algorithms for exact recovery of the underlying graph (i.e., estimating all edges with high probability); see, e.g.,  \cite{ACFKP13_trace_complexity, NS12_cascades, hoffman2019learning, hoffman2020learning, khim2018theory}. The recent work of Mossel and Sridhar \cite{mossel2024finding} has shown that using far less information than such algorithms, important aspects of the underlying graph -- such as high-degree vertices -- can still be estimated, even when no edges can be learned with certainty. Our work builds on their methods and solves key open problems posed in their work. Additional work has explored the structure learning problem for other types of dynamical processes over networks, such as Glauber dynamics and Hamiltonian dynamics \cite{bresler2017learning,gaitonde2024unified,gaitonde2024efficiently, bakshi2024structure}. We defer the interested reader to \cite{mossel2024finding} for further related references on graphical models and the structure learning problem. 

\subsection{Future directions}

There are many possible directions for future work on change-point detection; a few important avenues are outlined below.

\begin{itemize}

    \item {\bf Other types of change-points.} The present work focuses on detecting large jump discontinuities in the underlying rate function. One may also consider detecting points of non-smoothness (i.e., a point process analogue of the framework in \cite{yu2022localising}) or points where the rate function changes by a multiplicative factor; existing frameworks have been developed by \cite{elkaroui2017minimax, wang2023detecting}. 
    Additionally, while our methods are well-suited to detecting \emph{transient} changes, it would be interesting to explore whether better methods exist for detecting changes that persist over time in unknown, non-stationary environments.

    \item {\bf Realistic problem features.} In practice, the detection and reporting of events generated by a point process may experience noise in the form of missing information and quantization (i.e., if an event is missed or events are counted in binned periods of time).
    Developing robust algorithms for point process data that can tolerate such noise is an important, ongoing objective in the field; see, e.g., \cite{cheng2024point} and references therein.   
    While our empirical results indicate that our methods may be robust to this type of noise, it would be useful to develop a better understanding of fundamental roadblocks that arise, from the lens of theory or more detailed case studies. 

    \item {\bf More complex models.} Point processes are the fundamental building blocks for more complex processes, such as Hawkes processes and compound Poisson processes. Similarly, many generalizations of the Susceptible-Infected process exist (e.g., variations that account for periods of exposure and recovery).
    How can we extend our framework and algorithms to these more complicated but realistic settings?
    
\end{itemize}

\subsection{Basic notation} Throughout, we use standard Landau notation $o(\cdot)$, $O(\cdot)$, $\omega(\cdot)$ and $\Omega(\cdot)$. We also write $f \ll g$ and $f \gg g$ to respectively denote $f = o(g)$ and $f = \omega(g)$. 
We use standard notation for number sets. In particular, $\mathbb{R}$ and $\mathbb{Z}$ denote the set of real numbers and integers, respectively. For a positive integer $k$, we denote the set $[k] : = \{ 1, \ldots, k \}$.
For $x,y \in \mathbb{R}$, we will sometimes use the shorthand $x \lor y : = \max \{ x, y \}$ and $x \land y : = \min \{x, y \}$.
For a function $f : \mathbb{R} \to \mathbb{R}$ and a positive integer $k$, we denote $f^{(k)}$ to be its order $k$ derivative, whenever it is well-defined.
For a finite set $S \subset \R$, the \emph{minimum separation} is defined to be
\[
\sep(S) : = \min \{ |a - b| : a, b \in S \text{ and } a \neq b \},
\]
with $\sep(S) = \infty$ if $S = \emptyset$.

\subsection{Organization} In Section~\ref{sec:def_results}, we rigorously define our point process model and state our main change detection result. We then discuss applications to the case of rate functions that are smooth except for jump discontinuities (Section~\ref{subsec:smooth-jump-results}), as well as the case of detecting high-degree vertices in the Susceptible-Infected process (Section~\ref{subsec:SI-results}).
In Section~\ref{sec:techniques}, we motivate and describe our change-point detection procedure (Algorithm~\ref{alg:derivative_thresholding}), then outline the method of proof of our main result (Theorem~\ref{thm:derivative_thresholding}) which proceeds by (i) polynomially approximating the point process (Theorem~\ref{thm:approximation_error}), and (ii) showing that the counterfactual residual is captured by higher-order derivatives (Lemma~\ref{lemma:derivative_characterization}). In Section~\ref{sec:empirical} we present our empirical analyses. 
In Sections~\ref{sec:polynomial_approximation} and \ref{sec:higher_derivatives} we respectively prove Theorem~\ref{thm:approximation_error} and Lemma~\ref{lemma:derivative_characterization}, which completes the proof of our main result. In Sections~\ref{sec:smoothjump} and \ref{sec:SIprocess} we prove our results for the two applications described respectively in Sections~\ref{subsec:smooth-jump-results} and \ref{subsec:SI-results}.

\section{Definitions and Main Results}
\label{sec:def_results}
We assume that all random processes are defined on a common probability space equipped with a filtration $\{ \cF_t \}_{t \ge 0}$.
The rate function $\Lambda(t)$ of a point process $\{ N(t) \}_{t \ge 0}$ is given by
\begin{equation}
\label{eq:point_process_rate}
\Lambda(t) : = \lim_{\epsilon \to 0} \frac{ \E [ N(t + \epsilon) \vert \cF_t] - N(t)}{\epsilon}.
\end{equation}
Given observations of the point process in a time interval $[0,T]$, our goal is to identify when abrupt changes (i.e., large jump discontinuities) occurred in $\Lambda$.
In the most general version of our results, we allow $\Lambda(t)$ to be a $\cF_t$-adapted random process undergoing several abrupt changes; this level of generality captures several important cases of interest, such as self-exciting point processes and epidemic models.

To clearly convey the intuition behind our results, we focus for now on a simpler setting, where $\Lambda(t)$ is an unknown, deterministic function that is smooth outside of a single abrupt jump discontinuity.
In particular, we assume the following simple representation for the sake of exposition: 
\begin{equation}
\label{eq:basic_rate_function}
\Lambda(t) := B x(t) + A y(t - t^*) \mathbf{1}(t \ge t^*),
\end{equation}
where $x, y: \mathbb{R}_{\ge 0} \to \mathbb{R}_{\ge 0}$ are (a priori unknown) smooth functions with $y(0) > 0$, $t^* \in \mathbb{R}_{\ge 0} \cup \{ \infty \}$ is the time of the abrupt change (where, in the case $t^* = \infty$, no change occurs), and $A,B > 0$ are positive constants. 
One may interpret $B$ as the magnitude of \underline{b}aseline behaviors before the change, and $A$ as the magnitude of the \underline{a}brupt change at $t^*$.
Given observations of the point process generated by $\Lambda$ in a finite time window $[0,T]$, our goal is to estimate $t^*$.

Our results concern the detectability of \emph{abrupt} changes, which we take to mean that $A$ is large while the parameters $x,y, T$ are fixed. 
Of particular interest to us are cases where it additionally holds that $A \ll B$. In such scenarios, the change is indeed abrupt, but it can be overshadowed or possibly explained by larger trends in the baseline behavior of the rate function. 
This phenomenon is concretely illustrated in the following example, which shows that it is generally information-theoretically impossible to detect change-points under the rate function \eqref{eq:basic_rate_function} if $A \ll \sqrt{B}$.

\begin{prop}
\label{prop:impossibility}
Fix $T > 1$. 
Suppose that the rate function has the form 
\[
\Lambda(t) = B + A e^{-(t - 1)} \mathbf{1}(t \ge 1).
\]
If $A / \sqrt{B} \to 0$, then no estimator can correctly detect a jump at $t = 1$ with probability greater than $1/2 + o(1)$.
\end{prop}

The idea behind Proposition \ref{prop:impossibility} is that when $A \ll \sqrt{B}$, the magnitude of the abrupt change is smaller than the inherent randomness of the point process. Indeed, due to the baseline rate of $B$, we should expect the number of observed events in any constant-sized interval to experience normal fluctuations on the order of $\sqrt{B}$ around its mean. Since the impact of the change-point is non-negligible for only a finite amount of time due to the exponential decay, it is indistinguishable from the normal fluctuations of the process when $A \ll \sqrt{B}$. The full proof details can be found in Section \ref{sec:smoothjump}.

Although the change-point's impact overcomes the inherent randomness of the point process when $A$ is larger than $\sqrt{B}$, the fact that $\Lambda$ is fully unknown poses a major roadblock for detection. Indeed, if $\Lambda$ is unknown, it is unclear whether perceived changes in the sequence of events are due to natural variations in the rate function, or due to an abrupt change.
This issue is particularly problematic when $A$ is orderwise smaller than $B$, as the abrupt change caused by $y$ may be overshadowed by larger and more significant trends in $x$.
Our main result shows that, perhaps surprisingly, this identifiability issue can be fully circumvented: whenever $A$ is much larger than $\sqrt{B}$, change-points can be accurately detected. 

\begin{thm}[Informal and simplified version of Theorem \ref{thm:derivative_thresholding}]
\label{thm:main_simplified}
Let $T > 0$, and suppose that the point process given by the rate function \eqref{eq:basic_rate_function} is observed in the time window $[0,T]$.
If $A / \sqrt{B} \to \infty$, then there is an algorithm such that with probability $1 - o(1)$, the following hold:
\begin{enumerate}
    \item (No false alarms.) If $t^*= 0$ or $t^* \ge T$, then the algorithm outputs $\emptyset$. 
    \item (Accurate estimation.) If $t^* \in (0,T)$, then the algorithm outputs $\hat{t}$ satisfying $| \hat{t} - t^*| = o(1)$.
\end{enumerate}
\end{thm}

In words, we show that it is possible to detect abrupt changes of size larger than $\sqrt{B}$. Additionally, our algorithm does not raise a false alarm (i.e., output an estimated change-point when there is none), with high probability. 

The idea behind Theorem \ref{thm:main_simplified} is that a change-point occurring at time $t$ can be readily detected if we knew how the rate function would evolve if an abrupt change had \emph{not} happened at time $t$. If we somehow had access to this counterfactual trajectory, we could compare the rate of events in the observed process to the counterfactual after time $t$. This difference would be of order $A$ if there was indeed an abrupt change at time $t$. Otherwise, the difference would capture the fluctuations of the point process around its mean, which are of order at most $\sqrt{A + B}$.

The crux of the problem therefore reduces to estimating the counterfactual trajectory. We do so by building a predictor of events after time $t$ based only on observations in a short time window \textit{before} $t$. 
If the predicted number of events (the estimated counterfactual) differs significantly from the observed number of events shortly after $t$, we declare a change-point. 
More details on how we design our predictor can be found in Section \ref{sec:techniques}.

\begin{rem}[Asymptotic notation]
The various asymptotic quantities in Theorem \ref{thm:main_simplified} (i.e., $A, B$ and $| \hat{t}-  t^*|$) can be made more precise. In our general results, we show that if $A \ge B^{1/2 + \epsilon}$ for some $\epsilon > 0$, then $|\hat{t} - t^*| \lesssim 1/B^{2 \epsilon}$.
More details can be found in Section \ref{subsec:smooth-jump-results} and Theorem \ref{thm:smooth-plus-jump}.
\end{rem}

\begin{rem}[Non-asymptotic results]
Our general results provide insight into non-asymptotic regimes as well, in which case more assumptions on the smoothness of $x,y$ are needed. Some nuances that arise include:
\begin{itemize}
    \item The function $y$ cannot decay to 0 too quickly, otherwise changes will be too transient to be detectable. We formalize this by assuming a bound on higher-order derivatives of $y$. See Definition \ref{def:Eabrupt} for more details.
    \item If higher-order derivatives of $x$ are large, then potential fluctuations in $x$ can mask abrupt changes caused by $y$. We control for such effects by assuming a bound on higher-order derivatives of $x$. See Assumption \ref{as:regularity} for more details.
\end{itemize}
As a result of these effects, in our general results the parameter $B$ in Theorem \ref{thm:main_simplified} is replaced by a different parameter that captures the ``smoothness" of $x$; more details can be found in Section \ref{sec:techniques}. We note that for the setting of Theorem \ref{thm:main_simplified}, this smoothness parameter is equal to $B$ up to constant factors.
\end{rem}

\begin{rem}[Applicability to more complex models]
Our general results (Theorem~\ref{thm:derivative_thresholding}) go beyond the simplified setting we have discussed so far in several ways: there may be multiple abrupt changes, there may exist small jumps in the rate function in addition to large, abrupt jumps, and $\Lambda$ itself is allowed to be a random process with \cadlag sample paths. This latter generalization is especially important for classes of self-exciting point processes such as epidemic models and Hawkes processes. From a technical perspective, the essential ideas are the same for analyzing deterministic and random rate functions, though several nuances need to be taken into consideration. 
See Section~\ref{subsec:general_result} for more details. 
\end{rem}

\subsection{Detecting multiple change-points}\label{subsec:smooth-jump-results}

In this section, we discuss our results on detecting multiple change-points in the case where $\Lambda$ is an unknown piecewise smooth function that experiences abrupt jump discontinuities.
In particular, we consider the following generalization of \eqref{eq:basic_rate_function}.

\begin{defn}[Smooth + jump rate function]
\label{def:smooth_plus_jump}
Let $T > 0$ be a time horizon, and let $\cD \subset [0,T]$ be a finite set with $0 \in \cD$. Let $t_1, \ldots, t_m$ denote the elements of $\cD$ sorted in increasing order. Additionally, let $\{A_i \}_{i = 1}^m$ be a collection of positive constants, and let $\{ x_i \}_{i =1}^m$ be a collection of smooth functions satisfying $x_i : \mathbb{R}_{\ge 0} \to \mathbb{R}_{\ge 0}$ and $x_i(0) > 0$ for all $i \in \{1, \ldots, m \}$. 
We say that $\Lambda$ is a \emph{smooth + jump rate function} with respect to $\cD$, $\{A_i \}_{i = 1}^m$, and $\{ x_i \}_{i = 1}^m$ if it has the form
\begin{equation}
\Lambda(t) = \sum_{i =1}^m A_i x_i(t- t_i) \mathbf{1}(t \ge t_i).
\end{equation}
\end{defn}

We remark that the conditions $x_i(0) > 0$ and $A_i > 0$ for all $i \in \{1, \ldots, m \}$ ensure that all the jumps are positive. In principle, our results can also handle negative jumps, but one would have to impose additional conditions to ensure that the rate function is non-negative. As before, we will assume that $\cD$ and $\{ x_i \}_{i = 1}^m$ are fixed as the $A_i$'s grow large. 

We next discuss the characteristics of jumps in $\Lambda$. We allow for small jumps in addition to the large, abrupt jumps we aim to detect, provided they are well-separated in magnitude. 

\begin{assumption}[Jump sizes]
\label{as:smooth_jump_jump_sizes}
Assume that there exists $\theta > 0$ such that for all $j \in \{ 1, \ldots, m \}$, either $A_j \ge ( \sum_{i = 1}^m A_i )^\theta$ or $A_j = (\sum_{i = 1}^m A_i)^{o(1)}$, where $o(1) \to 0$ as $\sum_{i = 1}^m A_i \to \infty$. The set of abrupt jumps $\cT$ is given by all those in the former case, that is
\[
\cT : = \left \{ t_j \in \cD \setminus \{0 \} : A_j \ge \Big( \sum_{i = 1}^m A_i \Big)^\theta \right \}.
\]
\end{assumption}

We remark that the quantity $\sum_{i = 1}^m A_i$ serves as an orderwise upper bound for the maximum value of the rate function, and essentially takes on the same role as the parameter $B$ in Theorem \ref{thm:main_simplified}. With this interpretation, the set $\cT$ therefore captures jumps that are larger than some power of the maximum baseline rate of events. We exclude possible jumps at $t = 0$ due to the lack of any information before this point in time.

Before stating our results, we introduce some useful metrics for assessing the effectiveness of change-point detection algorithms. 
Given a set of estimated change-points $\widehat{\cT}$, a basic requirement is that all change-points should be detected; that is, $| \widehat{\cT}| = | \cT|$. Given that $| \widehat{\cT}| = | \cT|$, we can also measure the accuracy of estimated change-points through the following metric. 

\begin{defn}
Let $\cS, \cT$ be finite subsets of $\R$ with $|\cS| = |\cT|$. 
If $\cS = \cT = \emptyset$, we let $\cpdist(\cS, \cT) = 0$.
Otherwise, write the elements of $\cS$ and $\cT$ in increasing order as $s_1, \ldots, s_m$ and $t_1, \ldots, t_m$, respectively. We define the maximal distance between ordered elements as
\[
\cpdist(\cS, \cT) : = \max_{i \in \{1, \ldots, m \}} |s_i - t_i |.
\]
\end{defn}

We remark that these metrics are natural, and have been used in prior work on detecting and estimating multiple change-points; see for instance \cite{wang2023detecting}.
With these definitions and assumptions in place, we have the following result on change-point detection for the smooth + jump model. The full proof details can be found in Section \ref{sec:smoothjump}.

\begin{thm} 
\label{thm:smooth-plus-jump}
Let $T, \epsilon, \spacing > 0$ be fixed.
Suppose that: 
\begin{enumerate}
    \item $\Lambda$ is a smooth + jump rate function as described in Definition \ref{def:smooth_plus_jump};
    \item Assumption \ref{as:smooth_jump_jump_sizes} holds with $\theta = 1/2 + \epsilon$;
    \item it holds that $\cT  \subset (0,T)$ and $\sep(\cT) \ge \spacing$.
\end{enumerate}
Then there is an algorithm that outputs a set of time indices $\widehat{\cT}$ such that, as $\sum_{i = 1}^m A_i \to \infty$, it holds with probability $1 - o(1)$ that $| \cT | = | \widehat{\cT}|$ and $\cpdist(\cT, \widehat{\cT}) \le ( \sum_{i = 1}^m A_i )^{- 2\epsilon + o(1)}$.
\end{thm}

The theorem naturally generalizes Theorem \ref{thm:main_simplified} to the case of multiple change-points, and states that all abrupt changes of magnitude larger than $\sqrt{\sum_{i = 1}^m A_i}$ can be accurately detected. 
The only significant addition to this theorem is the condition that $\sep(\cT) \ge \spacing$ for some $\spacing  >0$; that is, all abrupt changes are separated by at least $\spacing$. This condition is useful in eliminating edge cases where a cluster of change-points may occur in quick succession, making it difficult to discern the correct number of change-points from the data. The algorithm used in Theorem \ref{thm:smooth-plus-jump} crucially uses the value of $\spacing$ to ensure that all change-points are properly identified.

\subsection{Detecting abrupt changes in network epidemic models}\label{subsec:SI-results}

Abrupt changes of the type we study arise naturally in point processes driven by complex networks. 
This idea was recently studied by Mossel and Sridhar \cite{mossel2024finding} in the context of network epidemic models. Our general results on change-point detection, translated to the setting of network epidemics, settle open questions from \cite{mossel2024finding}. We elaborate on the details below.

Our results are based on the Susceptible-Infected (SI) process on graphs, which is a point process that models the occurrence of infection events as an epidemic spreads through a population (see \cite[Chapter 7]{brauer2012mathematical} and references therein). 
At a high level, the model operates as follows. There is a population of individuals who are represented by the vertices of a graph $G = (V,E)$. Edges in $G$ represent potential interactions that can spread a disease from one individual to another.
Initially, only a single individual is infected with the disease, while all others are susceptible. Over time, the disease spreads via the edges of $G$ to the rest of the population (more precisely, it spreads to the connected component of $G$ containing the initial infective). 
The point process of interest is given by the sequence of vertex infection times caused by the epidemic.   
We remark that while the SI process does not account for many realistic aspects of epidemic data (e.g., recovery, re-infection, delays between infection and symptom onset), it serves as a useful baseline model for mathematical analysis. 

To formally define the model, we let $\cascade(t) \subset V$ denote the set of vertices that have been infected until and including time $t$. Initially, we have $\cascade(0) = \{v_0 \}$, where $v_0$ is the initial infective. 
Given $\cascade(t)$, the probability that an uninfected vertex $v \in V \setminus \cascade(t)$ becomes infected shortly after time $t$ is given by 
\begin{equation}\label{eq:SI-evolution}
    \p ( v \in \cascade(t + \epsilon) \vert \cascade(t) ) = \epsilon | \cN(v) \cap \cascade(t) | + o(\epsilon),
\end{equation}
where $o(\epsilon) \to 0$ faster than $\epsilon$ as $\epsilon \to 0$.
That is, edges from infected vertices spread the infection at rate 1 (which we set for notational simplicity; our results hold for any other constant infection rate). 
The observed point process is $I(t) : = | \cascade (t)|$, which is equal to the number of individuals infected until (and including) time $t$.

It was shown in \cite{mossel2024finding} that when a vertex of large degree (i.e., a super-spreader) gets infected, many of its neighbors become suddenly exposed to the disease, which causes an abrupt increase in the rate function of $I(t)$.
To see this more formally, let $\{ \cF_t \}_{t \ge 0}$ be the natural filtration corresponding to the evolution of the SI process. Then the rate function of $I(t)$ is given by
\begin{align*}
\Lambda(t) & : = \lim_{\epsilon \to 0} \frac{\E [ I(t + \epsilon) - I(t) \vert \cF_t]}{\epsilon} \\
& = \lim_{\epsilon \to 0} \frac{1}{\epsilon} \sum_{v \in V \setminus \cascade(t)} \p ( v \in \cascade(t + \epsilon) \vert \cF_t ) \\
& = \sum_{v \in V \setminus \cascade(t)} | \cN(v) \cap \cascade(t) | \\
& =  \cut( \cascade(t) ),
\end{align*}
where, for a set $S \subseteq V$, $\cut(S)$ is the \emph{graph cut} corresponding to $S$; i.e., the number of edges in $G$ between $S$ and $V \setminus S$. 
Observe that $\Lambda(t)$ as defined above is piecewise constant, with jumps that coincide with infection events. In a bit more detail, if vertex $v$ becomes infected at time $t$, the jump in the rate function is given by 
\[
\Lambda(t) - \Lambda(t^-) = | \cN(v) \setminus \cascade(t) | - | \cN(v) \cap \cascade(t) |  =   \deg(v) - 2 | \cN(v) \cap \cascade(t) | .
\]
When $\deg(v)$ is large enough, it holds with high probability that $| \cN(v) \cap \cascade(t) | \ll \deg(v)$; see \cite[Lemma 15]{mossel2024finding}.
Consequently, the infection events of high-degree vertices can be identified by estimating when correspondingly large abrupt changes occur in the rate function of $I(t)$.

To formalize this approach to detecting high-degree vertices, it is useful to have a well-defined notion of ``high degree" and ``low degree". 
Formally, we assume that most vertices in $G$ have a relatively small degree, except for at most a constant number of vertices of substantially large degree. We precisely capture this structural property through the family of graphs described below.

\begin{defn}
\label{def:G_conditions}
We say that $G \in \cG(n,m,p,d,D)$ if and only if (1) $G$ is a connected graph on $n$ vertices; 
(2) there are at most $m$ vertices of degree at least $D$ and all other vertices have degree at most $d$;
(3) if $\deg(u), \deg(v) \ge D$ then the shortest-path distance between $u$ and $v$ is at least $p$.
For brevity, we will often instead write $\cG$ when the values of $n, m, p, d, D$ are clear.
Additionally, for $G \in \cG$ we define 
$\highdeg(G)$ to be the set of vertices of degree larger than $D$.
\end{defn}

We make a few remarks about the definition above. 
The assumption of connectivity is necessary, since if certain vertices are never infected, it is impossible to learn anything about them. 
Additionally, the assumption that high-degree vertices are separated by some minimal distance $p$ ensures that the \emph{temporal} effects of distinct high-degree vertices can also be distinguished.\footnote{The choice of $p$ serves a similar purpose as the condition $\sep( \cT) \ge \spacing$ in the statement of Theorem \ref{thm:smooth-plus-jump}. For further details on the impact of $p$, see Section~\ref{sec:SIprocess}.}
While control over the minimal separation was not needed in \cite{mossel2024finding}, our results tackle a more subtle regime, and more precise control over structural aspects of $G$ is needed to avoid overcomplicating our analysis.

Next, we assume the following about the scaling of the parameters $m,p, d, D$ with respect to $n$.

\begin{assumption}
\label{as:graph_parameters}
We assume that as $n \to \infty$, $m$ is constant, $p = \omega(1)$, $d = n^{o(1)}$ and $D = n^\alpha$ for some $\alpha \in (0,1)$.
\end{assumption}

This assumption is same as the one made by Mossel and Sridhar \cite{mossel2024finding}, except for the additional condition on $p$, which ensures that high-degree infection events are temporally separated. We remark that if $m \in \{0,1 \}$, then the value of $p$ plays no role in the structural properties of $G$.

Our main results for the inference of high-degree vertices in the SI process settle two fundamental questions:
\begin{enumerate}
    \item \emph{(Detection)} When can the presence of high-degree vertices in $G$ be detected?
    \item \emph{(Estimation)} When can high-degree vertices be precisely identified, possibly using additional information?
\end{enumerate}

We show that for both questions, there is a sharp phase transition at $\alpha = 1/2$ between the possibility and impossibility of solution concepts.
In particular, this settles the main open questions raised by Mossel and Sridhar \cite{mossel2024finding} in recent work.

\subsubsection{Detecting high-degree vertices.}
We first consider the problem of \emph{detecting the presence of high-degree vertices}.
We do so by
estimating the infection times of high-degree vertices, given by 
\[
\cT : = \{ t_v : v \in \highdeg(G) \},
\]
which correspond to abrupt changes in the rate function.

\begin{thm}[Detecting high-degree vertices]
\label{thm:detecting_high_degree}
Suppose that $G \in \cG$ satisfies Assumption \ref{as:graph_parameters} with $\alpha \in (1/2, 1)$. 
Then there is an algorithm that outputs a set of time indices $\widehat{\cT}$ such that, as $n \to \infty$, it holds with probability $1 - o(1)$ that $| \cT | = | \widehat{\cT}|$ and $\cpdist(\cT, \widehat{\cT}) \le n^{-(2\alpha - 1) + o(1)}$.
\end{thm}

In other words, Theorem \ref{thm:detecting_high_degree} shows that for $G \in \cG$, high-degree vertices with at least $n^{1/2 + \epsilon}$ neighbors for some $\epsilon > 0$ can be detected with high probability. Furthermore, Theorem \ref{thm:detecting_high_degree} generalizes a recent result of Mossel and Sridhar \cite{mossel2024finding}, who proved a version of Theorem \ref{thm:detecting_high_degree} for $\alpha > 3/4$.
To establish our results, it is not enough to sharpen the techniques of Mossel and Sridhar; we make use of new change-point estimators that overcome fundamental limitations of their prior work. We defer a detailed discussion to Sections~\ref{sec:techniques} and~\ref{sec:SIprocess}.

In light of the recent work \cite{mossel2024finding}, Theorem \ref{thm:detecting_high_degree} is essentially best possible: if $\alpha < 1/2$, there exist cases where it is information-theoretically impossible to tell whether there even exists a high-degree vertex in the graph. Formally, we have the following result.

\begin{thm}[Impossibility of detection]
\label{thm:high_degree_detection_impossibility}
Let $\alpha \in (0, 1/2)$ and suppose that $D = n^\alpha$, $d \ge \log^2 n,$ and $m = 1$.
Then there exists an ensemble $\cH \subset \cG$ and a probability distribution $\mu$ on graphs in $\cH$ and source vertices $v_0$ such that it is impossible to tell with probability greater than $1/2 + o(1)$ whether there exists a high-degree vertex in $(G,v_0) \sim \mu$.
\end{thm}

While not explicitly stated in \cite{mossel2024finding}, the result follows readily from their arguments. We defer the details to Section~\ref{subsec:SI-detection-proofs}. 

\subsubsection{Estimating high-degree vertices}

Information from a single cascade trace is not enough to precisely identify the set of high-degree vertices. 
Indeed, if vertices $u$ and $v$ are infected in quick succession, it becomes difficult to determine which vertex caused an abrupt jump in the infection rate.
To mitigate this identifiability issue, one can combine information from multiple cascade traces on the same graph.
To combine information in a reasonable way, we will assume access to the \emph{cascade trace}, which is the collection of infection times from the SI process and associated vertices. Formally, the cascade trace is given by $\mathbf{T} : = \{ (v, t_v) \}_{v \in V}$.
Given a positive integer $K$ and a collection of source vertices $\{ v_i \}_{i = 1}^K$, we observe the outcome of $K$ independent SI processes on the same graph $G$ which begin at different source vertices, in the form of the independent cascade traces $\mathbf{T}_1, \ldots, \mathbf{T}_K$.
Our main result on estimating high-degree vertices shows that when $\alpha > 1/2$, a constant number of cascade traces suffices to estimate them exactly.

\begin{thm}[Estimating high-degree vertices, $m = 1$]
\label{thm:estimating_high_degrees}
Suppose that $\alpha \in (1/2, 1)$. If $K > 1 / (2\alpha - 1)$, then there is an estimator $\widehat{\mathrm{HD}}$ such that for any $G \in \cG$ and any collection of source vertices $\{v_{i} \}_{i = 1}^K$, it holds with probability $1 - o(1)$ as $n \to \infty$ that $\widehat{\mathrm{HD}}(\mathbf{T}_1, \ldots, \mathbf{T}_K) = \highdeg(G)$.
\end{thm}

The proof follows the approach of Mossel and Sridhar \cite{mossel2024finding}. In light of Theorem \ref{thm:detecting_high_degree}, the absolute difference between the infection time of a high-degree vertex and one of the estimated change-points is at most $n^{- (2 \alpha - 1) + o(1)}$ in all $K$ cascades.
On the other hand, we show that the probability that the infection time of a given low-degree vertex is within $n^{- (2 \alpha - 1) + o(1)}$ of an estimated change-point in all $K$ cascades is at most $n^{- K(2 \alpha - 1) + o(1)}$.
For $K > 1 / (2 \alpha - 1)$, this probability is $o ( 1/n)$.
Consequently, the proposed estimator, which identifies all vertices within $n^{- (2 \alpha - 1) + o(1)}$ of an estimated change-point across all cascades, correctly outputs the set of high-degree vertices with probability $1 - o(1)$.

Our result on estimation complements the following impossibility result established in \cite{mossel2024finding}.

\begin{thm}[Impossibility of estimation \cite{mossel2024finding}]
\label{thm:high_degree_estimation_impossibility}
Fix $\epsilon > 0$ and let $\alpha \in (0,1/2)$. Suppose that $D = n^\alpha$, $d \ge \log^2 n$ and $m = 1$. Additionally, assume that 
\[
K \le \left( \frac{1 - 2 \alpha - \epsilon}{5} \right) \frac{\log n}{\log \log n}.
\]
Then there exists an ensemble $\cH \subset \cG$ and a probability distribution $\mu$ over graphs in $\cH$ and source vertices $\{ v_1, \ldots, v_K \}$ such that, for any estimator $\mathrm{HD}$, 
\[
\p_{G, v_1, \ldots, v_K \sim \mu} ( \mathrm{HD} ( \mathbf{T}_1, \ldots, \mathbf{T}_K ) = \highdeg(G) ) = o(1),
\]
where $\p_{G, v_1, \ldots, v_K \sim \mu}$ denotes the probability distribution over cascade traces induced by $\mu$.
\end{thm}

\begin{figure}[t]
  \centering
  \includegraphics[width=0.8\linewidth]{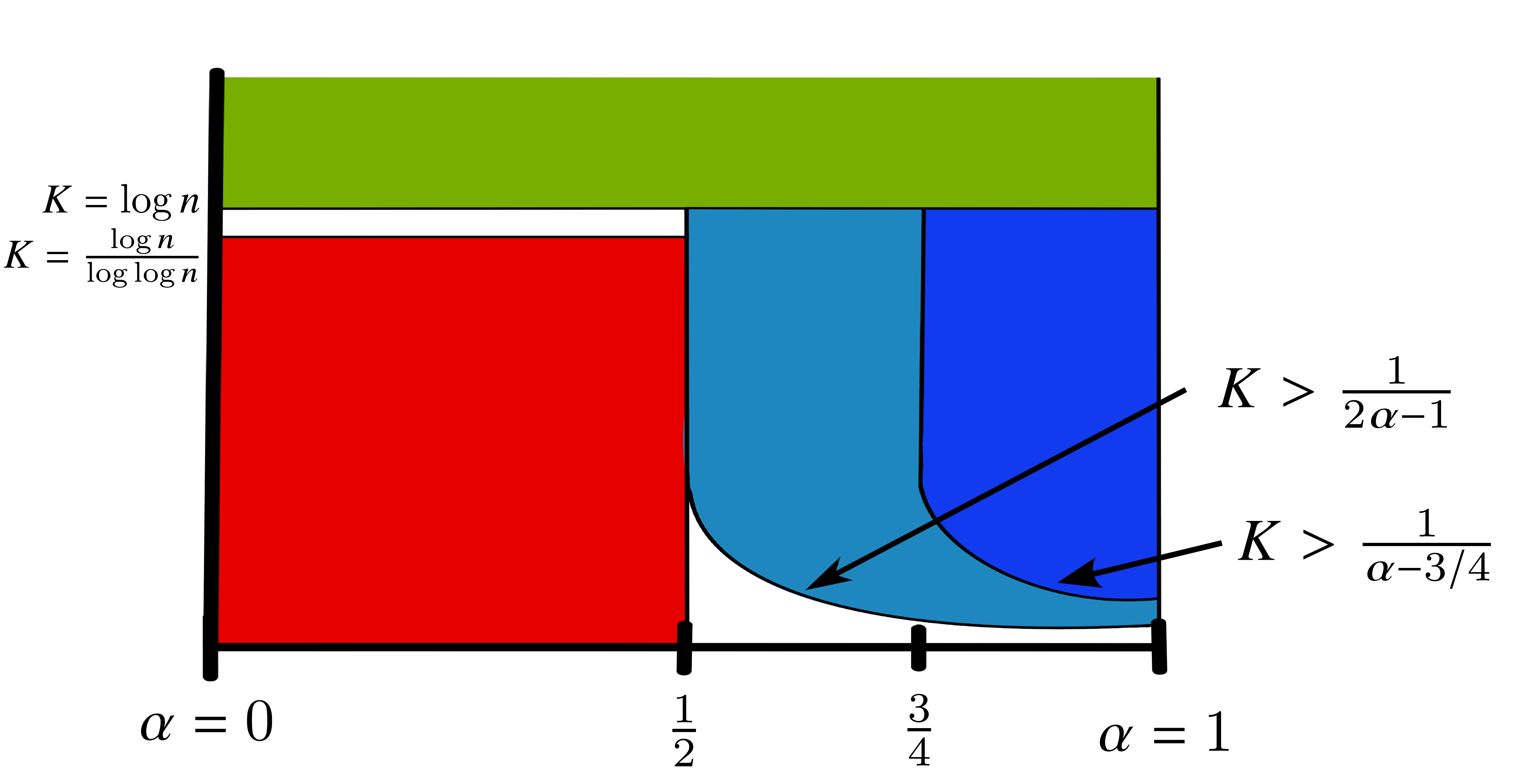}
  \caption{Phase diagram for the possibility and impossibility of estimating high-degree vertices in a graph $G \in \cG$.
  \emph{Dark blue region:} Achievability region for the algorithm of prior work (second derivative thresholding) \cite{mossel2024finding}, for which high-degree vertices can be estimated from more than $1/(\alpha - 3/4)$ cascade traces.
  \emph{Light blue region:} High-degree vertices can be estimated from more than $1/(2\alpha - 1)$ cascade traces (proven in Theorem \ref{thm:estimating_high_degrees}).
  \emph{Red region:} Estimating high-degree vertices is impossible, even when $G$ is known to be a tree \cite{mossel2024finding}.
  \emph{Green region:} Full recovery of $G$ is possible if $G$ is a tree, hence estimation of high-degree vertices is also possible \cite{mossel2024finding, ACFKP13_trace_complexity}.
  \emph{White region:} Regimes with small gaps between the bounds provided by our analysis. 
  }
  \label{fig:si_phase_transitions}
\end{figure}

Together, Theorems \ref{thm:estimating_high_degrees} and \ref{thm:high_degree_estimation_impossibility} show there is a sharp phase transition at $\alpha = 1/2$ in the sample complexity of estimating high-degree vertices. 
On one hand, when $\alpha > 1/2$, high-degree vertices can be estimated using a constant number of cascades (essentially the minimum amount of information possible). 
On the other hand, when $\alpha < 1/2$, at least $\log (n) / \log \log (n)$ cascades are needed for estimation, which nearly matches the sample complexity of learning the \emph{entire} graph, under the additional restriction that it is a tree \cite{ACFKP13_trace_complexity}.
To summarize, there is a sharp dichotomy between minimal and maximal data regimes for the task of learning high-degree vertices from cascade traces. 
A visualization of possibility and impossibility regimes can be found in Figure \ref{fig:si_phase_transitions}.

\section{Overview of algorithms and techniques}
\label{sec:techniques}

In this section, we discuss our methods and results in full detail. Our algorithm for change-point detection is described in Section \ref{subsec:algorithms}, and the proof techniques are discussed in Section \ref{subsec:general_result}.

\subsection{Algorithm for change-point detection}
\label{subsec:algorithms} 
Recall our discussion that the key to detecting change-points is estimating the counterfactual residual. It turns out that for $\ell \ge 2$, the order $\ell$ discrete derivative captures the counterfactual residual of a degree $\ell - 2$ polynomial of past data in a short window. 
To be precise, for an integer $\ell \ge 1$ and $\delta > 0$, the order $\ell$ derivative of $N(t)$ with discretization $\delta$ is given by
\[
\Delta_\delta^{(\ell)} N(t) : = \sum_{j = 0}^{\ell} (-1)^{\ell - j} \binom{\ell}{j} N(t + (j - \ell + 1) \delta),
\]
which is well defined for $t \ge (\ell - 1) \delta$. Our detection algorithm given in  Theorem~\ref{thm:simplified_derivatives} then simply returns times at which this discrete derivative exceeds a given threshold. More details can be found in Algorithm \ref{alg:derivative_thresholding}.
\smallskip 

To motivate this method, let us assume for simplicity that the rate function $\Lambda$ is right-continuous, piecewise constant, and bounded, i.e., there exists a constant $S > 0$ such that $| \Lambda(s) | \le S$ for all $s \ge 0$. 
We assume that all jump discontinuities are separated by at least some $\spacing > 0$.
For a given $t \ge 0$ and $\delta > 0$, let $\Lambda$ be continuous in $(t - \delta, t)$. 
Assuming there is no discontinuity in the interval $(t - \delta, t + \delta)$, the number of events in $(t - \delta, t)$ and $(t, t + \delta)$ have the same mean. Consequently, a natural estimate of the process in the interval $(t , t + \delta)$ is simply $\widehat{N}(t, t + \delta) : = N(t) - N(t - \delta)$. 
The \emph{counterfactual residual} between the observed number of events in $(t, t + \delta)$ and the estimated counterfactual is given by 
\[
N(t + \delta) - N(t) - \widehat{N}(t, t + \delta) = N(t + \delta) - 2 N(t) + N(t - \delta) =: \Delta_\delta^{(2)} N(t),
\]
where one may recognize that $\Delta_\delta^{(2)} N(t)$ is the \emph{second-order discrete derivative} of the point process.

To understand the behavior of the counterfactual residual in more detail, note that if $\Lambda$ is continuous in $(t - \delta, t) \cup (t, t + \delta)$ (possibly jumping at time $t$), then $N(t + \delta) - N(t) \sim \mathrm{Poi}(\delta \Lambda(t))$ and $N(t) - N(t - \delta) \sim \mathrm{Poi}(\delta \Lambda(t^-))$. Hence,
\begin{equation}
\label{eq:second_derivative_characterization}
\frac{1}{\delta} \Delta_\delta^{(2)} N(t) = \Lambda(t) - \Lambda(t^-) \pm O \left( \sqrt{\frac{S}{\delta}} \right),
\end{equation}
where the $O(\sqrt{S / \delta})$ fluctuations, which we call the \emph{concentration error}, appears since the variance of both Poisson random variables is at most $\delta S$. The concentration error can be understood as the ambient randomness of the point process, which is on the order of the standard deviation of $\Delta_\delta^{(2)} N(t) / \delta$.

An important consequence of  \eqref{eq:second_derivative_characterization} is that jumps in $\Lambda$ are detectable through this method only if the size of the jump, denoted by $\Lambda(t) - \Lambda(t^-)$, is larger than the standard deviation.
If $\delta$ is of constant order with respect to $S$ (such a parameter selection is natural if the minimal spacing between jumps, given by $\kappa$, is bounded away from 0 as a function of $S$), this condition reduces to 
\begin{equation}
\label{eq:jump_condition_1}
| \Lambda(t) - \Lambda(t^-) | \gg \sqrt{ S}.
\end{equation}
This condition is tight: if the jump size is asymptotically smaller than $\sqrt{S}$, then it is generally impossible to distinguish the jump from the ambient randomness of the point process (see Proposition \ref{prop:impossibility} and surrounding discussion in the text).

While the second derivative method is tailored towards piecewise constant rate functions, it can still identify jumps in more complex rate functions. As an example, let us assume that $\Lambda$ is a right-continuous, piecewise smooth function. By the definition of the point process, we have that
\begin{align*}
N(t) - N(t - \delta) & \sim \mathrm{Poi} \left( \int_{t - \delta}^t \Lambda(s) ds \right) \\
N(t + \delta) - N(t) & \sim \mathrm{Poi} \left( \int_{t}^{t + \delta} \Lambda(s) ds \right).
\end{align*}
Assume that there exists a constant $S > 0$ such that for all $s \ge 0$, $| \Lambda(s) | \le S$ and $| \Lambda'(s) | \le S$ wherever the derivative exists. Then the variances of $N(t) - N(t - \delta)$ and $N(t + \delta) - N(t)$ are both at most $\delta S$, hence
\begin{equation}
\label{eq:second_derivative_characterization_v2}
\frac{1}{\delta} \Delta_\delta^{(2)} N(t) = \frac{1}{\delta} \left( \int_{t}^{t + \delta} \Lambda(s) ds - \int_{t - \delta}^t \Lambda(s) ds \right) \pm O \left( \sqrt{\frac{S}{\delta}} \right),
\end{equation}
where, as before, the second term on the right hand side is the \emph{concentration error}.
To simplify the right hand side of \eqref{eq:second_derivative_characterization_v2}, suppose there are no discontinuities in the intervals $(t - \delta, t)$ and $(t, t + \delta)$. Then, since the absolute value of $\Lambda'$ is bounded by $S$ wherever it exists, it holds that $|\Lambda(s) - \Lambda(t^-) | \le \delta S$ for all $s \in (t - \delta, t)$ and $|\Lambda(s) - \Lambda(t) | \le \delta S$ for all $s \in (t, t + \delta)$. As a result, we have that
\begin{align}
\frac{1}{\delta} \Delta_\delta^{(2)}N(t) & =  \Lambda(t) - \Lambda(t^-) + \frac{1}{\delta} \left(\int_t^{t + \delta} ( \Lambda(s) - \Lambda(t)) ds - \int_{t - \delta}^t ( \Lambda(s) - \Lambda(t) ) ds\right) \pm O \left( \sqrt{\frac{S}{\delta}} \right) \nonumber \\
\label{eq:second_derivative_characterization_v3}
& =  \Lambda(t) - \Lambda(t^-) \pm O \left( \delta S + \sqrt{\frac{S}{\delta}} \right).
\end{align}
From  \eqref{eq:second_derivative_characterization_v3}, we see that jumps are detectable from the second-order discrete derivative if they are asymptotically larger than both $\delta S$ and $\sqrt{S / \delta}$. Optimizing over $\delta$ leads to the condition
\begin{equation}
\label{eq:jump_condition_2}
| \Lambda(t) - \Lambda(t^-) | \gg S^{2/3}.
\end{equation}
Note that this condition is much more strict than \eqref{eq:jump_condition_1}. This is primarily due to the additional error term $\delta S$, which captures the error incurred by approximating a smooth function by a constant function over an interval of size $\delta$. This approximation error increases with $\delta$, while the concentration error {decreases} with $\delta$. Balancing both errors leads to strictly worse performance than the case of piecewise constant rate functions. 

Fortunately, the gap between the conditions \eqref{eq:jump_condition_1} and \eqref{eq:jump_condition_2} can be reduced if we consider more complex counterfactual estimates of $N(t + \delta) - N(t)$. Previously, we estimated this quantity through a locally constant approximation, i.e., $\widehat{N}(t, t + \delta) = N(t) - N(t - \delta)$. A natural generalization is to use a locally \emph{linear} approximation; that is, based on the number of events in $(t - 2 \delta, t - \delta)$ and $(t - \delta, t)$, we can construct a linear estimate of $N(t + \delta) - N(t)$. 
This leads to the following predictor:
\begin{align}
\label{eq:locally_linear_estimator}
\widetilde{N}(t, t + \delta) & : = N(t) - N(t - \delta) + \left( (N(t) - N(t - \delta)) - (N(t - \delta) - N(t - 2 \delta)) \right) \\
& = 2 N(t) - 3 N(t - \delta) + N(t - 2 \delta). \nonumber
\end{align}
The expression in \eqref{eq:locally_linear_estimator} effectively adds a linear correction to the earlier estimate $\widehat{N}(t, t + \delta)$ by considering the changes in the number of events in the intervals $(t- 2 \delta, t - \delta)$ and $(t - \delta, t)$.
The difference between the true and predicted number of events in $(t, t + \delta)$ is given by
\begin{align*}
N(t + \delta) - N(t) - \widetilde{N}(t + \delta) = N(t + \delta) - 3 N(t) + 3 N(t - \delta) - N(t - 2 \delta) = : \Delta_\delta^{(3)} N(t),
\end{align*}
where one may recognize that $\Delta_\delta^{(3)} N(t)$ is the third-order discrete derivative of the point process. 
Let us now analyze the behavior of $\Delta_\delta^{(3)} N(t)$. Assume that $| \Lambda(s) |, | \Lambda'(s) |, | \Lambda''(s) | \le S$ for all $s \ge 0$, where the bounds on $\Lambda'$ and $\Lambda''$ hold whenever they are well-defined.
Through similar arguments used in the derivation of \eqref{eq:second_derivative_characterization_v2}, we have that 
\begin{equation}
\label{eq:third_derivative_characterization}
\frac{1}{\delta} \Delta_\delta^{(3)} N(t) = \frac{1}{\delta}\left( \int_t^{t + \delta} \Lambda(s) ds - 2 \int_{t - \delta}^t \Lambda(s) ds +  \int_{t - 2 \delta }^{t - \delta} \Lambda(s) ds \right) \pm O \left( \sqrt{ \frac{S}{\delta}} \right).
\end{equation}
If $\Lambda$ is smooth on $(t - 2 \delta, t) \cup (t, t + \delta)$, then by Taylor's theorem we have
\begin{align*}
\Lambda(s) & = \Lambda(t^-) + (s - t) \Lambda'(t^-) \pm O(\delta^2 S), \qquad s \in (t - 2 \delta, t) \\
\Lambda(s) & = \Lambda(t) + (s - t) \Lambda'(t) \pm O(\delta^2 S), \qquad s \in (t, t + \delta).
\end{align*}
Substituting these expressions into \eqref{eq:third_derivative_characterization}, we arrive at
\begin{equation}
\label{eq:third_derivative_characterization_v2}
\frac{1}{\delta} \Delta_{\delta}^{(3)} N(t) = \Lambda(t) - \Lambda(t^-) + \frac{\delta}{2} ( \Lambda'(t) - \Lambda'(t^-) ) \pm O \left( \delta^2 S + \sqrt{\frac{S}{\delta}} \right).
\end{equation}
Compared to \eqref{eq:second_derivative_characterization_v3}, the fluctuations above are smaller, with $\delta S$ replaced by $\delta^2 S$.
To clearly see the implications of \eqref{eq:third_derivative_characterization_v2} for change-point detection, let us assume that if an abrupt jump discontinuity occurs at $t$, the jump in the first derivative of $\Lambda'$ is not significantly larger than the jump in $\Lambda$ at time $t$. That is, 
\begin{equation}
\label{eq:jump_of_derivative}
\frac{ | \Lambda'(t)- \Lambda'(t^-)|}{ |\Lambda(t) - \Lambda(t^-) |} \le \rho, 
\end{equation}
where $\rho$ is of constant order with respect to $S$ and $A$. (For the applications we study, a version of this inequality does indeed hold.) It follows that 
\begin{equation}
\label{eq:third_derivative_characterization_v3}
\frac{1}{\delta} \Delta_{\delta}^{(3)} N(t) = \left( 1 \pm \frac{\delta \rho}{2} \right) ( \Lambda(t) - \Lambda(t^-) ) \pm O \left( \delta^2 S + \sqrt{\frac{S}{\delta}} \right).
\end{equation}
Optimizing over $\delta$ shows that a jump at time $t$ can be detected by the third-order discrete derivative if 
\begin{equation}
\label{eq:jump_condition_3}
| \Lambda(t) - \Lambda(t^-) | \gg S^{3/5},
\end{equation}
which is a significant improvement over \eqref{eq:jump_condition_2}.

Generalizing this idea further to higher order discrete derivatives, we have the following result. 
\begin{thm}[Simplified version of Theorem \ref{thm:derivative_thresholding}]
\label{thm:simplified_derivatives}
Let $\Lambda$ be a piecewise smooth function with discontinuity points given by a finite set $\cT \subset (0,T)$.
Let $T, \rho, \spacing > 0$ be fixed constants, and let $\theta \in (1/2, 1)$.
Let $\ell \ge 1$ be an integer satisfying $\frac{\ell + 1}{2 \ell + 1} < \theta < 1$.
Additionally, suppose that the parameters $A, S$ and the set of change-points $\cT$ satisfy
\begin{enumerate}
    \item (Smoothness.) For all $t \in [0,T] \setminus \cT$, it holds that $| \Lambda(t) |, | \Lambda^{(1)}(t)|, \ldots, | \Lambda^{(\ell)}(t) | \le S$.
    \item (Jump size.) For all $t \in \cT$, $| \Lambda(t) - \Lambda(t^-) | \ge A$.
    \item \label{item:jumps_of_derivatives}
    (Jumps of derivatives.) For all $t \in \cT$, it holds that
    \[
    \max_{1 \le k \le \ell} \frac{| \Lambda^{(k)}(t) - \Lambda^{(k)}(t^-) |}{ |\Lambda(t) - \Lambda(t^-) |} \le \rho.
    \]
    \item (Separation of change-points.) $\sep(\cT) \ge \spacing$.
\end{enumerate}
Suppose that as $S \to \infty$, it holds that $A \ge S^{\theta}$ and $\delta = S^{-1/(2 \ell + 1)}$ while $T, \spacing, \rho$ are fixed. 
Then with probability $1 - o(1)$, 
\begin{equation}
\label{eq:simplified_detection_statement}
\cT \subseteq \left \{ t \in [\ell \delta,T] : \left| \frac{1}{\delta} \derivative N(t) \right| \ge \frac{A}{2} \right \} \subseteq \left \{ t \in [0,T] : \exists t' \in \cT \text{ with } |t - t' | \le (\ell + 1) \delta  \right \}.
\end{equation}
\end{thm}

Theorem \ref{thm:simplified_derivatives} is an \emph{algorithmic} version of Theorem \ref{thm:main_simplified} that highlights how thresholding higher-order derivatives can identify change-points. The first inclusion in \eqref{eq:simplified_detection_statement} states that this thresholding procedure produces no false negatives; all abrupt change-points lead to a large absolute value of the order $\ell + 1$ derivative.
The second inclusion in \eqref{eq:simplified_detection_statement} states that the thresholding procedure produces no false positives, in the sense that any time indices that pass the threshold must be sufficiently close in proximity to an actual change-point. 
We make a few additional remarks about the theorem. 
\begin{itemize}
    \item The number of derivatives needed, captured by $\ell + 1$, depends on how close $\theta$ is to $1/2$. If $\theta = 1/2 + \epsilon$ then $\ell = \Omega ( 1/ \epsilon)$ is required for the theorem to apply.
    \item Condition \ref{item:jumps_of_derivatives}, which concerns the jumps of derivatives of $\Lambda$ at change-points, is a generalization of \eqref{eq:jump_of_derivative} when higher-order derivatives are used. At a high level, it ensures that abrupt changes are not too transient to be detectable. 
    \item Our general result, Theorem \ref{thm:derivative_thresholding}, holds for a much broader class of deterministic and possibly random rate functions, though the conditions are more complicated to state. We discuss the details in Section \ref{subsec:general_result}.
\end{itemize}

We also remark that the property \eqref{eq:simplified_detection_statement} readily yields an estimator for change-points that satisfies the algorithmic guarantees outlined in Theorems \ref{thm:main_simplified}, \ref{thm:smooth-plus-jump} and \ref{thm:detecting_high_degree}. This is formalized in the following result.

\begin{prop}
\label{prop:algorithm}
Let $\delta > 0$, and suppose that $\sep(\cT) \ge 4 (\ell + 1) \delta$. Additionally, let 
\[
\widehat{\cT} \subset \left \{ t \in [0,T] : \left| \frac{1}{\delta}  \derivative N(t) \right| \ge \frac{A}{2} \right \}
\]
be a subset of maximal size satisfying $\sep(\widehat{\cT}) > 2 (\ell + 1) \delta$, i.e., $\widehat{\cT}$ is a maximal $2(\ell + 1)\delta$-packing. If \eqref{eq:simplified_detection_statement} holds, then $| \cT | = | \widehat{\cT}|$ and $\cpdist(\cT, \widehat{\cT}) \le 2 (\ell + 1) \delta$.
\end{prop}

\begin{proof}
It suffices to show that for any $t \in \cT$, there exists a unique $t' \in \widehat{\cT}$ with $|t - t'| \le 2 (\ell + 1)\delta$. We prove by way of contradiction that this must be the case.

If such a $t'$ does \emph{not} exist, then we could add $t$ to $\widehat{\cT}$ since $\left| \frac{1}{\delta} \derivative N(t) \right| \ge \frac{A}{2}$ in light of the first inclusion in \eqref{eq:simplified_detection_statement}, and since $|s - t| > 2(\ell + 1) \delta$ for all $s \in \widehat{\cT}$.
However, adding $t$ to $\widehat{\cT}$ would contradict the maximality of $\widehat{\cT}$.

Suppose now that there exists $t, t'' \in \widehat{\cT}$ such that $|t - t'|, |t - t''| \le 2 (\ell + 1) \delta$. Since $\sep(\cT) \ge 4(\ell + 1) \delta$, the second inclusion of \eqref{eq:simplified_detection_statement} implies that $|t - t'|, |t - t''| \le (\ell + 1)\delta$, which in turn implies that $|t' - t''| \le 2(\ell + 1) \delta$. However, this contradicts the property that $\sep(\widehat{\cT}) > 2 (\ell + 1) \delta$.
\end{proof}

For clarity, we have outlined the full change-point detection procedure in Algorithm \ref{alg:derivative_thresholding}.
As stated, Algorithm \ref{alg:derivative_thresholding} requires one to compute the discrete derivative for each $t \in [0,T]$, which is an infinite set. In practice, it is sufficient to first discretize the set of time indices, and to compute the discrete derivative over just the discretized set.

\smallskip 
\begin{breakablealgorithm}
\caption{Finding abrupt changes via derivative thresholding}
\label{alg:derivative_thresholding}
\begin{algorithmic}[1]
\Require{Point process $N(t)$, time horizon $T\ge 0$, integer $\ell \ge 0$, discretization $\delta > 0$, threshold $A > 0$}
\Ensure{A finite set $\widehat{\cT} \subset [0,T]$}
\State For each $t \in [0,T]$, compute $\derivative N(t)$.
\State Output a $2 (\ell + 1) \delta$-packing of $ \left \{ t \in [0,T] : \left| \frac{1}{\delta} \derivative N(t) \right| \ge A / 2 \right \}$.
\end{algorithmic}
\end{breakablealgorithm}
\smallskip 

\begin{rem}[Detecting a single change-point]
If one has a priori knowledge of only a single change-point, Step 2 of Algorithm \ref{alg:derivative_thresholding} is equivalent to outputting $\argmax_{t \in [0,T]} \left| \frac{1}{\delta} \derivative N(t) \right|$.
\end{rem}

\subsection{General results}
\label{subsec:general_result}
In this section, we dive into the details of our general results. At a high level, our strategy is as follows:

\begin{enumerate}
    \item (Polynomial approximation) We first show that, if no abrupt changes occur in the interval $(t, s]$, then we can approximate $N(s)$ by $\overline{N}_\ell(s, t)$, where $\overline{N}_\ell(\cdot, t)$ is a degree $\ell$ polynomial with $\cF_t$-measurable coefficients. The approximation error $N(s) - \overline{N}_\ell(s, t)$ depends on $|s - t|$ as well as the ``smoothness'' of the point process in regions away from abrupt changes, much like a Taylor approximation.

    \item (Analysis of higher-order derivatives) We next show that 
    \begin{equation}
    \label{eq:prediction_error_informal}
    \derivative N(t) = N\left( t + \delta \right) - \overline{N}_\ell \left( t + \delta, t^- \right) \pm \delta \err,
    \end{equation}
    where $\err$ captures the error of the polynomial approximation. We can interpret $\overline{N}_\ell(t + \delta , t^-)$ as a degree $\ell$ estimate of \emph{counterfactual trajectory} of the process (i.e., how the process would evolve if there were no change-point at time $t$).\footnote{Formally, $\overline{N}_\ell(\cdot, t^-) = \lim_{s \to t^-} \overline{N}_\ell(\cdot, s)$, which is well-defined since we assume the polynomial coefficients are \cadlag (see Assumption \ref{as:lambda_k}).} 
    With this interpretation, the order $\ell + 1$ derivative captures the counterfactual residual based on the degree $\ell$ approximation.
    We show formally that if $t$ is not close to an abrupt change, the counterfactual trajectory correctly captures the behavior of the observed process at time $t + \delta$ (up to random fluctuations that are on the order of the process variance).
    If an abrupt change occurs at time $t$, there is a significant difference between the observed behavior and the counterfactual trajectory at time $t + \delta$.

    \item (Conditions for detection) The arguments above indicate that the value of the order $(\ell + 1)$ derivative will be large in the presence of an abrupt change, provided the counterfactual residual overcomes the ``baseline'' error of the polynomial approximation. By choosing a value of $\ell$ appropriately, we show this occurs precisely when the magnitude of the abrupt change is larger than the square root of the ``smoothness" of the point process (appropriately defined).
\end{enumerate}

\subsubsection{Polynomial approximation of the point process}

Let $s, t \ge 0$. Define the \emph{order $\ell$ local polynomial approximation} of $N(s)$ around $t$ to be 
\begin{equation}
\label{eq:poly_approx}
\overline{N}_\ell(s,t) : = N(t) + \sum_{k = 1}^\ell \frac{ \lambda^{(k)}(t)}{k!} (s - t)^k,
\end{equation}
where the polynomial coefficients $\lambda^{(1)}(t), \ldots, \lambda^{(\ell)}(t)$ are $\cF_t$-measurable random variables. 
At a high level, these coefficients capture the local characteristics of the point process in ``nice'' regions (i.e., away from the abrupt changes in $\cT$). 
Our goal is to choose the $\lambda^{(k)}$'s so that the following event holds with high probability. 

\begin{defn}[Polynomial approximation]
\label{def:Eapprox}
Let $T, \err, \delta > 0$ and let $\ell \ge 1$ be an integer. We say that the event $\Eapprox(T, \err, \delta, \ell)$ holds if and only if the following conditions are satisfied for all $t \in [0,T]$:
\begin{enumerate}
    \item \label{item:forward_prediction}
    (Forward predictions) If $\cT \cap (t, t + \delta) = \emptyset$, then it holds for all $s \in (t,t + \delta)$ that 
    \begin{equation}
    \label{eq:event_forward_prediction}
    | N(s) - \overline{N}_\ell(s, t) | \le \delta \err.
    \end{equation}
    \item \label{item:backward_prediction}
    (Backward predictions) If $\cT \cap (t - \delta, t) = \emptyset$, then it holds for all $s \in (t - \delta, t)$ that 
    \begin{equation}
    \label{eq:event_backward_prediction}
    | N(s)  - \overline{N}_\ell(s, t^-) | \le \delta \err.
    \end{equation}
\end{enumerate}
As a shorthand, we will often write $\Eapprox$ instead of $\Eapprox(T, \err, \delta, \ell)$.
\end{defn}

Observing that $\overline{N}_\ell(s, t)$ resembles a Taylor approximation, a natural choice is to let $\lambda^{(k)}(t)$ be the $k$th order conditional derivative of $N(t)$. As an example, when $k = 1$, the conditional first derivative of $N(t)$ is simply the rate function of the point process, given by 
\[
\Lambda (t) = \lim_{\epsilon \to 0} \frac{ \E [ N(t + \epsilon) - N(t) \vert \cF_t]}{\epsilon}.
\]
However, notice that the computation of $\Lambda(t)$ may involve contributions from sample paths where an abrupt change occurs, which could cause the approximation error $N(s) - \overline{N}_\ell(s,t)$ to be quite large. 
To avoid such sample paths, we introduce the \emph{modified} conditional first derivative, given by 
\begin{equation}\label{eq:def-lambda-1}
\lambda^{(1)}(t) : = \lim_{\epsilon \to 0} \frac{1}{\epsilon} \E \left[ \left. \int_t^{t + \epsilon} \mathbf{1}(x \notin \cT) d N(x) \right \vert \cF_t \right ],
\end{equation}
which can be viewed as a version of $\Lambda(t)$ with the contributions of sample paths leading to $\cT$ removed.
More generally, for $k \ge 1$, we can inductively define the higher-order modified conditional derivative
\[
\lambda^{(k+1)}(t) : = \lim_{\epsilon \to 0} \frac{1}{\epsilon} \E \left[\left. \int_t^{t + \epsilon} \mathbf{1}(x \notin \cT) d \lambda^{(k)}(x)  \right \vert \cF_t \right].
\]
For notational consistency, we also denote $\lambda^{(0)}(t) : = N(t)$. 
We assume the following regularity conditions related to $\Lambda(t)$ and $\lambda^{(k)}(t), k \ge 0$.

\begin{assumption}[Smoothness and jumps]
\label{as:regularity}
Let $R > 0$ and let $\mathbf{S} := \{ S_k \}_{k \ge 1}, \mathbf{J} : = \{ J_k \}_{k \ge 0}$ be sequences of positive real numbers. 
Assume that it holds almost surely that $\sup_{t \in [0,T]} \Lambda(t) \le R$.
Additionally, assume that the following statements hold almost surely:
\begin{enumerate}
    \item $\sup_{t \in [0,T]} | \lambda^{(k)}(t) | \le S_k$ for all integers $k \ge 1$;
    \item $\sup_{t \in [0,T] \setminus \cT} | \lambda^{(k)}(t) - \lambda^{(k)}(t^-) | \le J_k$ for all integers $k \ge 0$.
\end{enumerate}
We will often use the shorthand $S_{\le \ell} := \max_{1 \le k \le \ell} S_k$ and $J_{\le \ell} : = \max_{0 \le k \le \ell} J_k$.
\end{assumption}
At a very high level, this notion of regularity allows us to bound how quickly the rate function of the point process changes in local areas around $t$.
The parameter $R$ controls the \emph{maximum rate} of events throughout the process, $S_k$ controls the \emph{smoothness} of the $k$th order conditional expected derivative of the point process, and $J_k$ controls the size of the \emph{jumps} in the $k$th order conditional expected derivative, apart from time indices in $\cT$.
We will additionally need to assume some properties about the $\lambda^{(k)}(t)$'s for $t \in \cT$ (see Definition \ref{def:Eabrupt}), though these are not relevant for our results on polynomial approximation.

In the applications we will consider, the $\lambda^{(k)}$'s are \cadlag with discontinuities that are either triggered by events in the point process, or are caused by external factors. 
To be more precise, we assume the following about the form of the $\lambda^{(k)}$'s. 

\begin{assumption}
\label{as:lambda_k}
Let $\mathbf{L} : = \{ L_k \}_{k \ge 0}$ be a sequence of positive real numbers and let $\cD \subset \R_+$ be a finite set.
Assume that for each $k \ge 0$, there exist $\cF$-adapted processes $\alpha^{(k)}, \beta^{(k)}, \gamma^{(k)}$ satisfying
\begin{equation}
\label{eq:lambda_k_decomposition}
\lambda^{(k)}(t) = \int_0^t \alpha^{(k)}(x) dx + \int_0^t \beta^{(k)}(x) dN(x) + \sum_{x \in \cD \cap [0,t]} \gamma^{(k)}(x),
\end{equation}
where the process $\alpha^{(k)}$ satisfies $\sup_{x \in [0,T]} |\alpha^{(k)}(x) | \le L_k$. We again often use the shorthand $L_{\leq \ell} := \max_{0 \leq k \leq \ell} L_k$. 
\end{assumption}

In essence, Assumption \ref{as:lambda_k} states that $\lambda^{(k)}$ can be decomposed into a sum of continuous and jump parts (which is in general possible for functions of bounded variation). The first term on the right hand side of \eqref{eq:lambda_k_decomposition} captures the continuous part; the assumption that $\sup_{x \in [0,T]} |\alpha^{(k)}(x)| \le L_k$ ensures that this part is $L_k$-Lipschitz. The second term on the right hand side captures the impact of jumps due to events in the point process. The third term on the right hand side captures the impact of jumps due to external factors (described by the deterministic set $\cD$). We remark that in the applications we will consider, the decomposition \eqref{eq:lambda_k_decomposition} is quite natural and easy to verify.

We are now ready to state our main result concerning polynomial approximation guarantees. The proof can be found in Section \ref{sec:polynomial_approximation}.

\begin{thm}
\label{thm:approximation_error}
Let $T, \delta > 0$ and let $\ell \ge 1$ be an integer. Define
\begin{equation}
\label{eq:error_bound}
\err : =  c_1 \left[ \delta^{\ell}S_{\ell + 1} + \sqrt{\frac{J_{\le \ell}^2 S_{1}}{\delta} \log \left( \frac{T^2 R^2 ( L_{\le \ell} + S_{\le \ell})^3 }{   \sep( \cD) \land 1} \right)} \right],
\end{equation}
where $c_1 : = c_1(\ell)$ is a positive constant.
Furthermore, assume that $\delta$ satisfies
\begin{equation}
\label{eq:delta_lower_bound_generic}
\delta \ge \frac{4 \ell}{S_{\le \ell + 1}} \log \left( \frac{T^2 R^2 ( L_{\le \ell} + S_{\le \ell})^3 }{   \sep( \cD) \land 1} \right).
\end{equation}
Then $\p ( \Eapprox) \ge 1 - 2 / S_{\le \ell}$.
\end{thm}

We make a few key remarks about the theorem. 

\begin{itemize}

\item (Interpreting the error bound \eqref{eq:error_bound}.) The first term on the right hand side corresponds to the approximation error incurred by a degree $\ell$ polynomial approximation, and the while the second term captures the error due to subgaussian-type concentration of the random process. Importantly, the subgaussian-type error bound only holds if $\delta$ is large enough; this is the reason for requiring the lower bound \eqref{eq:delta_lower_bound_generic}.

\item (Tradeoffs in $\delta$.) For a fixed $\ell \ge 1$, there exists a tradeoff in choosing $\delta$. As $\delta$ decreases, the polynomial approximation error naturally decreases as well, which reflects the suitability of the polynomial approximation in local regions. On the other hand, the concentration error increases as $\delta$ decreases, as random processes generally experience better (relative) concentration in larger areas. This tradeoff highlights a fundamental tension between approximation and concentration.

\item (Mitigating the tradeoff.) Increasing the value of $\ell$ -- which corresponds to using higher-order derivatives for detection -- mitigates the approximation-concentration tradeoff. Indeed, increasing $\ell$ significantly reduces the approximation error while only marginally affecting the concentration error.
    
\end{itemize}

The full proof details of the theorem can be found in Section \ref{sec:polynomial_approximation}.

\subsubsection{Analysis of higher-order derivatives}

Our next result formalizes \eqref{eq:prediction_error_informal}, showing that the order $\ell + 1$ discrete derivative captures the counterfactual residual associated with a degree $\ell$ polynomial approximation. The proof can be found in Section \ref{sec:higher_derivatives}.

\begin{lem}
\label{lemma:derivative_characterization} 
On the event $\Eapprox(T, \err, 3 \ell \delta, \ell)$ (see Definition \ref{def:Eapprox}), there exists a constant $c_2 = c_2(\ell) > 0$ such that following holds for all $t \in [0,T]$:
\begin{enumerate}

\item \label{item:predictive_error}
If $\cT \cap (t - 3\ell \delta, t) = \emptyset$, then 
\[
\left| \derivative N(t) - \left( N(t + \delta) - \overline{N}_\ell(t + \delta, t^-) \right) \right| \le c_2 \max \{ \delta \err, 1 \}.
\]

\item \label{item:predictive_error_polynomial}
If $\cT \cap (t -  3\ell \delta , t) = \emptyset$ and $\cT \cap (t, t + 3 \ell \delta) = \emptyset$ then 
\[
\left| \derivative N(t) - \left( \overline{N}_\ell(t + \delta, t) - \overline{N}_\ell(t + \delta, t^-) \right) \right| \le c_2 \max \{\delta \err , 1 \}.
\]

\item \label{item:small_derivative}
If $\cT \cap (t -  \ell \delta, t +  2\delta) = \emptyset$ then $| \derivative N(t) | \le c_2 \delta \err$.

\end{enumerate}
\end{lem}

In the lemma above, Item \ref{item:predictive_error} shows that the order $\ell + 1$ discrete derivative indeed captures the residual between the point process at time $t + \delta$ and its projected value based on information before time $t$. Items \ref{item:predictive_error_polynomial} and \ref{item:small_derivative} are a few consequences of this fact. 
Item \ref{item:small_derivative} establishes that, when $t$ is not too close to an abrupt change, the counterfactual residual -- and therefore the discrete derivative -- is small. It remains to show that $|\derivative N(t)|$ is large for $t \in \cT$, and Item \ref{item:predictive_error_polynomial} provides a way to do so via the polynomial coefficients of $\overline{N}_\ell( \cdot , t)$ and $\overline{N}_\ell(\cdot, t^-)$. To better understand the difference between these polynomials, we need to assume some regularity conditions related to $\cT$, which are given by the following event.

\begin{defn}[Properties of $\cT$]
\label{def:Eabrupt}
Let $A, \rho, \spacing > 0$, and let $\ell \ge 2$ be an integer. We say that the event $\Eabrupt(A, \rho, \spacing, \ell)$ holds if and only if the following conditions are satisfied:
\begin{enumerate}
    \item For all $t \in \cT$, $| \lambda^{(1)}(t) - \lambda^{(1)}(t^-) | \ge A$.
    \item For all $t \in \cT$, it holds that
    \[
    \max_{2 \le k \le \ell} \frac{| \lambda^{(k)}(t) - \lambda^{(k)}(t^-) |}{| \lambda^{(1)}(t) - \lambda^{(1)}(t^-) |} \le \rho.
    \]
    \item For distinct $s, t \in \cT$, it holds that $|s - t| \ge \spacing$ (equivalently, $\sep(\cT) \ge \spacing$).
\end{enumerate}
As a shorthand, we will often write $\Eabrupt$ instead of $\Eabrupt(A, \rho, \spacing,\ell)$.
\end{defn}

The main result of this subsection is that, on the event $\Eapprox \cap \Eabrupt$, the elements of $\cT$ can indeed be identified by large values of the discrete derivative.
The proof can be found in Section \ref{sec:higher_derivatives}.

\begin{thm}
\label{thm:derivative_thresholding}
Fix an integer $\ell \ge 1$. Suppose that as $S_{\le \ell + 1} \to \infty$, it holds that 
\begin{equation}
\label{eq:thm_parameter_conditions}
\max \left \{ J_{\le \ell},  \log \left( \frac{T^2 R^2 ( L_{\le \ell} + S_{\le \ell})^3 }{\sep( \cD) \land 1} \right) , \rho, \frac{1}{\spacing} \right \} = S_{\le \ell + 1}^{o(1)},
\end{equation}
and that $\delta = S_{\le \ell + 1}^{- 1 / (2 \ell + 1)}$.
Then the event $\Eapprox : = \Eapprox(T, \err, 3 \ell \delta, \ell)$ holds with probability $1 - o(1)$.
Moreover, if $A \ge S_{\le \ell + 1}^\theta$ for some $\theta > (\ell + 1) / (2 \ell + 1)$, we have on the event $\Eapprox \cap \Eabrupt$ that
\begin{equation}
\label{eq:estimating_T}
\cT \subseteq \left \{ t \in [0,T] : \left| \frac{1}{\delta} \derivative N(t) \right| \ge \frac{A}{2} \right \} \subseteq \left \{ t \in [0,T] : \exists t' \in \cT \text{ with } |t - t' | \le (\ell + 1) \delta  \right \}.
\end{equation}
\end{thm}

The main takeaway of the theorem is \eqref{eq:estimating_T}, which shows that there are effectively no false positive or false negatives if change-points are detected by thresholding the order $\ell + 1$ discrete derivative of the point process.
Indeed, the first inclusion in \eqref{eq:estimating_T} ensures that all abrupt changes in the rate function do lead to large values in the order $\ell + 1$ derivative.
On the other hand, the second inclusion shows that for all points in time that are not $(\ell + 1)\delta$-close to a change-point, the order $\ell + 1$ derivative will not exceed the threshold $A /2$.

The main assumption of note is \eqref{eq:thm_parameter_conditions}.
The assumptions on $J_{\le \ell}, \rho$, and $\spacing$ are quite natural. 
If $J_{\le \ell} = S_{\le \ell + 1}^{o(1)}$, then there is a clear separation between ``small" jumps and the large, abrupt jumps characterized by $\cT$; this allows us to establish clear results on change-point detection. 
Concerning $\rho$ and $\spacing$, prior discussions in the text have (sometimes implicitly) assumed that they are constant order with respect to $S_{\le \ell + 1}$; the theorem gives us slightly more leeway in allowing $\rho$ and $\spacing$ to grow or decay, respectively, slower than a polynomial function of the smoothness parameter.
Finally, the condition that
\[
\log \left( \frac{T^2 R^2 ( L_{\le \ell} + S_{\le \ell})^3 }{\sep( \cD) \land 1} \right) = S_{\le \ell + 1}^{o(1)}
\]
is made for mathematical simplicity. 
Various scenarios could cause the condition above not be satisfied; these include $T$ or $R$ being extremely large relative to $S_{\le \ell + 1}$, the smoothness parameters $L_k, S_k$ growing extremely quickly as a function of $k$, or points of discontinuity being extremely close to each other (with a distance decaying faster than an inverse polynomial of $S_{\le \ell + 1}$). All of these situations are quite unusual for the examples of interest, leading \eqref{eq:thm_parameter_conditions} to be readily satisfied.

\section{Empirical analysis}
\label{sec:empirical}

\subsection{A smooth-plus-jump rate function}

To illustrate our results in a simple and clean setting, we assume that the rate function has the form 
\begin{equation}
\label{eq:simple_rate_function}
\Lambda(t) = 10^6(1 + \sin(t)) + A e^{- (t - t_0)} \mathbf{1}(t \ge t_0),
\end{equation}
where $A, t_0 > 0$ are fixed parameters. The point process $N(t)$ with rate function $\Lambda$ is an example where a relatively transient jump could be masked by global trends or fluctuations, represented by the sinusoid. 
We consider values of $A$ between 20,000 and 80,000, which are quite small compared to the sinusoidal component of $\Lambda$.
Nevertheless, our methods can identify such jumps.
Figure \ref{fig:simple_rate_derivatives} shows an example of our procedure in the case $(A, t_0) = (40,000, 9)$ for $t \in [0, 20]$. The jump is nearly invisible to the naked eye from the plots of binned event counts of $N(t)$ (i.e., the first order discrete derivative).
With another derivative, a small spike at $t_0 = 9$ can be observed. This spike is enhanced with three derivatives, and a fourth derivative clearly exposes the change-point through a sharp change of large magnitude.

\begin{figure}[t]
    \centering
    \begin{subfigure}[b]{0.45\textwidth}
        \centering
        \includegraphics[width=\textwidth]{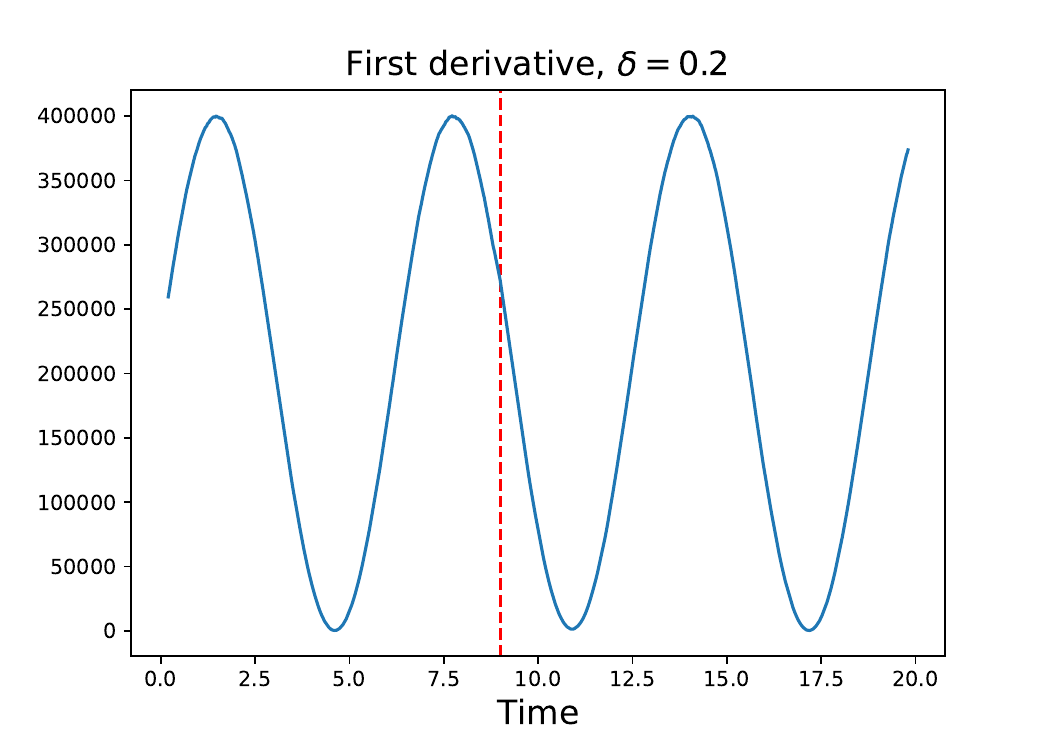}
        \caption{Plot of $\Delta_\delta^{(1)} N(t)$.}
    \end{subfigure}
    \begin{subfigure}[b]{0.45\textwidth}
        \centering
        \includegraphics[width=\textwidth]{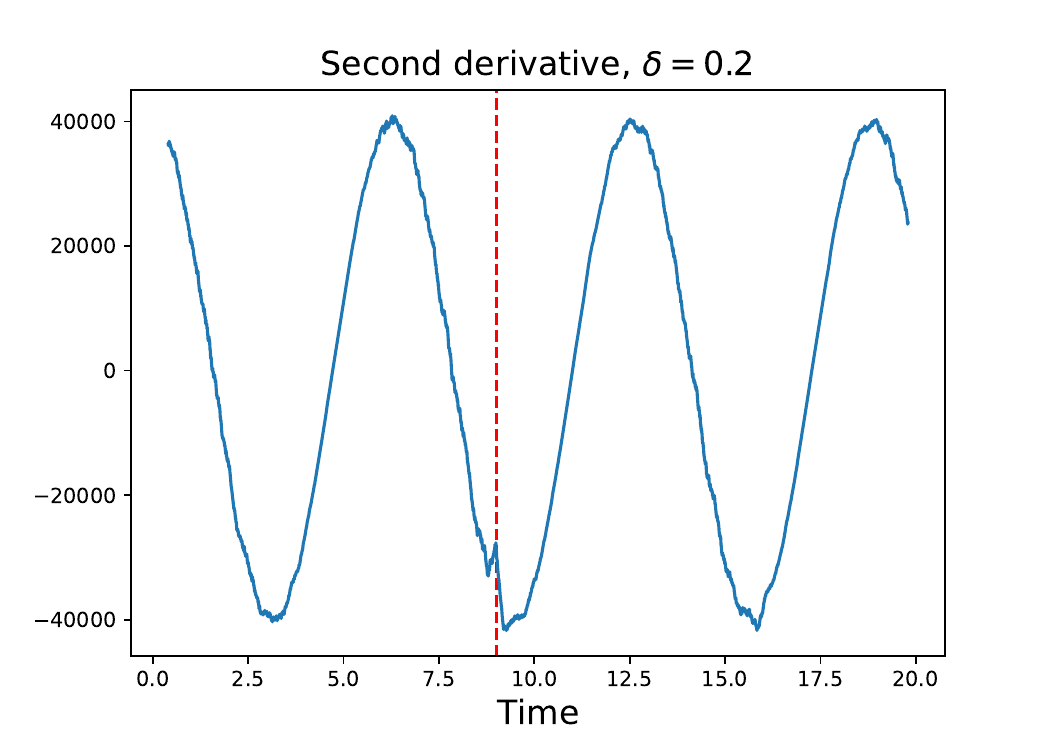}
        \caption{Plot of $\Delta_\delta^{(2)} N(t)$.}
    \end{subfigure}
    \\
    \begin{subfigure}[b]{0.45\textwidth}
        \centering
        \includegraphics[width=\textwidth]{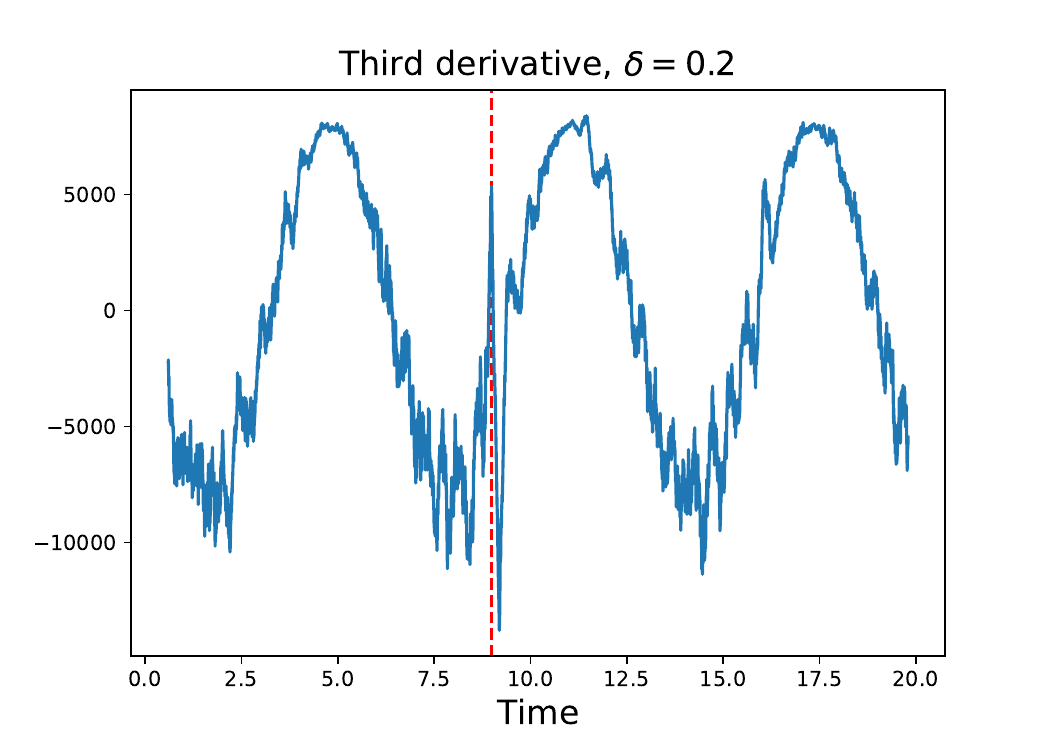}
        \caption{Plot of $\Delta_\delta^{(3)} N(t)$.}
    \end{subfigure}
    \begin{subfigure}[b]{0.45\textwidth}
        \centering
        \includegraphics[width=\textwidth]{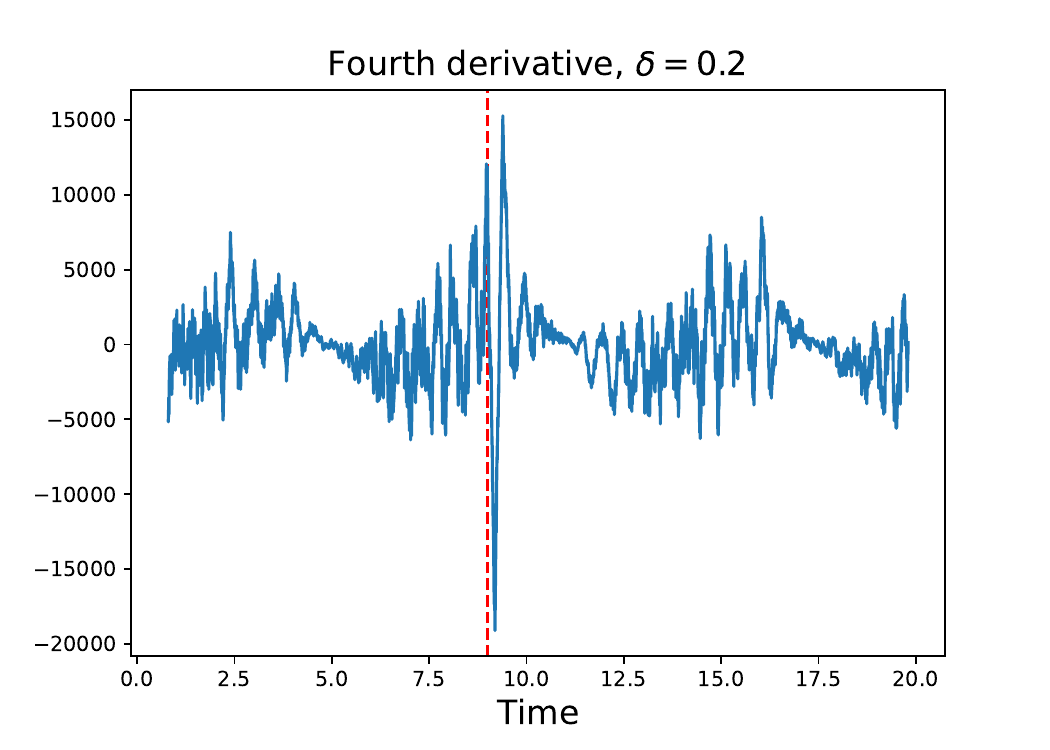}
        \caption{Plot of $\Delta_\delta^{(4)} N(t)$.}
    \end{subfigure}
    \caption{Discrete derivatives of $N(t)$ with $\delta = 0.2$. The red dotted line represents the jump occurring at $t_0 = 9$.}
    \label{fig:simple_rate_derivatives}
\end{figure}

Next, we assess our method's performance by considering $A \in \{20000, 40000, 60000, 80000 \}$ and $t_0$ drawn uniformly from the interval $[5, 15]$.\footnote{Since we simulate the process in the interval $[0, 20]$, choosing the jump to be in $[5, 15]$ avoids potential boundary effects associated with taking discrete derivatives close to $t = 0$ or $t = 20$.}
For a given $(k, \delta)$, we estimate $t_0$ via $\hat{t} = \argmax_{t \ge 0} | \Delta_\delta^{(k)} N(t) |$, i.e., it is the time index with the largest absolute value of the order $k$ derivative with discretization $\delta$. 
We run 100 independent simulations of the point process for each value of $A$ and for $k \in \{1, \ldots, 10\}$ and $\delta \in [0.05, 0.5]$.
The estimation error, which is equal to the absolute difference between $t_0$ and our estimate of $t_0$, is averaged over all trials for a given $(k, \delta)$ and visualized in the heatmaps of Figure \ref{fig:heatmap_synthetic}.

To interpret the results displayed in the heatmap, we begin by making a few baseline observations. Since $t_0$ is sampled uniformly from $[5, 15]$, a na\"{i}ve estimate which outputs a uniform random element of $[5, 15]$ has expected estimation error $10/3 \approx 3.33$.
If we examine the heatmap in Figure \ref{fig:heatmap_A_20} corresponding to $A = 20000$, we see that most of the values are at least as large as 3, indicating that our methods are largely ineffective.
However, for larger values of $A$, our methods are quite effective, as indicated by the large number of small values for the average estimation error for appropriate values of $k$ and $\delta$. Naturally, as $A$ increases, the number of small values for the average estimation error increases correspondingly, since larger jumps are easier to detect from a statistical point of view.

We remark that our theoretical results imply that jumps of size asymptotically larger than the square root of the ``smoothness" parameter -- which in this case is smaller than 1500 for all examples considered -- can be detected. However, in this empirical example, a jump of larger size ($A = 20000$) cannot be detected with reasonable power through our methods. This indicates that while our methods are asymptotically optimal, there is still room for improvement in finite-parameter regimes. We defer a more detailed investigation of this phenomenon to future work.

Finally, we comment on the interplay between $A, k, \delta$, and the estimation error. From Figures \ref{fig:heatmap_A_40}, \ref{fig:heatmap_A_60} and \ref{fig:heatmap_A_80}, there exists a minimum number of derivatives required to have a small estimation error (this number being $k = 4$ for $A = 40000$ and $k = 3$ for $A = 60000, 80000$). There appears to be a phase transition in this value of $k$; with $k - 1$ derivatives, the minimum value of the estimation error is considerably larger than the minimum estimation error with $k$ derivatives. Additionally, for a fixed $k$ larger than the critical value, the estimation error behaves non-monotonically as a function of $\delta$; this is best illustrated in the cases $A = 40000, 60000$ (Figures \ref{fig:heatmap_A_40} and \ref{fig:heatmap_A_60}). 
The estimation error is large for very small $\delta$, and then sharply decreases at some critical value. The estimation error stays low, and then gradually increases as $\delta$ becomes larger. 
We interpret this phenomenon as follows. If $\delta$ is too small, then the impact of the jump may be lost in the ambient randomness of the process; the point at which the estimation error decreases is when $\delta$ is large enough to ensure that the measurement of the jump is larger than the standard deviation of the process.
On the other hand, since the discrete derivative essentially deals with a version of the point process that is quantized into bins of size $\delta$, the estimation error gradually increases with $\delta$ since the quantization becomes more significant, but in a continuous manner.

\begin{figure}[t]
  \centering
  \begin{subfigure}[b]{0.45\textwidth}
    \includegraphics[width=\textwidth]{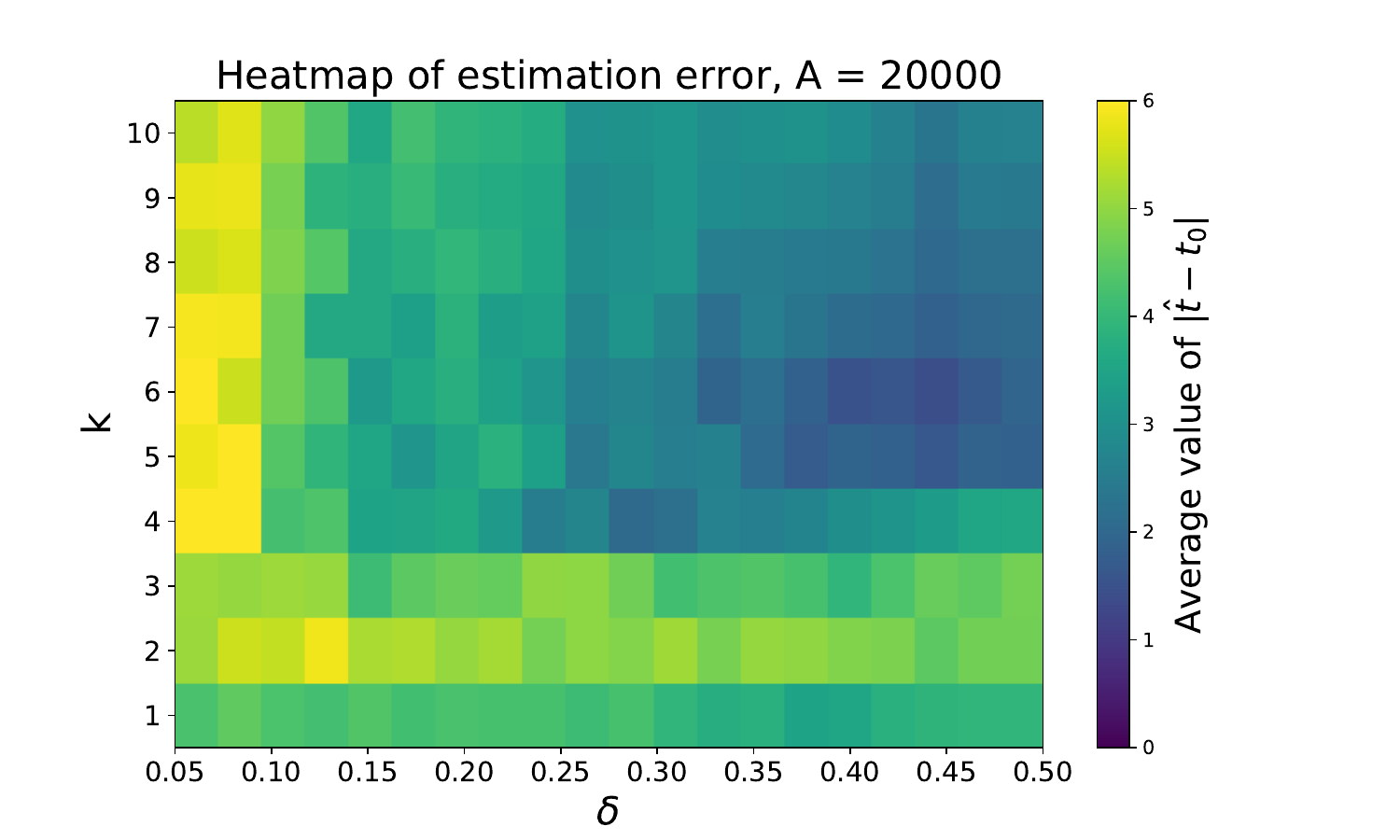}
    \caption{$(k_{\min}, \delta_{\min}) = (6, 0.45)$, estimation error = 1.45.}
    \label{fig:heatmap_A_20}
  \end{subfigure}
  \begin{subfigure}[b]{0.45\textwidth}
    \includegraphics[width=\textwidth]{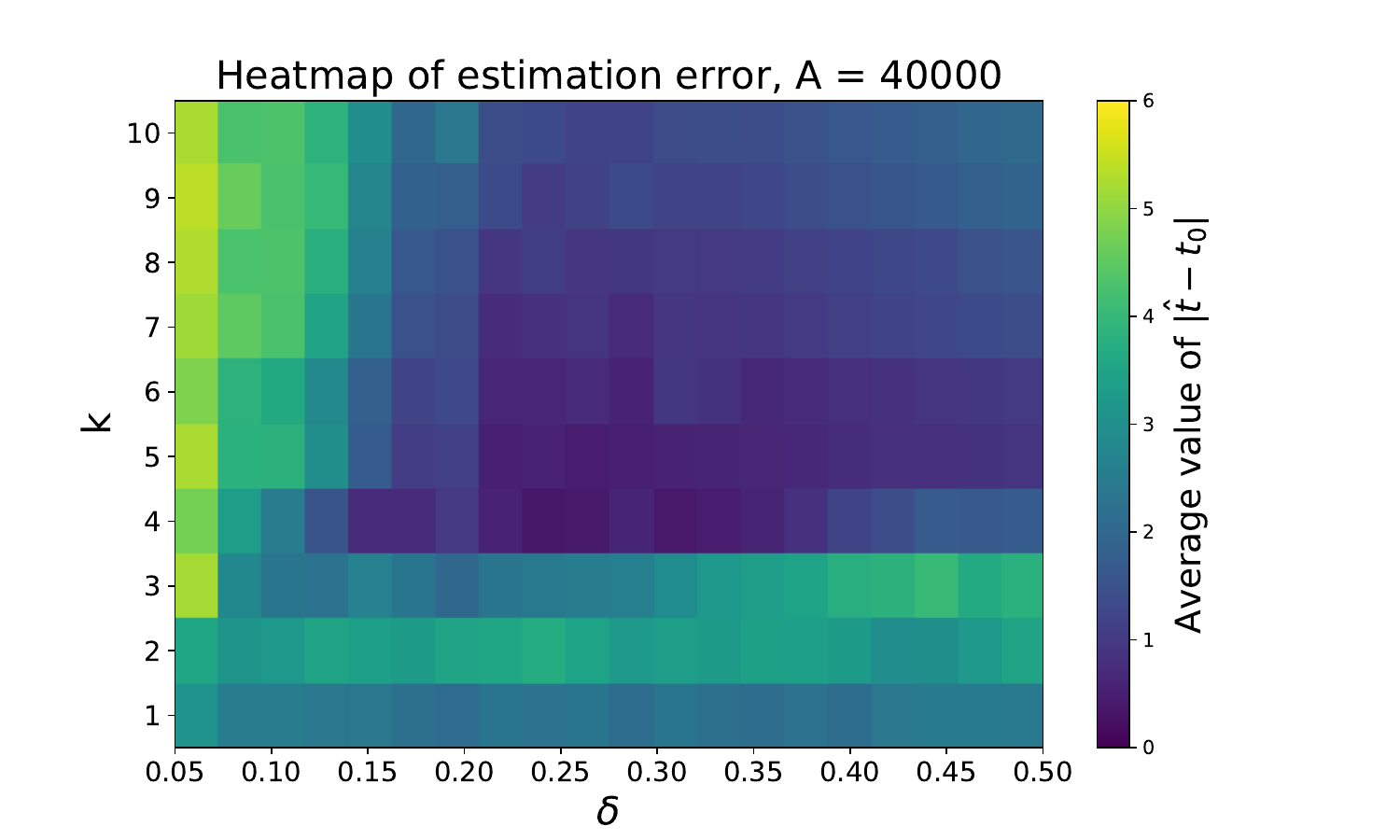}
    \caption{$(k_{\min}, \delta_{\min}) = (4, 0.24)$, estimation error = 0.39.}
    \label{fig:heatmap_A_40}
  \end{subfigure}
  \\[1ex]
  \begin{subfigure}[b]{0.45\textwidth}
    \includegraphics[width=\textwidth]{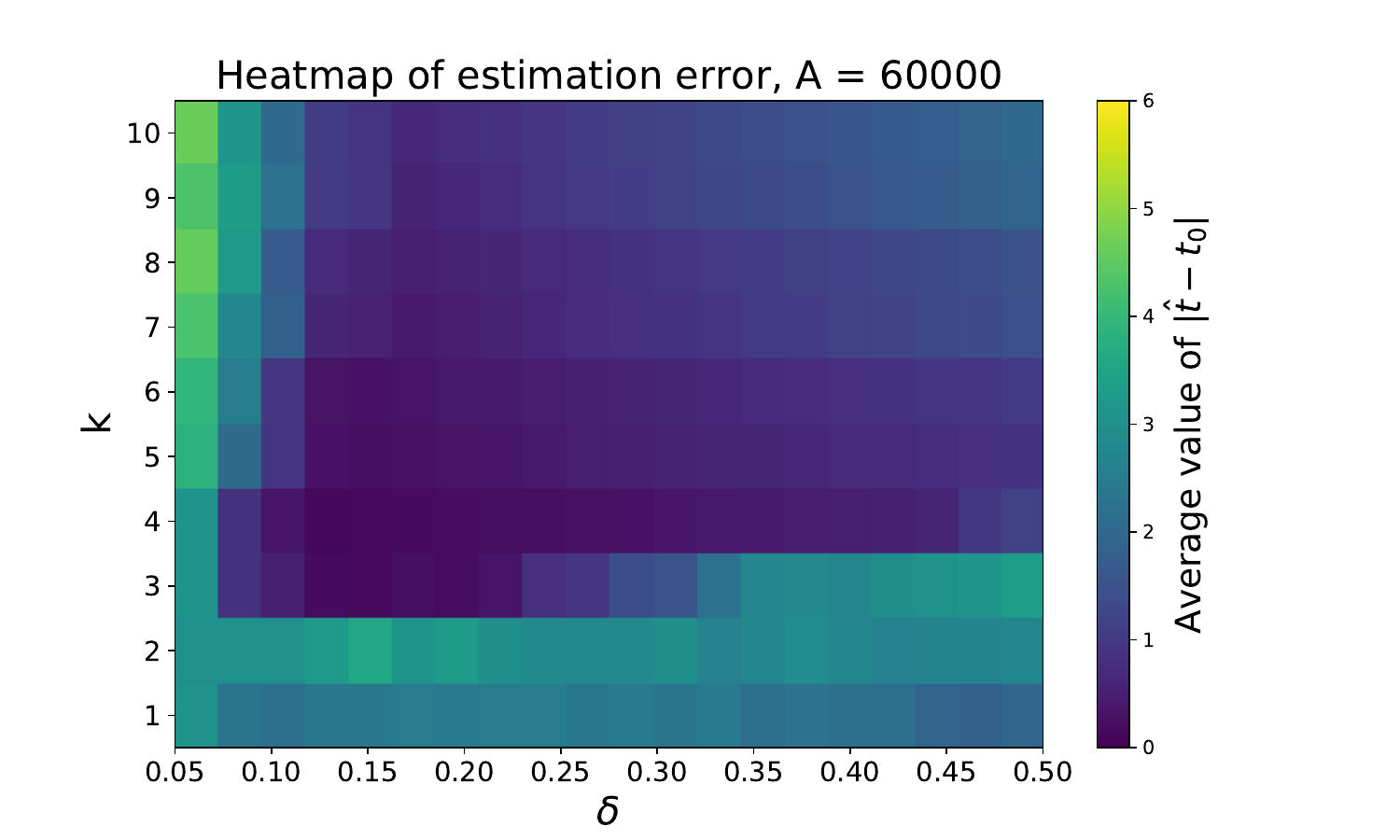}
    \caption{$(k_{\min}, \delta_{\min}) = (4, 0.12)$, estimation error = 0.12.}
    \label{fig:heatmap_A_60}
  \end{subfigure}
  \begin{subfigure}[b]{0.45\textwidth}
    \includegraphics[width=\textwidth]{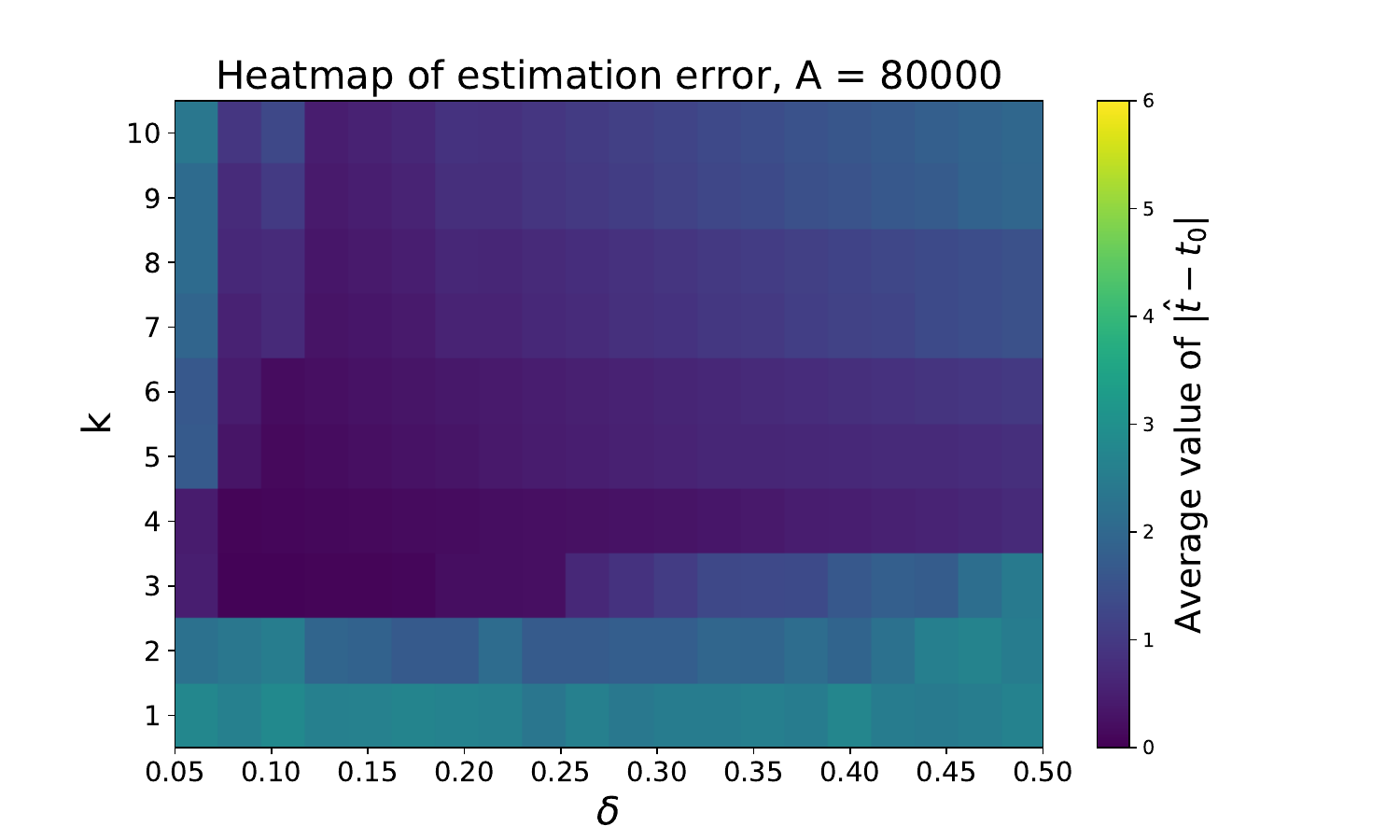}
    \caption{$(k_{\min}, \delta_{\min}) = (3, 0.07)$, estimation error = 0.05.}
    \label{fig:heatmap_A_80}
  \end{subfigure}
  \caption{Heatmaps of the difference between the estimated and true value of $t_0$, for $A \in \{20000, 40000, 60000, 80000 \}$. In the subfigure captions above, $k_{\min}$ and $\delta_{\min}$ denote the values of $k, \delta$ that minimize the estimation error in the heatmap range.}
  \label{fig:heatmap_synthetic}
\end{figure}

\subsection{The Susceptible-Infected (SI) process}
Next, we apply our methods to detect high-degree vertices in the graph underlying a SI process, providing a concrete illustration of Theorem~\ref{thm:estimating_high_degrees}.
We assume that the graph has a simple structure: it is a balanced binary tree of roughly 500,000 vertices (height = 18), with $D$ children added to one of the vertices in the second-to-last layer of the binary tree.
This construction produces a graph where all vertices have degree at most 3 except for a single high-degree node of degree $D + 2$.
While our results hold for much more general topologies, this graph provides a simple baseline for which our results can be clearly visualized.\footnote{We chose this graph in particular for a few reasons. It is faster to simulate large-scale SI processes on trees, since there is a single path between any pair of vertices. Additionally, balanced binary trees are sparse and simple to describe.}
Even for values of $D$ that are quite small relative to the size of the graph, our methods can cleanly detect when the high-degree vertex is infected. Figure \ref{fig:si_derivatives} illustrates our methods for a particular realization of the SI process with $D = 3000$.
The jump in infections caused by the high-degree vertex cannot be clearly discerned from from the cumulative infection curve (Figure \ref{fig:si_d0}) or the first derivative (Figure \ref{fig:si_d1}). However, a sharp jump can be seen from the plot of the second derivative (Figure \ref{fig:si_d2}), with a clear spike in the third derivative that aligns with the infection time of the high-degree vertex (Figure \ref{fig:si_d3}).

\begin{figure}[t]
    \centering
    \begin{subfigure}[b]{0.45\textwidth}
        \centering
        \includegraphics[width=\textwidth]{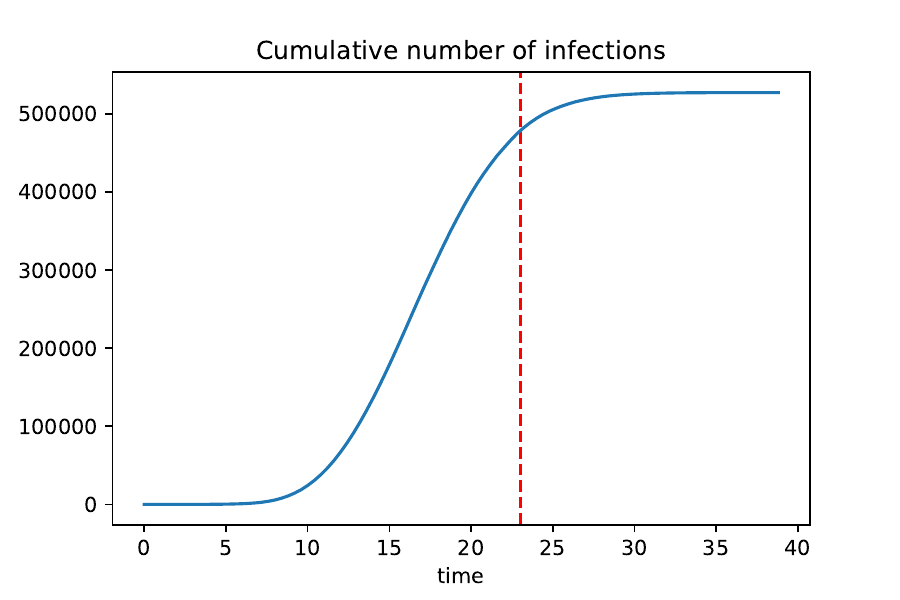}
        \caption{Plot of $N(t)$.}
        \label{fig:si_d0}
    \end{subfigure}
    \begin{subfigure}[b]{0.45\textwidth}
        \centering
        \includegraphics[width=\textwidth]{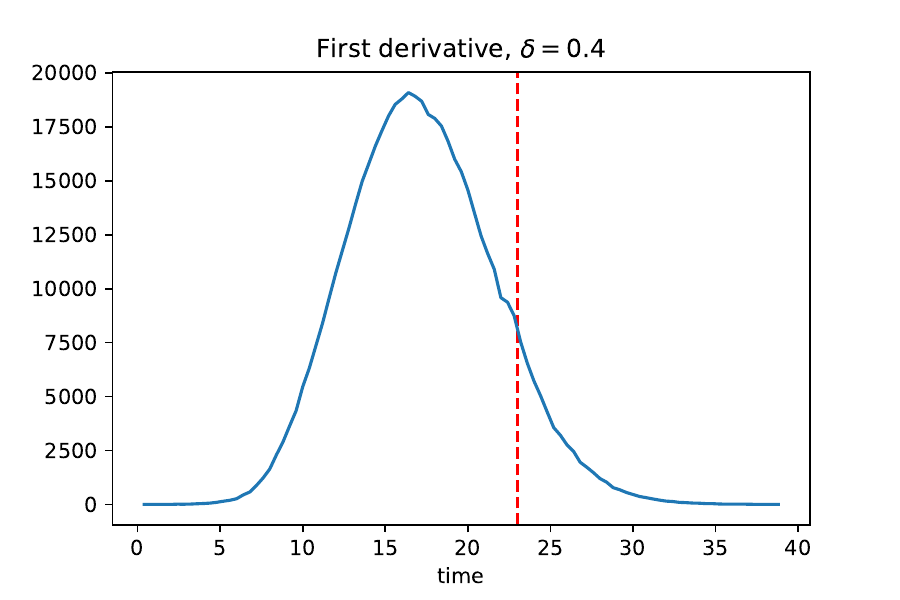}
        \caption{Plot of $\Delta_\delta^{(1)} N(t)$.}
        \label{fig:si_d1}
    \end{subfigure}
    \\
    \begin{subfigure}[b]{0.45\textwidth}
        \centering
        \includegraphics[width=\textwidth]{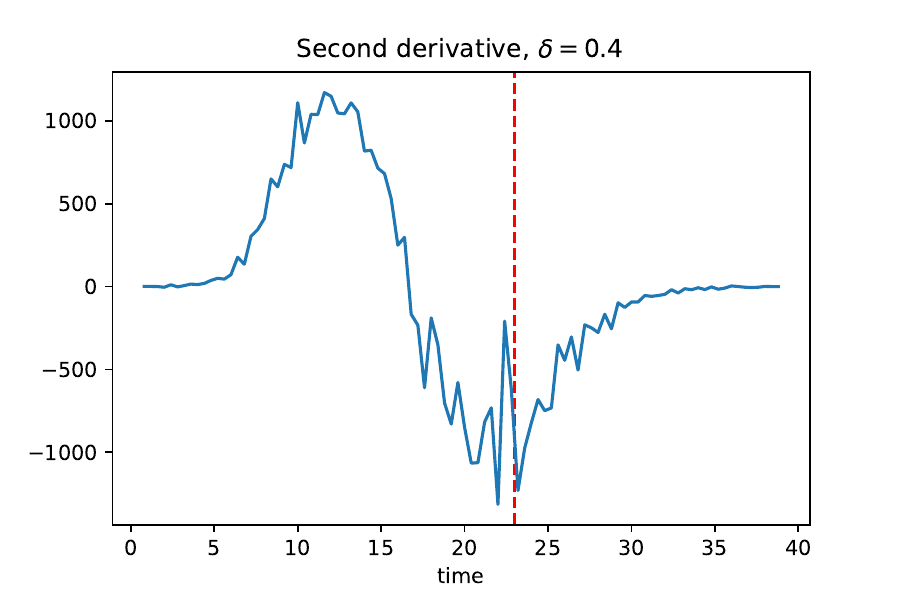}
        \caption{Plot of $\Delta_\delta^{(2)} N(t)$.}
        \label{fig:si_d2}
    \end{subfigure}
    \begin{subfigure}[b]{0.45\textwidth}
        \centering
        \includegraphics[width=\textwidth]{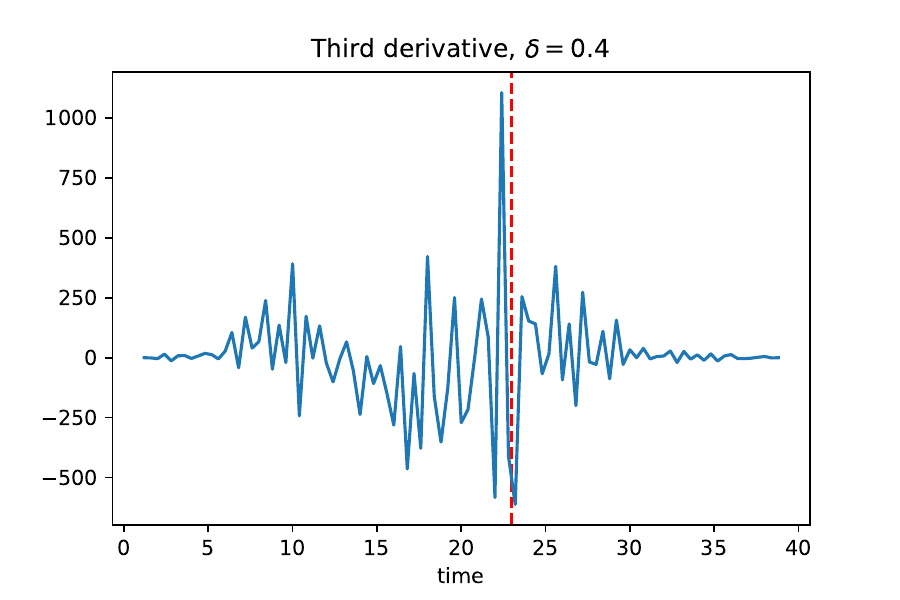}
        \caption{Plot of $\Delta_\delta^{(3)} N(t)$.}
        \label{fig:si_d3}
    \end{subfigure}
    \caption{Discrete derivatives of $N(t)$ with $\delta = 0.4$, for a SI process spreading on a balanced binary tree on approximately 500,000 vertices with a single vertex of degree $3002$. The red dotted line represents the infection time of this high-degree vertex, which is at $t = 23.00$.}
    \label{fig:si_derivatives}
\end{figure}

To assess our method's performance in more detail, we analyze the results of 200 independent simulations of the SI process on the same graph, where we choose $D \in \{2000, 4000, 6000, 8000\}$, $k \in \{1, \ldots, 5 \}$ and $\delta \in [0.1, 2.0]$. 
For a given $(k, \delta)$, our estimate of the high-degree vertex's infection time is $\hat{t} : = \argmax_{t \ge 0} | \Delta_\delta^{(k)} N(t) |$.
The estimation error, which is equal to the absolute difference between the infection time of the high-degree vertex and our estimate, is averaged over all trials for a given $(k, \delta)$ and visualized in the heatmaps of Figure \ref{fig:si_heatmaps}.

To properly assess the performance of our algorithm, it is useful to compare it to a simple baseline method. 
Since a large number of events occur in a small window after the change-point, a natural na\"ive method is to output the time index $t$ for which the largest number of events occur in the time interval $[t, t + \delta]$. In other words, we choose the argmax of the discrete first-order derivative; we refer to this as the \emph{first derivative estimator}. 
A more principled baseline, which was studied in detail for the SI process in \cite{mossel2024finding}, is the argmax of the second-order discrete derivative (succinctly, the \emph{second derivative estimator}). As seen in the heatmaps of Figure \ref{fig:si_heatmaps}, the second derivative estimator outperforms the first derivative estimator only when $D \in \{6000,8000\}$. This aligns with our theory, which posits that the second derivative estimator works well when $D \gtrsim n^{2/3} \approx 6000$. 
Finally, we note that it is clear in Figure \ref{fig:si_heatmaps} that more derivatives helps for the choices of $D$ considered, as seen by the presence of darker cells for appropriate $\delta$ and $k \ge 3$. In Figure \ref{fig:si_heatmaps}, we highlight this insight by comparing the estimation error of the optimal choices of $k,\delta$ (chosen based on the minimum error of empirical results) to the estimation error of the first and second derivative estimates, optimized over choices of $\delta$. Except for the case $D = 8000$ (Figure \ref{fig:si_D_8000}), there is a clear gain in taking more than two derivatives.

\begin{figure}[t]
    \centering
    \begin{subfigure}[b]{0.45\textwidth}
        \centering
        \includegraphics[width=\textwidth]{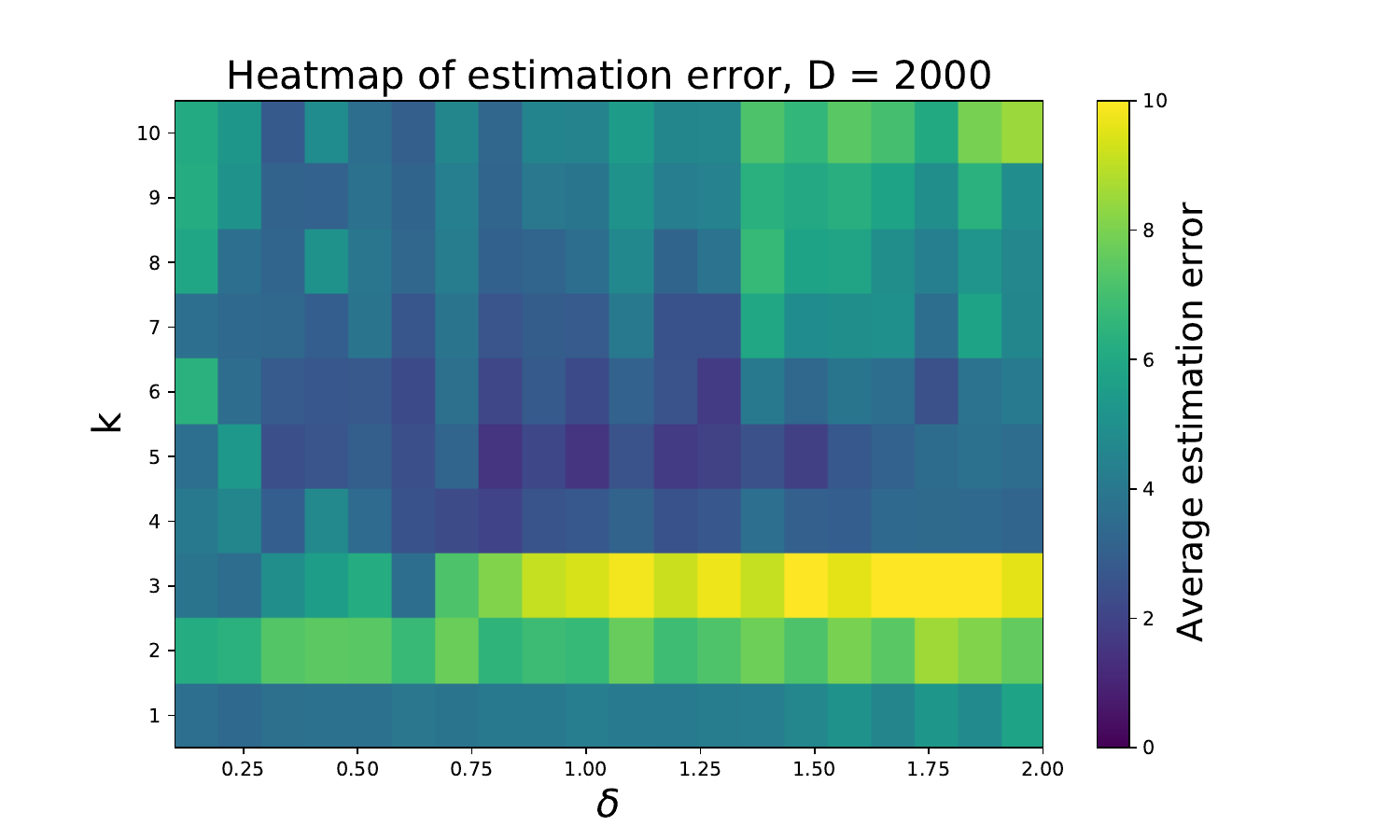}
        \caption{$(k_{\min}, \delta_{\min}) = (5, 0.8)$, estimation error = 1.48. \\
        Smallest estimation error of 1st derivative = 3.43.\\
        Smallest estimation error of 2nd derivative = 6.19.
        }
        \label{fig:si_D_2000}
    \end{subfigure}
    \begin{subfigure}[b]{0.45\textwidth}
        \centering
        \includegraphics[width=\textwidth]{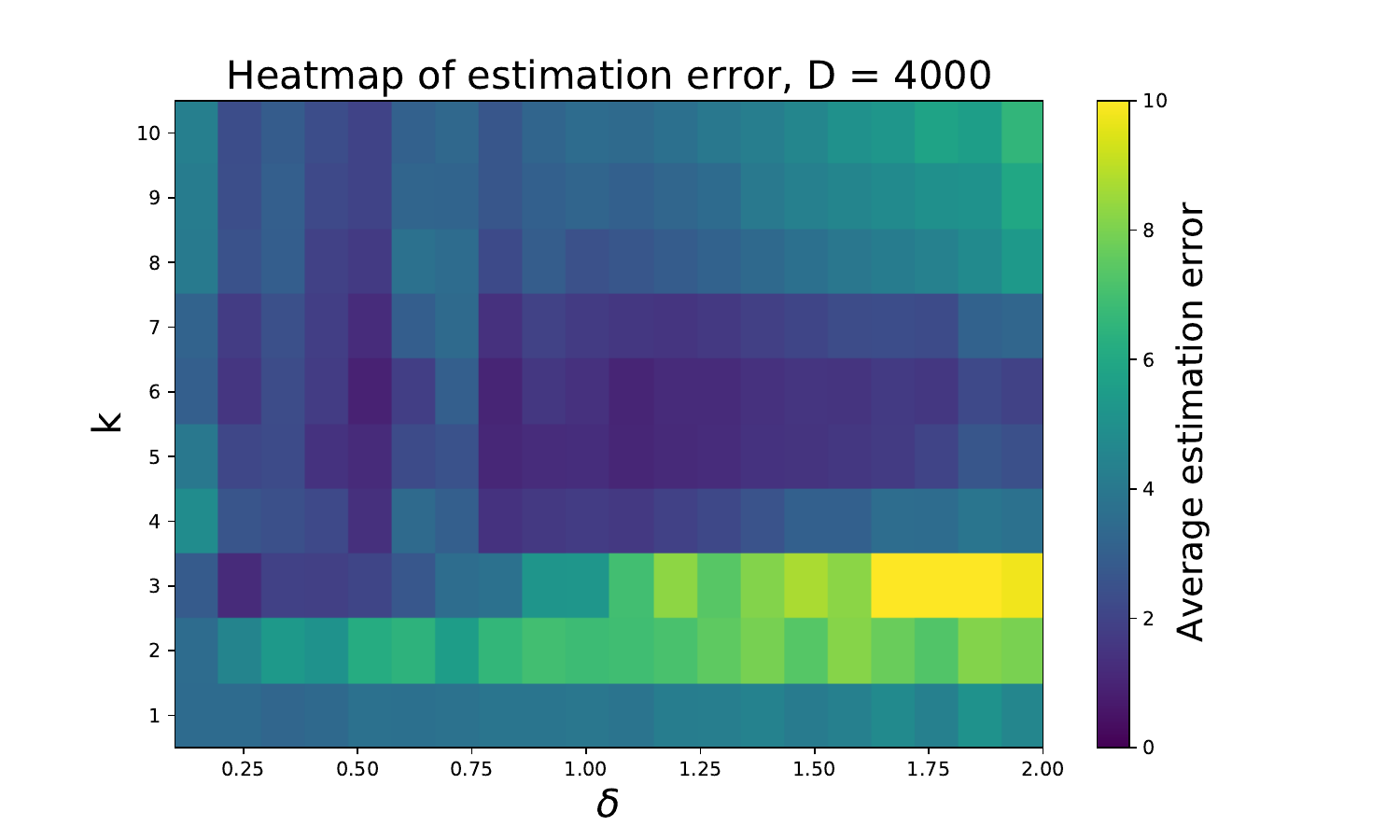}
        \caption{$(k_{\min}, \delta_{\min}) = (6, 0.5)$, estimation error = 0.95. \\
        Smallest estimation error of 1st derivative = 3.31. \\
        Smallest estimation error of 2nd derivative = 3.52.
        }
        \label{fig:si_D_4000}
    \end{subfigure}
    \\
    \begin{subfigure}[b]{0.45\textwidth}
        \centering
        \includegraphics[width=\textwidth]{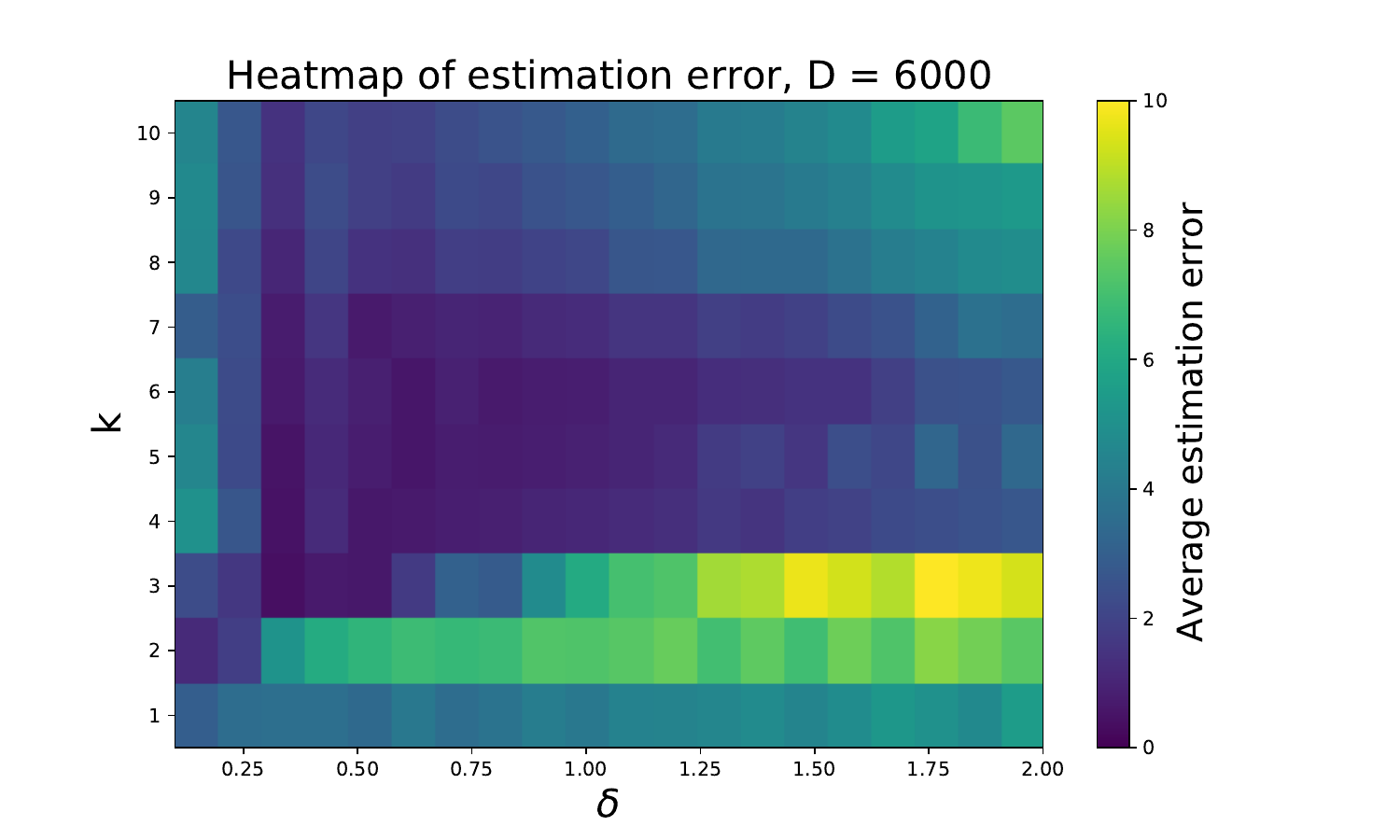}
        \caption{$(k_{\min}, \delta_{\min}) = (3, 0.3)$, estimation error = 0.42. \\
        Smallest estimation error of 1st derivative = 3.01. \\
        Smallest estimation error of 2nd derivative = 1.20.
        }
        \label{fig:si_D_6000}
    \end{subfigure}
    \begin{subfigure}[b]{0.45\textwidth}
        \centering
        \includegraphics[width=\textwidth]{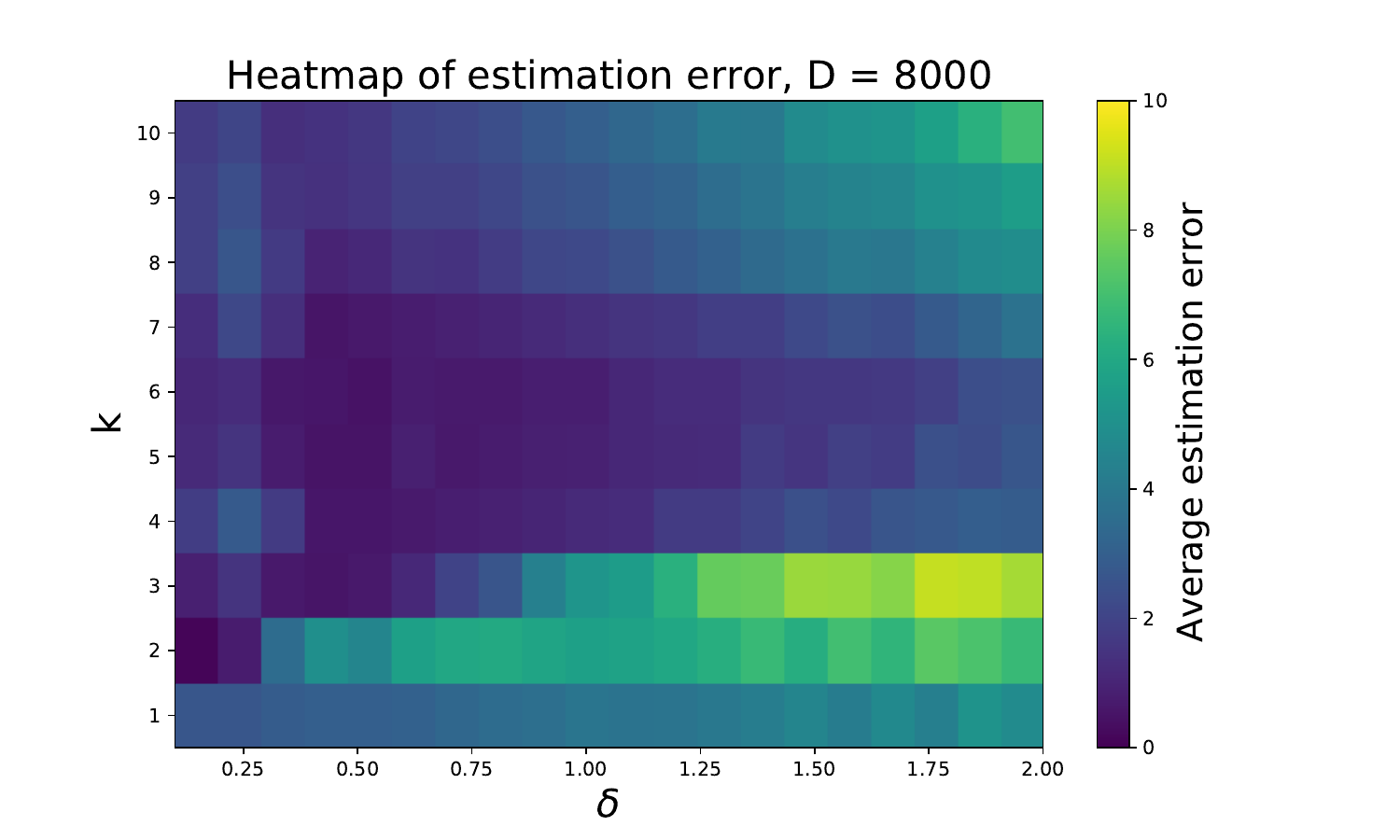}
        \caption{$(k_{\min}, \delta_{\min}) = (2, 0.1)$, estimation error = 0.14. \\
        Smallest estimation error of 1st derivative = 2.66. \\
        Smallest estimation error of 2nd derivative = 0.14.
        }
        \label{fig:si_D_8000}
    \end{subfigure}
    \caption{Heatmaps of the difference between estimated and true values of the infection time of the (unique) high-degree vertex, for $D \in \{ 2000, 4000, 6000, 8000 \}$, averaged over 200 independent simulations. In the subfigure captions above, $k_{\min}$ and $\delta_{\min}$ denote the values of $k, \delta$ that minimize the estimation error in the heatmap range.}
    \label{fig:si_heatmaps}
\end{figure}

\subsection{Detecting a super-spreader event}

A natural concrete application of our methods is detecting super-spreading events in the context of an epidemic. To be precise, we define a super-spreading event to be a scenario where a large number of individuals are suddenly exposed to (but not necessarily infected by) a pathogen. Noting that the number of infections in a population is a counting process with a time-varying rate equal to the number of exposed individuals at a given point in time (see Section~\ref{subsec:SI-results}), it is clear that a super-spreading event causes an abrupt additive shift in the rate function. 
Of course, real-time epidemic data is much more complicated than a point process model since the number of infections is not an observable process. Indeed, there are usually delays between infection events, symptom onset, and reported positive test results. 
Additional complications include partial observability (i.e., not everyone will develop symptoms or get tested) and quantization (i.e., an individual's infection time may not be known, and only aggregate metrics, like the number of daily cases, might be available). 
Nevertheless, as we show below, our methods can still detect super-spreading events in cases where their impact may appear to be obfuscated by broader trends. 
Our analysis is based on a public dataset of daily case counts of COVID-19 per county in the US, which was compiled by the New York Times from 2020 until 2023 \cite{nytimes_coronavirus}.

\begin{figure}[t]
  \centering
  \begin{subfigure}{0.45\textwidth}
    \centering
    \includegraphics[width=\linewidth]{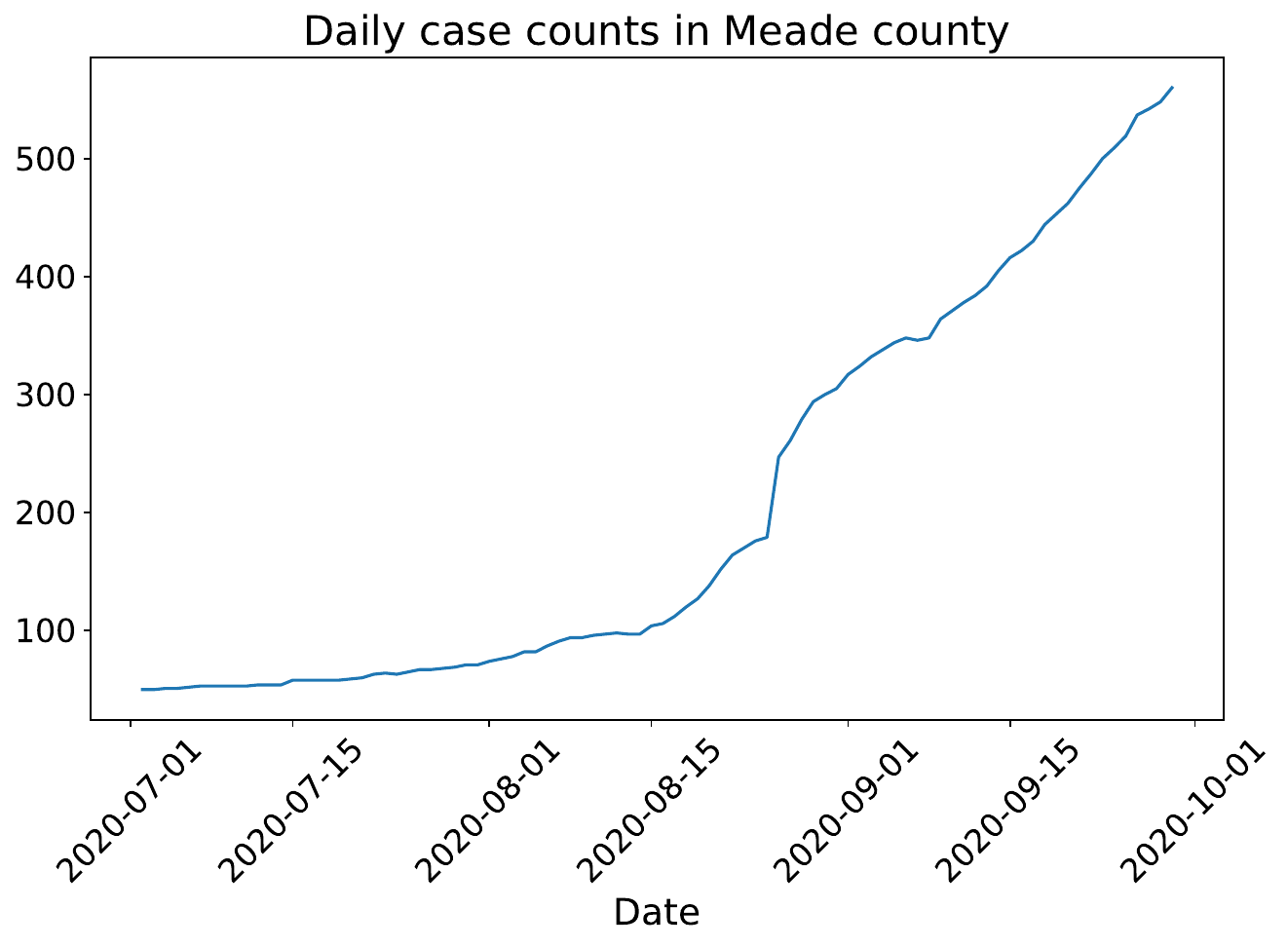}
    \caption{July 1 -- October 1, 2020.}
    \label{fig:meade1}
  \end{subfigure}
  \hfill
  \begin{subfigure}{0.45\textwidth}
    \centering
    \includegraphics[width=\linewidth]{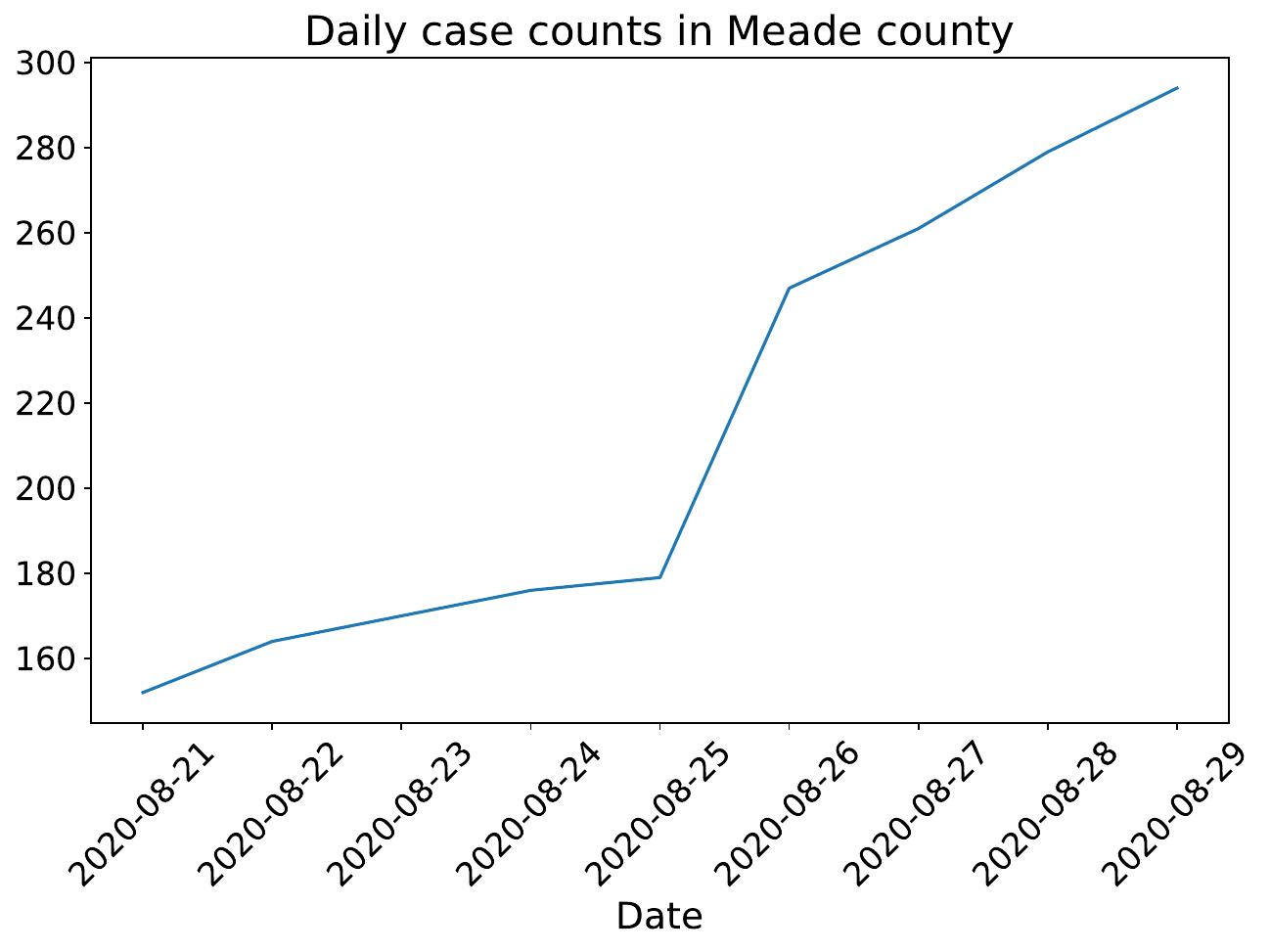}
    \caption{August 21-29, 2020.}
    \label{fig:meade2}
  \end{subfigure} \\
  \begin{subfigure}{0.45\textwidth}
    \centering
    \includegraphics[width=\linewidth]{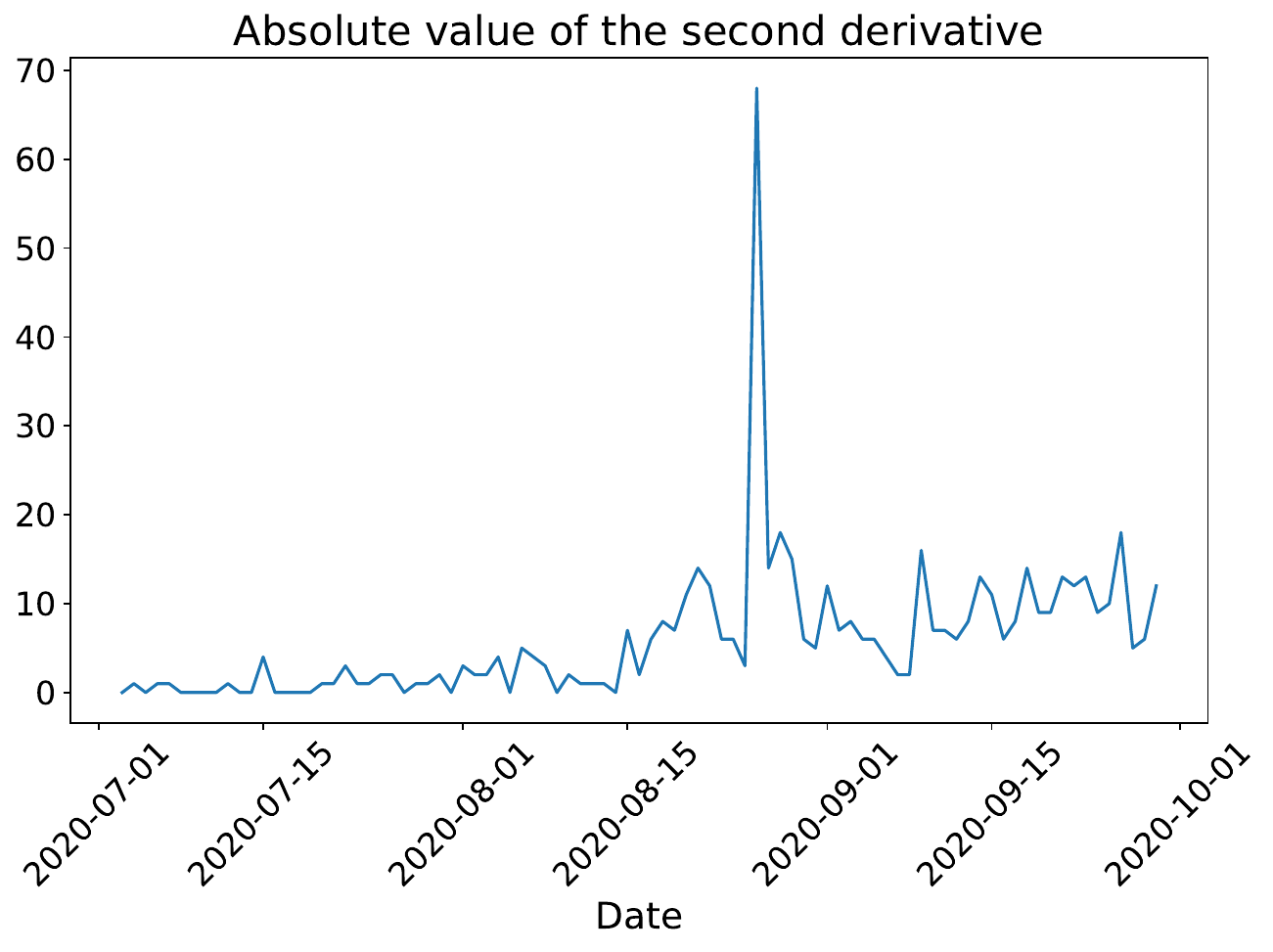}
    \caption{Second derivative, July 1 -- October 1, 2020.}
    \label{fig:meade3}
  \end{subfigure}
  \caption{Analysis of daily case counts in Meade county. Figure \ref{fig:meade1} shows a sharp increase in the daily number of cases following the rally; a closer inspection (Figure \ref{fig:meade2}) shows that the abrupt change happens over the course of a day, on August 26. This sharp change can be readily seen through the second derivative of the cumulative infection curve (Figure \ref{fig:meade3}), where the largest value occurs on August 26.}
  \label{fig:meade_county}
\end{figure}

As a specific case study, we examine the impact of the 2020 Sturgis Motorcycle Rally, a well-documented super-spreader event for COVID-19. The rally took place from August 7 - 16, 2020 in Sturgis, Meade County, South Dakota and had over 460,000 throughout the duration of the event (as a point of comparison, there were ~7,000 permanent residents of Sturgis and ~26,000 residents of Meade county at the time). No pandemic mitigation measures were strictly enforced, though masks were encouraged \cite{sturgis_superspreader}.  
While there were several documented super-spreading events during the course of the COVID-19 pandemic in 2020, we chose to study the Sturgis rally primarily for two reasons. First, the rally marked a significant increase in the rate of close interactions in a small city, making the corresponding abrupt change in the number of reported cases clear to see in daily case counts (see Figures \ref{fig:meade1} and \ref{fig:meade2}).
Second, there were no other recorded events in Meade county during the period July 1 - October 1, 2020 of a similar scale to the best of our knowledge, allowing us to effectively isolate the rally's impact (at least on the county scale). 

We now discuss the implications of our methods in the context of the rally. Let us first examine the cases in Meade county in detail. As shown by Figures \ref{fig:meade1} and \ref{fig:meade2}, the abrupt jump in cases can be seen from the daily case counts itself. This jump is more clearly identified from the plot of the second-order discrete derivative of the cumulative infection curve (which is also the first-order discrete derivative of the number of cases). As seen in Figure \ref{fig:meade3}, the peak in the second derivative occurs on August 26, 2020. We remark that this date is 10 days after the end of the rally. However, such delays are expected, due to the time it takes for symptom onset after infection, and the subsequent delay in testing for infection.

We next examined the aggregated daily case counts across all of South Dakota for the same time frame. It is worth noting that this aggregation reflects a significant increase in scale, with daily cases reaching up to ~22,000 compared to a maximum of ~500 cases per day in Meade County during the period analyzed.
As seen in Figure \ref{fig:sd1}, the impact of the rally is far less clear from daily case counts. If we examine the second derivative of the infection curve (equivalently, the first derivative of the daily case counts), we see a large value on August 26; see Figure \ref{fig:sd2}. Note that we also see several peaks of significance afterwards (even though there were no other super-spreading events of a similar magnitude in South Dakota during this time frame, to the best of our knowledge), with the largest peak slightly exceeding the value at August 26.
If we take another derivative (i.e., the third derivative of the cumulative infection curve and the second derivative of the daily case counts), we see a clear and clean spike on August 26; see Figure \ref{fig:sd3}. 
This can be explained as follows.
The first spike in the second derivative plot marks an abrupt change from low to high values. However, the subsequent spikes, while significant, appear to be a part of a broader trend of increasing values, and are also not as sharp as the first spike. Taking an additional derivative effectively zeroes out the broader trends in the background, exposing only the most abrupt change.

\begin{figure}[h]
  \centering
  \begin{subfigure}{0.45\textwidth}
    \centering
    \includegraphics[width=\textwidth]{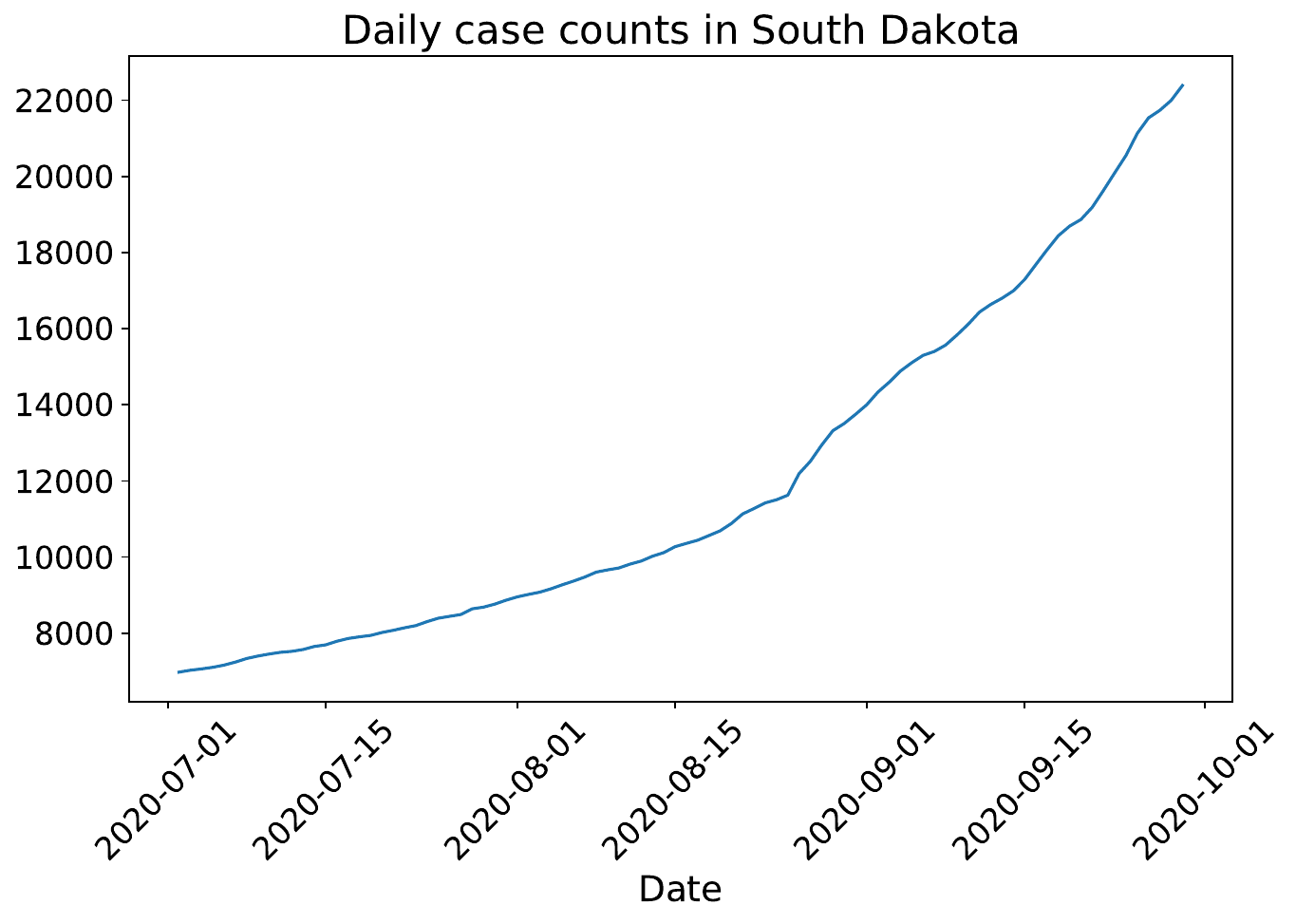}
    \caption{Plot of the first derivative.}
    \label{fig:sd1}
  \end{subfigure}
  \begin{subfigure}{0.45\textwidth}
    \centering
    \includegraphics[width=\textwidth]{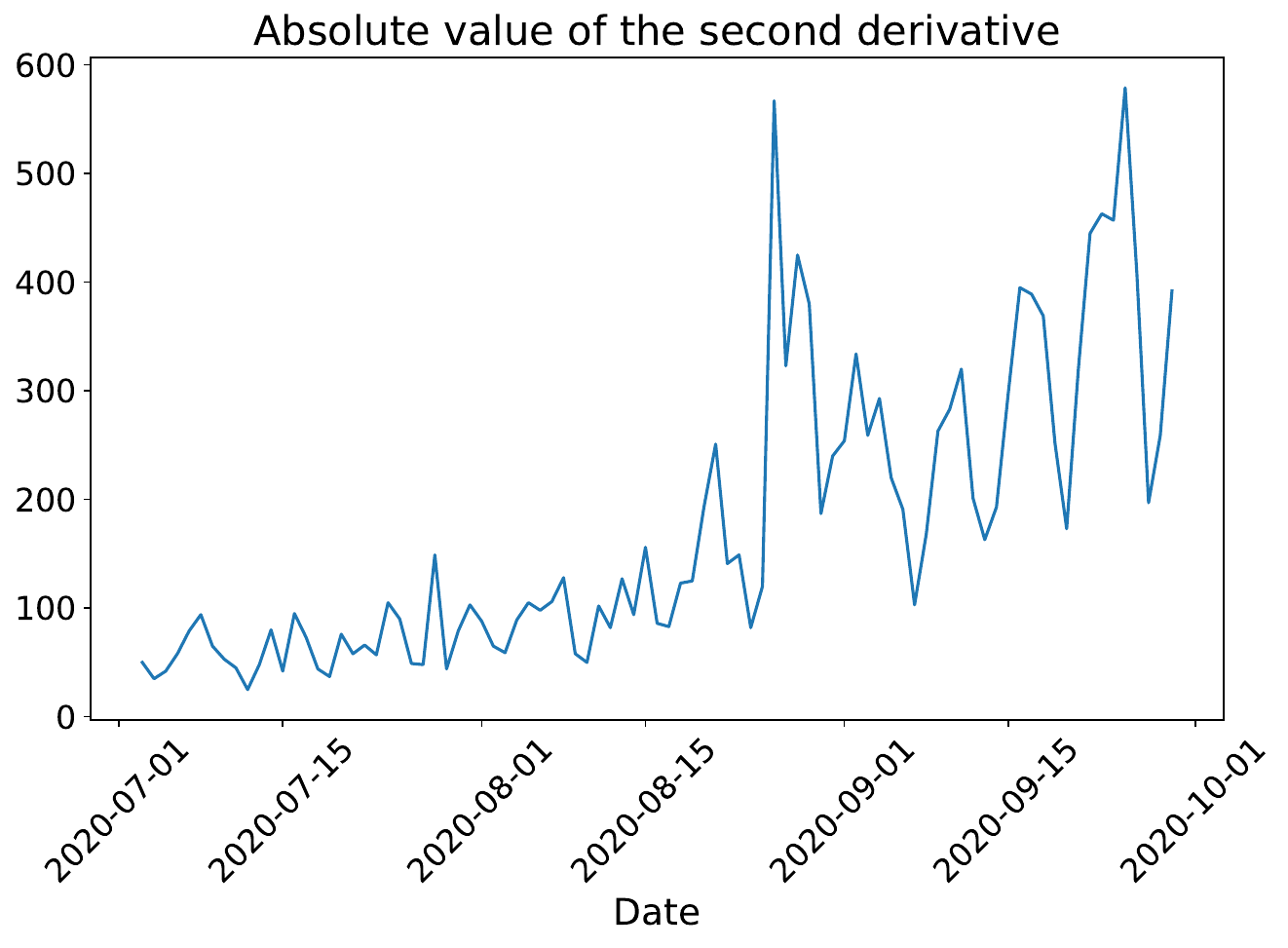}
    \caption{Plot of the second derivative.}
    \label{fig:sd2}
  \end{subfigure}
  \par\bigskip
  \begin{subfigure}{0.45\textwidth}
    \centering
    \includegraphics[width=\textwidth]{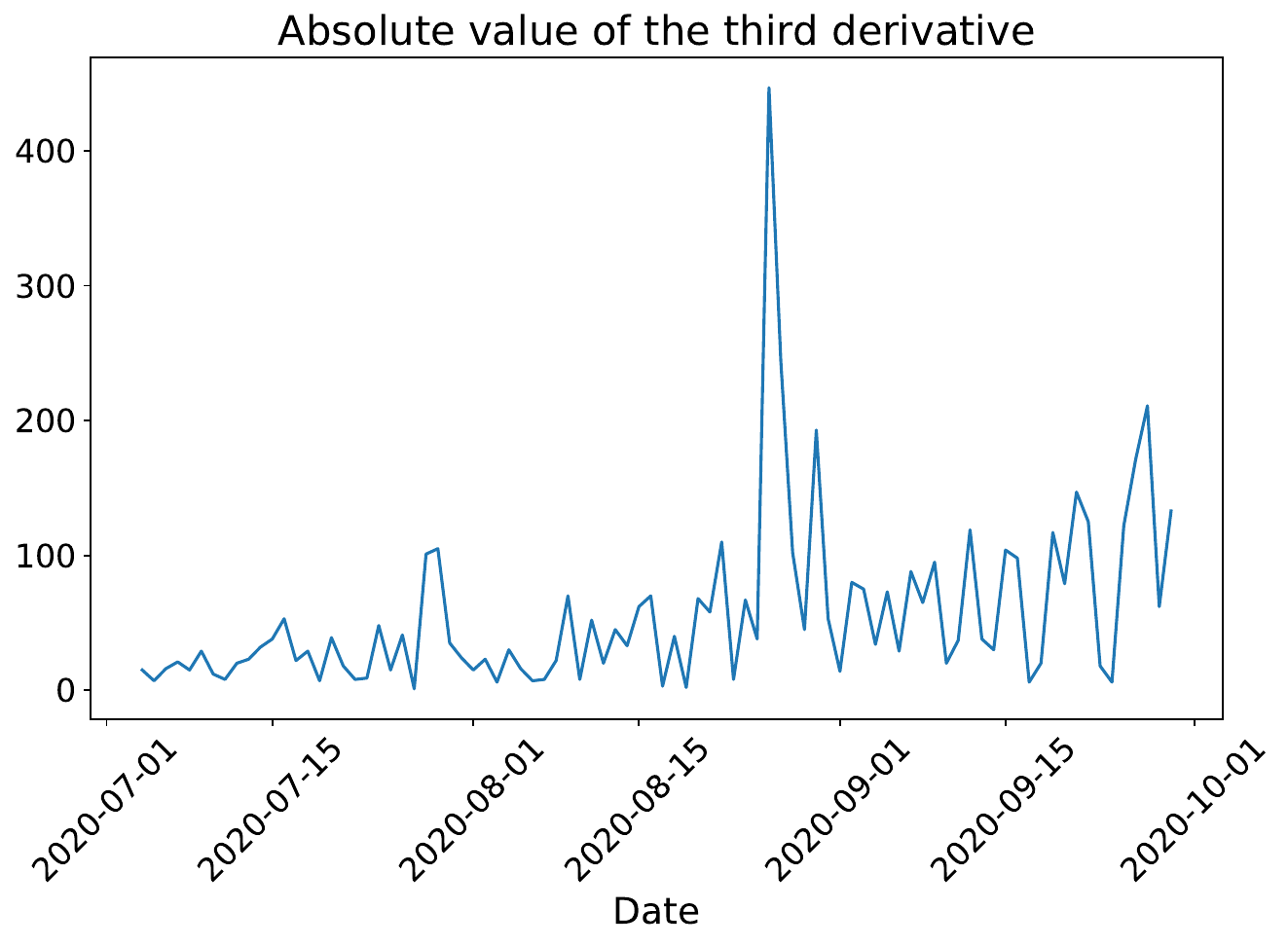}
    \caption{Plot of the third derivative.}
    \label{fig:sd3}
  \end{subfigure}
  \caption{Analysis of COVID-19 case counts aggregated across all of South Dakota, from July 1 - October 1, 2020. While the discontinuity at August 26 can still be seen in the daily case counts (Figure \ref{fig:sd1}), the impact is less clear compared to Figure \ref{fig:meade1}. While a spike in the second derivative can be observed on August 26, there are other (potentially spurious) large values, notably at the end of September 2020. The third derivative (Figure \ref{fig:sd3}) shows that the event on August 26 marks the most abrupt change.}
  \label{fig:your_label}
\end{figure}

\section{Polynomial approximation: Proof of Theorem \ref{thm:approximation_error}}
\label{sec:polynomial_approximation}

In this section, we prove Theorem \ref{thm:approximation_error}. 
Our first result is a probabilistic upper bound on the approximation errors $|N(s) - \overline{N}_\ell(s, t) |$ and $| N(t) - \overline{N}_\ell(t, s) |$ for a \emph{fixed} $s,t \in [0,T]$.

\begin{lem}
\label{lemma:polynomial_approximation_fixed_t}
Let $t, \delta, x > 0$, let $s \in (t, t + 2\delta)$, and fix an integer $\ell \ge 1$. 
Define the event
\[
\cA_{t, s}(x) : = \left \{ \max \left \{  |N(s) - \overline{N}_\ell(s, t) |,  |N(t) - \overline{N}_\ell(t, s ) | \right \} \ge (2\delta)^{\ell + 1} S_{\ell + 1} + x \right \}.
\]
Then \[
\p ( \cA_{t,s}(x) \cap \{ \cT \cap (t, s] = \emptyset \} ) \le 2 (\ell + 1) \exp \left( - \frac{x^2 / 2}{6 J_{\le \ell}^2 S_{1} \delta + J_{\le \ell} x / 3} \right).
\]
\end{lem}

The lemma is proved by bounding the error processes by the maximal deviation of certain martingales. The probability bound on the right hand side comes from standard martingale concentration results (i.e., Freedman's inequality \cite[Appendix B]{shorack_wellner}). Full details of the proof can be found in Section \ref{subsec:polynomial_approximation}.

Next, we translate our results for a \emph{fixed} $t,s$ to a simultaneous bound for \emph{all} $t,s \in [0,T]$ through a discretization argument (similar to the methods in \cite[Appendix B]{mossel2024finding}).
For a given $\varepsilon > 0$ (to be determined), we define the set of discretized time indices
\[
\mathbb{T} : = [0,T] \cap \{ k \varepsilon : k \in \mathbb{Z}_{\ge 0} \}. 
\]
Our next result shows that for any $t \in [0,T]$, the error processes can be essentially understood by examining only $t \in \mathbb{T}$, for $\varepsilon$ sufficiently small. We defer the proof details to Section \ref{subsec:discretization}.

\begin{lem}
\label{lemma:discretization}
Fix $T > 0, \delta \in (0, 1/2)$ and $\eta > 0$. Additionally, define
\begin{equation}
\label{eq:discretization_epsilon}
\varepsilon = \frac{\eta ( \sep( \cD) \land 1)}{6 TR^2 (L_{\le \ell} + S_{\le \ell})}.
\end{equation}
With probability at least $1 - \eta$, it holds for every $t \in [0,T]$ and $s \in (t, t + \delta)$ that there exist $t' \in \mathbb{T}$ and $s' \in \mathbb{T} \cap (t', t' + 2\delta)$ satisfying 
\begin{align}
\label{eq:discretization_bound_1}
 |N(s) - \overline{N}_\ell(s,t) | & \le  |N(s') - \overline{N}_\ell(s',t') | + 1, \\
\label{eq:discretization_bound_2}
 | N(t) - \overline{N}_\ell(t, s) | & \le |N(t') - \overline{N}_\ell(t',s') | + 1.
\end{align}
\end{lem}

Together, Lemmas \ref{lemma:polynomial_approximation_fixed_t} and \ref{lemma:discretization} prove the theorem.

\begin{proof}[Proof of Theorem \ref{thm:approximation_error}]

Let $\eta > 0$ and let $\varepsilon$ be given by \eqref{eq:discretization_epsilon}. We claim that if $x = c_1' J_{\le \ell} \sqrt{ \delta S_1 \log \left( \frac{T}{\varepsilon \eta} \right) }$ for some constant $c_1' = c_1'(\ell) >0$, then 
\begin{equation}
\label{eq:Eapprox_inclusion}
\bigcap_{t \in [0,T], s \in (t, t + \delta)} \left( \cA_{t,s}(x) \cap \{ \cT \cap (t, s] = \emptyset \} \right)^c \subseteq \Eapprox,
\end{equation}
with $c_1 = \max \{ 2^{\ell + 1}, c_1' \}$ in the definition of $\err$ in \eqref{eq:error_bound}.
Indeed, if the event on the left hand side of \eqref{eq:Eapprox_inclusion} holds, then Item \ref{item:forward_prediction} of Theorem \ref{thm:approximation_error} is immediate. Additionally, if $\cT \cap (t - \delta, t) = \emptyset$ and $t - \delta < s < r< t$, then $|N(s) - \overline{N}_\ell(s,r) | \le (2 \delta)^{\ell + 1} S_{\le \ell + 1} + x$.
Since the $\lambda^{(k)}$'s are \cadlag by Assumption \ref{as:lambda_k}, the process $\overline{N}_\ell(s, r)$ is also c\`{a}dl\`{a}g in $r$. Sending $r \to t^-$, we obtain $| N(s) - \overline{N}_\ell(s, t^-) | \le (2 \delta)^{\ell + 1} S_{\le \ell + 1} + x \le \delta \err$, hence Item \ref{item:backward_prediction} of Theorem \ref{thm:approximation_error} holds.

We proceed by upper bounding the probability of the complement of the left hand side of \eqref{eq:Eapprox_inclusion}, showing in particular that there exists $c_1'$ such that probability of the aforementioned event is at most $2 \eta$.
Let $\Edisc$ denote the event described in the statement of Lemma \ref{lemma:discretization}. 
Then we can bound
\begin{align}
& \p \left( \Edisc \cap \bigcup_{t \in [0,T], s \in (t, t + \delta)} \cA_{t,s} \left(\frac{x}{2} + 1 \right) \cap \{ \cT \cap (t, s] = \emptyset \}  \right) \nonumber \\
& \hspace{4cm}  \stackrel{(a)}{\le} \p \left(  \bigcup_{t,s \in \mathbb{T}: s \in (t, t +  2\delta)} \cA_{t,s}\left( \frac{x}{2} \right) \cap \{ \cT \cap (t, s] = \emptyset \} \right) \nonumber \\
& \hspace{4cm} \le \sum_{t,s \in \mathbb{T}: s \in (t, t +  2\delta)} \p ( \cA_{t,s} \left( \frac{x}{2} \right) \cap \{ \cT \cap (t, s] = \emptyset \} ) \nonumber \\
\label{eq:At_union_bound}
& \hspace{4cm} \stackrel{(b)}{\le} \frac{2 (\ell + 1) T^2}{\varepsilon^2} \exp \left( - \frac{x^2 / 8}{ 6J_{\le \ell}^2 S_{1} \delta + J_{\le \ell} x / 6}\right).
\end{align}
Above, $(a)$ is due to Lemma \ref{lemma:discretization} and $(b)$ follows from Lemma \ref{lemma:polynomial_approximation_fixed_t} as well as the bound $| \mathbb{T} | \le T / \varepsilon$.
To simplify the right hand side of \eqref{eq:At_union_bound}, assume that 
\[
\delta \ge \frac{4 \ell}{S_1} \log \left( \frac{T}{\varepsilon \eta} \right) \text{ and } x :=   \sqrt{ 120 \ell \delta J_{\le \ell}^2 S_{1} \log \left( \frac{T}{\varepsilon \eta} \right)} .
\]
It can be seen that with this choice of $x$ and $\delta$ satisfying the bound \eqref{eq:delta_lower_bound_generic}, the right hand side of \eqref{eq:At_union_bound} is at most $\eta$. 
It is also useful to note that for $\varepsilon \le \eta / T$ and $J_{\le \ell} \ge 1$, we have that $x \ge 2$.
Using the inequality $x/2 + 1 \le x$ as well as the result $\p ( \Edisc) \ge 1- \eta$ from Lemma \ref{lemma:discretization} shows that
\begin{multline*}
\p \left( \bigcup_{t \in [0,T], s \in (t, t + \delta)} \cA_{t,s}(x) \cap \{ \cT \cap (t,s] = \emptyset \}  \right) \\
\le \p ( \Edisc^c) + 
\p \left( \Edisc \cap \bigcup_{t \in [0,T], s \in (t, t + \delta)} \cA_{t,s} \left( \frac{x}{2} + 1 \right) \cap \{ \cT \cap (t,s] = \emptyset \} \right)  \le 2\eta. 
\end{multline*}
Finally, the theorem follows from setting $\eta = 1/S_{\le \ell}$. 
\end{proof}

\subsection{Bounding errors for a fixed $t$: Proof of Lemma \ref{lemma:polynomial_approximation_fixed_t}}
\label{subsec:polynomial_approximation}

Define the following stochastic process for $k \ge 0$:
\[
X_k(s) : = \lambda^{(k)}(s) - \lambda^{(k)}(t) - \int_t^s \lambda^{(k + 1)}(x) dx.
\] 
Our first supporting result shows that the approximation error between the processes $N(s)$ and $\overline{N}_\ell(s, t)$ can be bounded by the maximal fluctuations of the $X_k$'s. We remark that the result below is fully deterministic.

\begin{lem}
\label{lemma:poly_approx_deterministic}
Let $t, \delta > 0$, let $\ell \ge 1$ be an integer and let $s \in (t, t + \delta)$. Then
\begin{align*}
& \max \left \{ | N(s) - \overline{N}_\ell( s, t) |  , |N(t) - \overline{N}_\ell(t, s) | \right \} \le \delta^{\ell + 1} S_{\ell + 1} + \sum_{k = 0}^\ell \delta^k \sup_{r \in (t, s]} |X_k(r) |. 
\end{align*}
\end{lem}

\begin{proof}
Let $t < s < t + \delta$.
We first prove the bound in the lemma statement holds for $| N(s) - \overline{N}_\ell( s, t) |$. We will prove the slightly stronger claim that for each $i = 0, 1, \ldots, \ell$, it holds for all $s \in (t, t + \delta)$ that
\begin{equation}
\label{eq:taylor_induction_hypothesis}
 \left| \lambda^{(i)}(s)  - \left( \lambda^{(i)}(t) + \sum_{k = i + 1}^\ell \frac{\lambda^{(k)}(t)}{(k  - i)!} (s  - t)^{k - i} \right) \right| \le \delta^{\ell + 1 - i} S_{\ell + 1} + \sum_{k = i}^{\ell} \delta^{k - i } \sup_{r \in (t, s]} |X_k(r) |
\end{equation}
Our proof proceeds by induction. For the base case $i = \ell$, we have for $s \in (t, t + \delta)$ that
\begin{align*}
 \left | \lambda^{(\ell)}(s) - \lambda^{(\ell)}(t) \right| & \le | X_\ell(s) | + \left| \int_t^{s} \lambda^{(\ell + 1)}(x) dx \right|  \le \sup_{r \in (t, s]} |X_\ell(r) | + \delta S_{\ell + 1}.
\end{align*}
Above, our bound on the integral is due to the inequality $| \lambda^{(\ell + 1)}(x) | \le S_{\ell + 1}$, which holds under Assumption \ref{as:regularity}.

We now prove that the inductive step holds. Assume that \eqref{eq:taylor_induction_hypothesis} holds for some $1 \le i \le \ell$. Then for any $s \in (t , t + \delta)$, 
\begin{align}
& \left| \lambda^{(i - 1)}(s) - \left( \lambda^{(i - 1)}(t) + \sum_{k = i}^\ell \frac{\lambda^{(k)}(t)}{(k - i + 1)!} (s  - t)^{k - i +1} \right) \right| \nonumber \\ 
\label{eq:taylor_approx_bound_1}
& \hspace{1cm} \le |X_{i-1}(s )| + \left| \int_t^{s } \lambda^{(i)}(x) dx - \sum_{k = i}^\ell \frac{\lambda^{(k)}(t)}{(k - i + 1)!} (s  - t)^{k - i +1} \right| \\
\label{eq:taylor_approx_bound_2}
& \hspace{1cm} \le \left|\int_t^{s} \left( \lambda^{(i)}(t) + \sum_{k = i + 1}^\ell \frac{ \lambda^{(k)}(t)}{(k - i)!} (x - t)^k \right) dx - \sum_{k = i}^\ell \frac{\lambda^{(k)}(t)}{(k - i + 1)!} (s - t)^{k - i +1}  \right|  \\
\label{eq:taylor_approx_bound_3}
& \hspace{2cm} + |X_{i - 1}(s ) | + \left| \int_t^{s } \delta^{\ell + 1 - i} S_{\ell+1} dx \right|  + \left| \int_t^{s } \sum_{k = i}^{\ell} \delta^{k - i } \sup_{r \in (t, x]} |X_i(r) | dx \right|  \\
\label{eq:taylor_approx_bound_4}
& \hspace{1cm} \le  \delta^{\ell +1 - (i - 1)} S_{\ell + 1}  + \sum_{k = i -1}^{\ell - 1} \delta^{k - (i - 1)} \sup_{r \in (t, s]} |X_i(r) |.
\end{align}
In the display above, \eqref{eq:taylor_approx_bound_1} follows from replacing $\lambda^{(i - 1)}(s) - \lambda^{(i - 1)}(t)$ with $X_{i - 1}(s) + \int_t^s \lambda^{(i)}(x) dx$ and applying the triangle inequality. 
The inequality in \eqref{eq:taylor_approx_bound_2} holds by the induction hypothesis.
Finally, \eqref{eq:taylor_approx_bound_4} follows since the expression in \eqref{eq:taylor_approx_bound_2} is equal to 0, and the expressions in \eqref{eq:taylor_approx_bound_3} can be bounded by those in \eqref{eq:taylor_approx_bound_4} since $|s  - t | \le \delta$. Hence \eqref{eq:taylor_induction_hypothesis} is proved, and the bound in the lemma for $|N(s) - \overline{N}_\ell(s, t) |$ follows from the special case $i = 0$.

We now bound $|N(t) - \overline{N}_\ell(t, s) |$. Using identical arguments as above, we can prove by induction that for each $i = 0, \ldots, \ell$, 
\begin{equation*}
\left| \lambda^{(i)}(t) - \left( \lambda^{(i)}(s) + \sum_{k= i + 1}^\ell \frac{\lambda^{(k)}(s)}{(k - i)!} (t - s )^{k - i} \right) \right| \le \delta^{\ell + 1 - i} S_{\ell + 1} + \sum_{k = i}^{\ell} \delta^{k - i } \sup_{r \in (t, s]} |X_k(r) |.
\end{equation*}
The special case $i = 0$ proves the claimed result for $|N(t) - \overline{N}_\ell( t, s) |$.
\end{proof}

Our natural next objective is to control the $X_k$'s, which is fleshed out in the next result. 

\begin{lem}
\label{lemma:X_concentration}
Fix integers $\ell \ge k \ge 0$ and let $0 \le t < s < t + \delta$.
It holds that 
\[
\p \left( \left \{ \sup_{r \in (t, s]} | X_k(r ) | \ge x \right \} \cap \{ \cT \cap (t, s] = \emptyset \}  \right) \le 2 \exp \left( - \frac{x^2 / 2}{3 J_k^2 S_1 \delta + J_k x / 3} \right).
\]
\end{lem}

An important building block to the proof of Lemma \ref{lemma:X_concentration} is a standard concentration inequality for continuous-time martingales known as Freedman's inequality. The version stated below is adapted from \cite[Appendix B]{shorack_wellner}.

\begin{lem}[Freedman's inequality for continuous-time martingales]
\label{lemma:freedman}
Suppose that $\{ Y(t) \}_{t \ge 0}$ is a real-valued continuous-time martingale adapted to the filtration $\{ \cF_t \}_{t \ge 0}$. Suppose further that $Y(t)$ has jumps that are bounded by $C > 0$ in absolute value, almost surely.
For $s \ge t \ge 0$, define the predictable quadratic variation $\langle Y \rangle_t^s$ to be the unique predictable process so that $Y(s)^2 - \langle Y \rangle_t^s$ is a martingale.
Then for any $x \ge 0$ and $\sigma > 0$, it holds that 
\[
\p ( Y(s) - Y(t) \ge x \text{ and } \langle Y \rangle_t^s \le \sigma^2 \text{ for some $s \ge t$} ) \le \exp \left( - \frac{x^2 / 2}{\sigma^2 + Cx / 3} \right).
\]
\end{lem}

We are now ready to prove Lemma \ref{lemma:X_concentration}.

\begin{proof}[Proof of Lemma \ref{lemma:X_concentration}]
For $r \ge t$ and $k \ge 0$, define the process 
\[
Z_k(r) : = \int_t^r \mathbf{1}( x \notin \cT ) d \lambda^{(k)}(x) - \int_t^r \lambda^{(k + 1)}(x) dx.
\]
Importantly, by the definition of $\lambda^{(k + 1)}$, it follows that $Z_k$ is a martingale. Moreover, if $\cT \cap (t, s] = \emptyset$, then it holds for all $r \in (t, s]$ that $X_k(r) = Z_k(r)$. It therefore suffices to establish upper tail bounds for the process $\sup_{r \in (t, s]} |Z_k(r) |$. 

By Assumption \ref{as:regularity}, the jumps of $Z_k$ are at most $J_k$ in magnitude; it remains to bound the predictable quadratic variation. To this end, by Assumption \ref{as:lambda_k} we can write
\[
Z_k(r) = \int_t^r \left( \alpha^{(k)}(x) \mathbf{1}( x \notin \cT) - \lambda^{(k + 1)}(x) \right) dx + \int_t^r \beta^{(k)}(x) \mathbf{1}(x \notin \cT) dN(x) + \sum_{x \in \cD \cap (t, r]} \gamma^{(k)}(x) \mathbf{1}(x \notin \cT)
\]
We then have for $\xi \ge 0$ sufficiently small that 
\begin{multline}
\label{eq:quadratic_variation_expansion}
\E [ \left. (Z_k(r + \xi) - Z_k(r) )^2 \vert \cF_r] \le 3 \E \left[ \left( \int_r^{r + \xi} \left( \alpha^{(k)}(x) \mathbf{1}(x \notin \cT) - \lambda^{(k + 1)}(x)  \right) dx \right)^2 \right \vert \cF_r \right] \\
+ 3 \E \left[ \left. \left( \int_r^{r + \xi} \beta^{(k)}(x) \mathbf{1}(x \notin \cT ) dN(x) \right)^2 \right \vert \cF_r \right] + 3 \mathbb{E} \left[\left. \Big( \sum_{x \in \cD \cap (r, r + \xi]} \gamma^{(k)}(x) \mathbf{1}(x \notin \cT) \Big)^2 \right \vert \cF_r \right]
\end{multline}
The first term in \eqref{eq:quadratic_variation_expansion} can be bounded by $2\xi^2 (L_k + S_{k + 1})^2$, by Assumptions \ref{as:regularity} and \ref{as:lambda_k}. 
The second term in \eqref{eq:quadratic_variation_expansion} can be bounded as follows: 
\begin{align*}
& \E \left[\left. \left( \int_r^{r + \xi} b^{(k)}(x) \mathbf{1}(x \notin \cT) dN(x) \right)^2 \right \vert \cF_r \right]  \stackrel{(a)}{\le} J_k^2 \E \left[ \left. \left( \int_r^{r + \xi} \mathbf{1}(x \notin \cT) dN(x) \right)^2 \right \vert \cF_r \right] \\
& \stackrel{(b)}{\le} J_k^2 \E \left[ \left. \int_r^{r + \xi} \mathbf{1}(x \notin \cT) dN(x) \right \vert \cF_r \right] + J_k^2 \E \left[ \left. ( N(r + \xi) - N(r) )^2 \mathbf{1}( N(r + \xi) - N(r) \ge 2) \right \vert \cF_r \right] \\
& \stackrel{(c)}{\le} \xi J_k^2 \lambda^{(1)}(r) + o(\xi) + J_k^2 \E [ (N(r + \xi) - N(r))^4 \vert \cF_r ]^{1/2} \p ( N(r + \xi) - N(r) \ge 2 \vert \cF_r)^{1/2} \\
& \stackrel{(d)}{\le} \xi J_k^2 S_1 + J_k^2 (\xi R)^{3/2} + o(\xi) = \xi J_k^2 S_1 + o(\xi).
\end{align*}
Above, the inequality $(a)$ holds since $|b^{(k)}(x) | \mathbf{1}(x \notin \cT) \le J_k$ by Assumptions \ref{as:regularity} and \ref{as:lambda_k}. The inequality $(b)$ follows since, on the event $\{N(r + \xi) - N(r) \le 1\}$ we have that $\int_r^{r + \xi} \mathbf{1}( x \notin \cT) dN(x) \in \{0,1 \}$, hence the integral and its square are equal in this case. On the other hand. on the event $\{ N(r + \xi) - N(r) \ge 2 \}$, we can upper bound $\int_r^{r + \xi} \mathbf{1}(x \notin T) dN(x) \le N(r + \xi) - N(r)$. The inequality $(c)$ follows from the definition of $\lambda^{(1)}(\cdot)$ as well as the Cauchy-Schwartz inequality. The inequality $(d)$ holds since $| \lambda^{(1)}(r) | \le S_1$ by Assumption \ref{as:regularity}, and since $N(r + \xi) - N(r)$ can be stochastically bounded by a $\mathrm{Poi}(R \xi)$ random variable in light of Assumption \ref{as:regularity}.

Next, we study the third term in \eqref{eq:quadratic_variation_expansion}. Since $|c^{(k)}(x)| \mathbf{1}(x \notin \cT) \le J_k$ in light of Assumptions \ref{as:regularity} and \ref{as:lambda_k}, if $r < s$ it holds almost surely that 
\[
\Big( \sum_{x \in \cD \cap (r, r + \xi]} c^{(k)}(x) \mathbf{1}(x \notin \cT) \Big)^2 \le J_k^2 | \cD \cap (r, r + \xi] |^2 \le J_k^2 | \cD \cap (t, s]|^2.
\]
Importantly, the right hand side is equal to zero if $\cT \cap (t, s] = \emptyset$.

Putting everything together, we see that on the event $\{ \cT \cap (t, s] = \emptyset\}$, it holds for all $r \in (t, s]$ that $\langle Z_k \rangle_t^r \le 3 J_k^2 S_1 (r - t) \le 3 J_k^2 S_1 \delta$. Hence 
\begin{align*}
\p \left( \left \{ \sup_{r \in (t, s]} X_k(r) \ge x \right \} \cap \{ \cT \cap (t,s] = \emptyset \} \right) & \stackrel{(e)}{\le} \p \left( \exists r \ge t \text{ s.t. } Z_k(r) \ge x \text{ and } \langle Z_k \rangle_t^r \le 3 J_k^2 S_1 \delta\right) \\
& \stackrel{(f)}{\le} \exp \left( - \frac{x^2 / 2}{3 J_k^2 S_1 \delta + J_k x / 3} \right).
\end{align*}
Above, $(e)$ holds since, on the event $\{ \cT \cap (t, s] \}$, we have that $X_k(r) = Z_k(r)$ and $\langle Z_k \rangle_t^r \le 3 J_k^2 S_1 \delta$ for $r \in (t, s]$. The inequality $(f)$ follows from Lemma \ref{lemma:freedman}. Applying the same analysis to the martingale $-{X}_k$ shows that the same upper tail bound also holds for $-{X}_k$. Taking a union bound over the events corresponding to ${X}_k$ and $-{X}_k$ proves the lemma.
\end{proof}

We can now put everything together to prove the main result of this subsection.

\begin{proof}[Proof of Lemma \ref{lemma:polynomial_approximation_fixed_t}]
Let $t \in [0,T]$ and $s \in (t, t + 2 \delta)$.
For any $x \ge 0$, an application of Lemma \ref{lemma:poly_approx_deterministic} with $\delta$ replaced by $2 \delta$ implies that
\begin{align*}
 \max \left \{  |N(s) - \overline{N}_\ell(s, t) |,  |N(t) - \overline{N}_\ell(t, s) | \right \} \ge (2\delta)^{\ell + 1} S_{\ell + 1} +  x   &  \Rightarrow \sum_{k = 0}^\ell (2\delta)^k \sup_{r \in (t, s]} |X_k(r) | \ge x \\
&  \Rightarrow \exists k \in \{ 0, \ldots, \ell \} : \sup_{r \in (t, s]} |X_k(r) | \ge x.
\end{align*}
The implications above along with a union bound imply that
\begin{align*}
\p \left( \cA_{t, s}(x) \cap \{ \cT \cap (t, s] = \emptyset \} \right) & \le \sum_{k = 0}^\ell \p \left( \left \{ \sup_{r \in (t, s]} |X_k(r) | \ge x \right \} \cap \{ \cT \cap (t, s] = \emptyset \} \right) \\
& \le \sum_{k = 0}^\ell 2 \exp \left( - \frac{x^2 / 2}{6 J_k^2 S_1 \delta + J_k x / 3} \right) \\
& \le 2 (\ell + 1) \exp \left( - \frac{x^2 / 2}{6 J_{\le \ell}^2 S_1 \delta + J_{\le \ell} x / 3} \right).
\end{align*}
Above, the inequality on the second line follows from Lemma \ref{lemma:X_concentration}, with $\delta$ replaced by $2 \delta$.
\end{proof}

\subsection{Discretization results: Proof of Lemma \ref{lemma:discretization}}
\label{subsec:discretization}

We focus on proving \eqref{eq:discretization_bound_1}; the proof of \eqref{eq:discretization_bound_2} follows from similar arguments. 

Fix $\varepsilon > 0$ to be determined later, and recall the discretized set of time indices $\mathbb{T} : = [0,T] \cap \{ k \varepsilon : k \in \mathbb{Z}_{\ge 0} \}$.
Next, define the event 
\[
\cE : = \left \{\text{For each $r \in \mathbb{T}$, } | \cD \cap [r, r + \epsilon] | + (N(r + \epsilon) - N(r)) \le 1 \right \}.
\]
Assume that the event $\cE$ holds, and let $t \in [0,T]$.
Let $r \in \mathbb{T}$ satisfy $r \le t < r + \varepsilon$. By Assumption \ref{as:lambda_k}, at most one discontinuity point for $N(\cdot)$ or $\lambda^{(k)}(\cdot), k \ge 1$ can occur in the interval $[r, r + \varepsilon)$. If there is no discontinuity point in the interval, set $t' : = r$. Otherwise, if a single discontinuity point occurs in the interval at time $p \in (r, r + \varepsilon)$, choose $t' : = r$ if $p > t$, else $t' : = r + \varepsilon$. As a result of this construction, we have the useful property that the processes $N(\cdot)$ and $\lambda^{(k)}(\cdot), k \ge 1$ are continuous in the interval $(t' \land t, t' \lor t]$.
Similarly, we can find $s' \in \mathbb{T}$ such that $|s' - s| < \varepsilon$ such that $N(\cdot)$ and $\lambda^{(k)}(\cdot), k \ge 1$ are continuous in the interval $(s' \land s, s' \lor s]$.

Next, we relate $|N(s) - \overline{N}_\ell(s, t) |$ and $|N(s') - \overline{N}_\ell(s', t') |$. 
By the triangle inequality, we can bound
\begin{align}
& \left| |N(s) - \overline{N}_\ell(s, t) | - | N(s') - \overline{N}_\ell(s', t') | \right| \nonumber \\
& \hspace{2cm} \le | N(s) - N(s') | + | \overline{N}_\ell(s,t) - \overline{N}_\ell(s', t') | \nonumber \\
\label{eq:discretization_error}
& \hspace{2cm} \stackrel{}{\le} |N(s) - N(s') | + |N(t) - N(t') | + \sum_{k = 1}^\ell \frac{1}{k!}  | (s - t)^k \lambda^{(k)}(t) - (s'  - t')^k \lambda^{(k)}(t') |.
\end{align}
We proceed by bounding the terms on the right hand side of \eqref{eq:discretization_error}.
Recall that by the construction of $s'$, $N(\cdot)$ has no discontinuities in $(s' \land s, s' \lor s]$, hence $N(s) = N(s')$. The same argument shows that $N(t) = N(t')$, hence the first two terms in \eqref{eq:discretization_error} are equal to zero.

Next, we handle the summation on the right hand side of \eqref{eq:discretization_error}.
We can bound the $k$th term in the sum as follows:
\begin{align}
 \left| (s - t)^k \lambda^{(k)}(t) - (s' - t') \lambda^{(k)}(t') \right| & \stackrel{(a)}{\le} |s - t|^k \left| \lambda^{(k)}(t) - \lambda^{(k)}(t') \right| + \left| (s - t)^k - (s' - t')^k \right| \left| \lambda^{(k)}(t') \right| \nonumber \\
& \stackrel{(b)}{\le} \delta^k \left| \int_{t' \land t}^{t' \lor t} a^{(k)}(x) dx \right| + k |(s - s') - (t - t') | \left| \lambda^{(k)}(t') \right| \nonumber \\
\label{eq:discretization_triangle_inequality_bound}
& \stackrel{(c)}{\le} \delta^k \varepsilon L_k + 2k \varepsilon S_k.
\end{align}
Above, $(a)$ is due to the triangle inequality. 
The first term in inequality $(b)$ holds since 
$\lambda^{(k)}$ is continuous on $(t' \land t, t' \lor t]$, and we have used the representation of $\lambda^{(k)}$ in Assumption \ref{as:lambda_k}. The second term in inequality $(b)$ follows since $|s - t|, |s' - t'| < 1$ for $\delta < 1/2$ and $\varepsilon \le 1/4$, as well as the fact that $y \mapsto y^k$ is $k$-Lipschitz on $[0,1]$.
The inequality $(c)$ follows from Assumptions \ref{as:regularity} and \ref{as:lambda_k}.
Substituting the bound \eqref{eq:discretization_triangle_inequality_bound} into \eqref{eq:discretization_error} shows that
\begin{align*}
\left| |N(s) - \overline{N}_\ell(s,t) | - | N(s') - \overline{N}_\ell(s', t')| \right| & \le \sum_{k = 1}^\ell \frac{ ( \delta^k L_k + 2k S_k) \varepsilon}{k!} \\
& \le ( L_{\le \ell} + 2 S_{\le \ell} ) \sum_{k = 0}^\infty \frac{\varepsilon}{k!} \\
& \stackrel{(d)}{\le} 3 ( L_{\le \ell} + 2 S_{\le \ell} ) \varepsilon \\
& \stackrel{(e)}{\le} 1,
\end{align*}
where $(d)$ uses $\sum_{k = 0}^\infty 1/k! = e \le 3$, and $(e)$ follows from our assumptions on $\varepsilon$.

To prove the lemma, it remains to show that $\p ( \cE) \ge 1 - \eta$. To do so, we define 
\[
\varepsilon_0 := \min \left \{ \frac{1}{4}, \frac{1}{3 ( L_{\le \ell} + 6 S_{\le \ell})}, \frac{\sep(\cD)}{3}, \frac{\eta}{2T R^2}, \frac{\eta \sep(\cD)}{2TR} \right \},
\]
and assume that $\varepsilon \le \varepsilon_0$.
Additionally, we define two events: $\cE_1$ is the event where, for every $r \in \mathbb{T}$, $N(\cdot)$ jumps at most once in $[r, r + \varepsilon)$, and $\cE_2$ is the event where, for every $r \in \mathbb{T}$ satisfying $|\cD \cap [r, r + \varepsilon)| = 1$, $N(r + \varepsilon) = N(r)$. Notice that, if $\varepsilon \le \sep( \cD) / 3$, then $| \cD \cap [r, r + \varepsilon)| \le 1$ for all $r \in \mathbb{T}$, hence for such $\varepsilon$, we have that $\cE \supseteq \cE_1 \cap \cE_2$. 

We start by studying $\cE_1$. 
For a given $r \in \mathbb{T}$, the probability that $N(\cdot)$ jumps at least twice in $[r, r + \varepsilon)$ is at most $R^2 \epsilon^2$. Taking a union bound over the $T/ \varepsilon$ elements of $\mathbb{T}$ and using the definition of $\varepsilon$, we see that 
$\p ( \cE_1^c) \le T R^2 \varepsilon \le \eta/2$. Next, turning to the event $\cE_2$, notice that the probability that $N(\cdot)$ jumps in the interval $[r, r + \varepsilon)$ is at most $R \varepsilon$. Taking a union bound over $r \in \mathbb{T}$ satisfying $|\cD \cap [r, r + \varepsilon)| \ge 1$ shows that $\p ( \cE_2^c) \le | \cD | R \varepsilon \le TR\varepsilon / \sep(\cD) \le \eta / 2$. Putting everything together, a union bound shows that $\p ( \cE^c) \le \p ( \cE_1^c ) + \p ( \cE_2^c) \le \eta$, provided $\varepsilon \le \varepsilon_0$. 
We conclude by noting that if $L_{\le \ell}, S_{\le \ell}, T, R \ge 1$, then a specific valid choice for $\varepsilon$ is given by \eqref{eq:discretization_epsilon}.
\qed

\section{Higher-order derivatives: Proofs of Lemma \ref{lemma:derivative_characterization} and Theorem \ref{thm:derivative_thresholding}}
\label{sec:higher_derivatives}

In the proof of the Lemma \ref{lemma:derivative_characterization}, the following result will be useful. 

\begin{lem}
\label{lemma:polynomial_derivative}
Let $f$ be a polynomial of degree at most $\ell$. Then $\derivative f(t) = 0$ for all $t$.
\end{lem}

\begin{proof}
Since the discrete derivative is a linear operator, it suffices to consider the case $f(t) = (t + \ell \delta)^k$ for some $k \in \{0, \ldots, \ell \}$. We can then write
\begin{align}
\derivative f(t) & = \sum_{j = 0}^{\ell + 1} (-1)^{\ell + 1 - j} {\ell + 1 \choose j} \left (t + j \delta \right )^k \nonumber \\
& = \sum_{j= 0}^{\ell + 1} (-1)^{\ell + 1 - j} {\ell + 1 \choose j} \sum_{i = 0}^k {k \choose i} t^{k-i} j^i \delta^i \nonumber \\
\label{eq:derivative_expansion_polynomial}
& = \sum_{i = 0}^k {k \choose i} t^{k-i} \delta^i  \bigg( \sum_{j = 0}^{\ell + 1} (-1)^{\ell + 1 - j} {\ell + 1 \choose j} j^i \bigg).
\end{align}
It can be recognized that the inner summation over $j$ in \eqref{eq:derivative_expansion_polynomial} is a \emph{Stirling number of the second kind}, which counts the number of ways in which $i$ items can be divided into $\ell + 1$ non-empty subsets \cite[Ch. 24.1.4]{abramovitz_stegun}. When $i \le \ell$ this number is zero, which proves the claim.
\end{proof}

We are now ready to prove the lemma. 

\begin{proof}[Proof of Lemma \ref{lemma:derivative_characterization}]
Throughout this proof, we assume that the event $\Eapprox(T, \err, 3 \ell \delta, \ell)$ holds (we defer the reader to Definition \ref{def:Eapprox} for more details on this event). For brevity, we write $\Eapprox$ in this proof.

We start by making a few useful observations. Due to the definition of the discrete derivative, we have that 
\[
 N(t + \delta) - \derivative = \sum_{j = 0}^\ell (-1)^{ \ell - j} {\ell + 1 \choose j} N(t + (j - \ell) \delta).
\]
Additionally, since $\overline{N}_\ell(\cdot, t^-)$ is a polynomial of degree $\ell$, Lemma \ref{lemma:polynomial_derivative} implies that
\[
\overline{N}_\ell\left( t + \delta, t^- \right) = \sum_{j = 0}^\ell (-1)^{\ell - j} {\ell + 1 \choose j} \overline{N}_\ell \left( t + \left( j - \ell \right)\delta , t^- \right).
\]
If $\cT \cap (t - 3 \ell \delta, t) = \emptyset$, it follows that 
\begin{align}
& \left | \derivative N(t) - (N(t + \delta) - \overline{N}_\ell(t + \delta, t^-) ) \right| \nonumber \\
& \hspace{2cm} = \left| \sum_{j = 0}^\ell (-1)^{\ell - j} {\ell + 1 \choose j} \left [ \overline{N}_\ell \left( t + \left( j- \ell \right) \delta, t^- \right) - {N} \left( t + \left( j - \ell \right) \delta \right) \right] \right| \nonumber \\
& \hspace{2cm} \le (\ell + 1) | N(t) - N(t^-) | + \sum_{j = 0}^{\ell - 1} {\ell + 1 \choose j} 3 \ell \delta \err \nonumber \\
\label{eq:predictive_error_bound}
& \hspace{2cm} \le 3 \ell 2^{\ell + 1} \max \{ \delta \err , 1 \}.
\end{align}
Above, the first inequality holds on the event $\Eapprox$, and the second follows since the jumps of $N(\cdot)$ are at most 1. Item \ref{item:predictive_error} immediately follows.
Additionally, if $\cT \cap (t, t + 3 \ell \delta) = \emptyset$, then $|N(t + \delta) - \overline{N}_\ell(t + \delta, t)| \le \delta \err$ on the event $\Eapprox$. This, combined with \eqref{eq:predictive_error_bound}, proves Item \ref{item:predictive_error_polynomial}.

It remains to prove Item \ref{item:small_derivative}. Suppose that $\cT \cap (t - \ell \delta, t + 2\delta) = \emptyset$. It holds that
\begin{align*}
\left| \derivative N(t) \right| & = \left| \derivative N(t) - \sum_{j = 0}^{\ell + 1} (-1)^{\ell + 1 - j} {\ell + 1 \choose j} \overline{N}_\ell(t + (j - \ell) \delta, t - \ell \delta) \right | \\
& = \left| \sum_{j = 0}^{\ell + 1} (-1)^{\ell + 1 - j} {\ell + 1 \choose j} \left[ N \left (t + \left( j - \ell \right) \delta \right ) - \overline{N}_\ell \left (t + \left (j - \ell \right) \delta, t - \ell \delta \right) \right]  \right | \\
& \le \sum_{j = 0}^{\ell + 1} {\ell + 1 \choose j} 3 \ell  \delta \err \le 3 \ell 2^{\ell + 1 } \delta \err.
\end{align*}
Above, the equality on the first line follows from Lemma \ref{lemma:polynomial_derivative}, and the first inequality on the last line holds on the event $\Eapprox$.
\end{proof}

Finally, we prove Theorem \ref{thm:derivative_thresholding}.

\begin{proof}[Proof of Theorem \ref{thm:derivative_thresholding}]
Throughout the proof, we will assume that the event $\Eapprox \cap \Eabrupt$ holds.
Additionally, assume that 
\begin{equation}
\label{eq:delta_conditions}
A \ge 12 c_2 \err ~\text{ and }~ \frac{12 c_2}{A} \le \delta \le \min \left \{ \frac{1}{3 \rho}, \frac{\kappa}{3 \ell} \right \}.
\end{equation}
We start by showing that the first inclusion in \eqref{eq:estimating_T} holds. 
Since $3 \ell \delta \le \spacing$, it holds that $\cT \cap (t - 3 \ell \delta, t) = \emptyset$ and $\cT \cap (t, t + 3 \ell \delta) = \emptyset$ for $t \in \cT$. It follows that

\begin{align*}
\left| \derivative N(t) \right | & \stackrel{(a)}{\ge} \left| \overline{N}_\ell \left (t + \delta, t \right) - \overline{N}_\ell \left(t + \delta, t^- \right) \right| - c_2 \max \{ \delta  \err, 1 \} \\
& \ge \delta \left| \lambda^{(1)}(t) - \lambda^{(1)}(t^-) \right| - |N(t) - N(t^-) | - \sum_{k = 2}^\ell \frac{\delta^k}{k!}  \left| \lambda^{(k)}(t) - \lambda^{(k)}(t^-) \right| \\
& \hspace{1cm} - c_2 \max \{ \delta \err, 1 \} \\
& \stackrel{(b)}{\ge} \delta | \lambda^{(1)}(t) - \lambda^{(1)}(t^-) | \left( 1 - \delta \rho \right) - 1 - c_2 \max \{ \delta \err, 1 \} \\
& \stackrel{(c)}{\ge} \frac{2\delta A}{3 } - 2 c_2 \max \{ \delta \err, 1 \} \\
& \stackrel{(d)}{\ge} \frac{\delta A}{2}.
\end{align*}
Above, $(a)$ is due to Item \ref{item:predictive_error_polynomial} of Lemma \ref{lemma:derivative_characterization} and $(b)$ follows from the characterization of $\cT$ in $\Eabrupt$. Inequalities $(c)$ and $(d)$ hold if \eqref{eq:delta_conditions} is satisfied.

Next, we prove that the second inclusion in \eqref{eq:estimating_T} holds. As a shorthand, denote 
\[
\cT' : = \left \{ t \in [0,T]: \exists t' \in \cT \text{ with } |t - t' | \le  (\ell + 1) \delta \right \}.
\]
If $t \notin \cT'$ then in particular $\cT \cap (t - \ell \delta, t + 2\delta) = \emptyset$. 
Hence, due to Item \ref{item:small_derivative} of Lemma \ref{lemma:derivative_characterization} and \eqref{eq:delta_conditions},
\[
\left| \derivative N(t) \right| \le c_2 \max \{ \delta \err, 1 \} \le \frac{\delta A}{12}.
\]
It follows that $(\cT')^c \subseteq \{ t \in [0,T] : | \derivative N(t) | \le \delta A /12 \}$, which proves \eqref{eq:estimating_T}.

Finally, we show that the choice $\delta = S_{\le \ell + 1}^{- (\ell + 1) / (2 \ell + 1)}$ satisfies the conditions in \eqref{eq:delta_conditions}. 
We first note that for this choice of $\delta$, we have from \eqref{eq:error_bound} that 
\[
\err \le \delta^\ell S_{\le \ell + 1}^{1 + o(1)} + \sqrt{ \frac{S_{\le \ell + 1}^{1 + o(1)}}{\delta}} = S_{\le \ell + 1}^{(\ell + 1)/(2 \ell + 1) + o(1)}.
\]
Since $A \ge S_{\le \ell + 1}^\theta$ for some $\theta >  (\ell + 1) / (2 \ell + 1)$, the first inequality in \eqref{eq:delta_conditions} holds. Additionally, since $\min \{ 1/\rho, \kappa \} \ge S_{\le \ell + 1}^{-o(1)}$ by \eqref{eq:thm_parameter_conditions} and $\delta \ge S_{\le \ell + 1}^{- \theta - o(1)}$, the second statement in \eqref{eq:delta_conditions} holds as well.
\end{proof}

\section{Proofs for smooth + jump rate functions}\label{sec:smoothjump}

\subsection{Impossibility of detection: Proof of Proposition \ref{prop:impossibility}}

Define the rate functions
\begin{align*}
\Lambda_0(t) & : = B \\
\Lambda_1(t) & : = B + A e^{-(t - 1)} \mathbf{1}(t \ge 1).
\end{align*}
Additionally, let $\p_0$ and $\p_1$ denote the distribution over event times induced by the rate functions $\Lambda_0$ and $\Lambda_1$.
To prove the theorem, we show that $\mathrm{TV}(\p_0, \p_1) \to 0$ as $x_0 \to \infty$. In light of Pinsker's inequality, it suffices to prove that the Kullback-Liebler divergence between these distributions tends to zero. To this end, we have that
\begin{align}
\log \frac{d \p_0}{d \p_1}(  \cH ) &  = \sum_{\tau \in \cH} \log \frac{\Lambda_0(\tau)}{\Lambda_1(\tau)} - \int_0^\infty \Lambda_0(t) - \Lambda_1(t) dt \nonumber \\
& = - \sum_{\tau \in \cH : \tau \ge 1} \log \left( 1 + \frac{A}{B} e^{- (\tau - 1)} \right) + \int_1^\infty A e^{- (t - 1)} dt \nonumber \\
\label{eq:log_likelihood_bound}
& \le \sum_{\tau \in \cH : \tau \ge 1} \left( - \frac{A}{B} e^{- (\tau - 1)} + \frac{A^2}{B^2} e^{-2 (\tau - 1)} \right) + B,
\end{align}
where, in the final inequality above, we have used the bound $\log(1 + a) \ge a - a^2$ for $|a|$ sufficiently small.

Let us now label the event times occurring after $t = 1$ in chronological order, so that $\cH \cap [1, \infty) = \{ \tau_k \}_{k \ge 1}$, where $\tau_k < \tau_{k + 1}$ for all $i \ge 1$. Under the measure $\p_0$, we have that $\tau_k -1 \stackrel{d}{=} \mathrm{Gamma}(1, x_0)$, which implies
\begin{equation}
\label{eq:event_mgf}
\E_0 \left [ e^{- (\tau_k - 1)} \right] = \left( 1 + \frac{1}{B}\right)^{-k} \text{ and } \E_0 \left [ e^{- 2 ( \tau_k- 1)} \right ] = \left( 1 + \frac{2}{B} \right)^{-k}.
\end{equation}
Equations \eqref{eq:log_likelihood_bound} and \eqref{eq:event_mgf} show that
\begin{align*}
D_{\mathrm{KL}}(\p_0 \| \p_1 ) & \le \sum_{k \ge 1} \left( - \frac{A}{B} \left( 1 + \frac{1}{B} \right)^{-k} + \frac{A^2}{B^2} \left( 1 + \frac{2}{B} \right)^{-k} \right) + A = \frac{A^2}{2 B}.
\end{align*}
Under the assumptions in the theorem, we have that $D_{\mathrm{KL}}(\p_0 \| \p_1) \to 0$. This result, along with Pinsker's inequality, shows that $\mathrm{TV}(\p_0, \p_1) \to 0$ as well. \qed

\subsection{Detecting multiple change-points: Proof of Theorem \ref{thm:smooth-plus-jump}}
Recall from Section~\ref{subsec:smooth-jump-results} that we study a rate function of the form 
\begin{equation*}
\Lambda(t) = \sum_{i \ge 1} A_i x_i (t - t_i) \mathbf{1}(t \ge t_i),
\end{equation*}
where $t_1 < t_2 < \cdots < t_m$ correspond to the jump times, and the $x_i$'s are smooth functions satisfying $x_i(0) > 0$ (see Definition~\ref{def:smooth_plus_jump}). 
Further recall that, as per Assumption \ref{as:smooth_jump_jump_sizes}, $A_j \ge ( \sum_{i \ge 1} A_i )^\theta$ for some $\theta > 0$ or $A_j = ( \sum_{i \ge 1} A_i )^{o(1)}$.

\begin{proof}[Proof of Theorem \ref{thm:smooth-plus-jump}]
To check that Theorem \ref{thm:derivative_thresholding} can be applied, we first need to compute the $\lambda^{(k)}$'s. We prove by induction that for $k \ge 1$, it holds that 
\begin{equation}
\label{eq:lambda_smooth_plus_jump}
\lambda^{(k)}(t) = \sum_{i \ge 1} A_i x_i^{(k - 1)}(t - t_i) \mathbf{1}(t \ge t_i).
\end{equation}
We begin with the base case $k = 1$.
Since $\cT$ is a deterministic set, for every $t \ge 0$ we can find a deterministic $\varepsilon'$ so that $\cT \cap (t, t + \varepsilon'] = \emptyset$. Consequently, we have that
\begin{align*}
\lambda^{(1)}(t) & = \lim_{\varepsilon \to 0^+} \frac{1}{\varepsilon} \E \left[ \left. \int_t^{t + \varepsilon} \mathbf{1}(x \notin \cT) dN(x) \right \vert \cF_t \right] = \lim_{\varepsilon \to 0^+} \frac{1}{\varepsilon} \E [ N(t + \varepsilon) - N(t) \vert \cF_t] = \Lambda(t).
\end{align*}
We now prove the inductive case. For $k \ge 1$, we have that
\begin{align*}
\lambda^{(k+1)}(t) & = \lim_{\varepsilon \to 0^+} \frac{1}{\varepsilon} \int_t^{t + \varepsilon} \mathbf{1}(x \notin \cT) d \lambda^{(k)}(x) \\
& = \lim_{\varepsilon \to 0^+} \frac{\lambda^{(k)}(t + \varepsilon) - \lambda^{(k)}(t)}{\varepsilon} \\
& = \sum_{i \ge 1} A_i x_i^{(k)}(t - t_i) \mathbf{1}(t \ge t_i),
\end{align*}
where, in the first equality above, we have removed the expectation operator since $\lambda^{(k)}$ is deterministic for $k \ge 1$, by the inductive hypothesis.

To apply Theorem \ref{thm:derivative_thresholding}, we will need to bound the the absolute value of the $\lambda^{(k)}$'s. To do so, it is useful to note that since there are finitely many $x_i$'s and they are smooth in $[0,T]$, it holds for all integers $k \ge 1$ that there exists a constant $D_k$ such that for all $i \in \{1, \ldots, m \}$ 
\[
\sup_{t \in [0,T]} \left| x_i^{(k)}(t) \right | \le D_k.
\]
In light of the display above and the representation \eqref{eq:lambda_smooth_plus_jump}, it can be readily checked that Assumptions \ref{as:regularity} and \ref{as:lambda_k} are satisfied if $R, \mathbf{S}, \mathbf{J}, \mathbf{L}$ are set as as follows:
\begin{itemize}
    \item $R : = \sum_{i \ge 1} A_i$,
    \item $S_k : = D_{k - 1} \sum_{i \ge 1} A_i$ for $k \ge 1$,
    \item $J_k = ( \sum_{i \ge 1} A_i )^{o(1)}$,
    \item $L_k = S_{k + 1}$ for $k \ge 0$.
\end{itemize}
Additionally, for a fixed $\ell \ge 1$, the event $\Eabrupt$ (introduced in Definition \ref{def:Eabrupt}) holds with probability 1 if we set
\begin{itemize}
    \item $A = (\sum_{i \ge 1} A_i)^\theta$,
    \item $\rho = D_{\le \ell - 1}$,
    \item $\spacing = \sep(\cD)$.
\end{itemize}
It is readily checked that the condition \eqref{eq:thm_parameter_conditions} is satisfied with the parameter choices above.
Applying Theorem \ref{thm:derivative_thresholding} shows that \eqref{eq:estimating_T} holds on the event $\Eapprox \cap \Eabrupt$, and furthermore, in light of Proposition \ref{prop:algorithm}, the algorithmic conclusions of Theorem \ref{thm:smooth-plus-jump} are satisfied.
Finally, we note that $\p ( \Eapprox \cap \Eabrupt) = \p ( \Eapprox ) = 1 - o(1)$ as $\sum_{i \ge 1} A_i \to \infty$.
\end{proof}

\section{Proofs for the SI process}\label{sec:SIprocess}

To apply our general change-point detection results described in Section~\ref{subsec:general_result} to prove Theorems~\ref{thm:detecting_high_degree} and \ref{thm:estimating_high_degrees}, 
we need to characterize two quantities: (1) upper bounds for the derivatives when there are no big jumps around, and (2) infinitesimal changes in higher-order derivatives when big jumps occur. 

To compute the $\lambda^{(k)}(t)$'s for the SI process, we introduce some definitions. 

\begin{defn}[Marked trees]\label{def:marked-trees}
A marked tree is a tuple $(T,\mu)$ where $T$ is a rooted tree and $\mu : V(T) \to \{0,1\}$ a marking of the vertices. We inductively define a family of marked trees $\cM = \bigcup_{k\geq 1} \cM_k$, where $\cM_k$ is the set of marked trees with $k$ edges. Let $\cM_1 = \{ (T_1,\mu_1) \}$, where $T_1$ is the 1-edge path, and $\mu$ is such that the root $v$ has $\mu(v) = 0$ and its neighbor $u$ has $\mu(u) = 1$. Then, given $\cM_{k-1}$, we construct $\cM_k$ as follows: for each marked tree in $\cM_{k-1}$, add to $\cM_{k}$ the marked trees obtained by performing exactly one of the following operations: 
\begin{enumerate}
    \item adding a node $u$ with $\mu(u) = 1$ as a neighbor of the root, 
    \item picking one of the $1$-marked neighbors of the root, say $u$, setting it to now be the root, with $\mu(u) = 0$, and then attaching a new $1$-marked leaf to it.
\end{enumerate}
If there are duplicates (that is, distinct elements of $\cM_{k -1}$ lead to the same marked tree through the operations above), then we only add one copy to the set $\cM_k$.
\end{defn}

\begin{figure}[hbtp]
    \centering
    \input{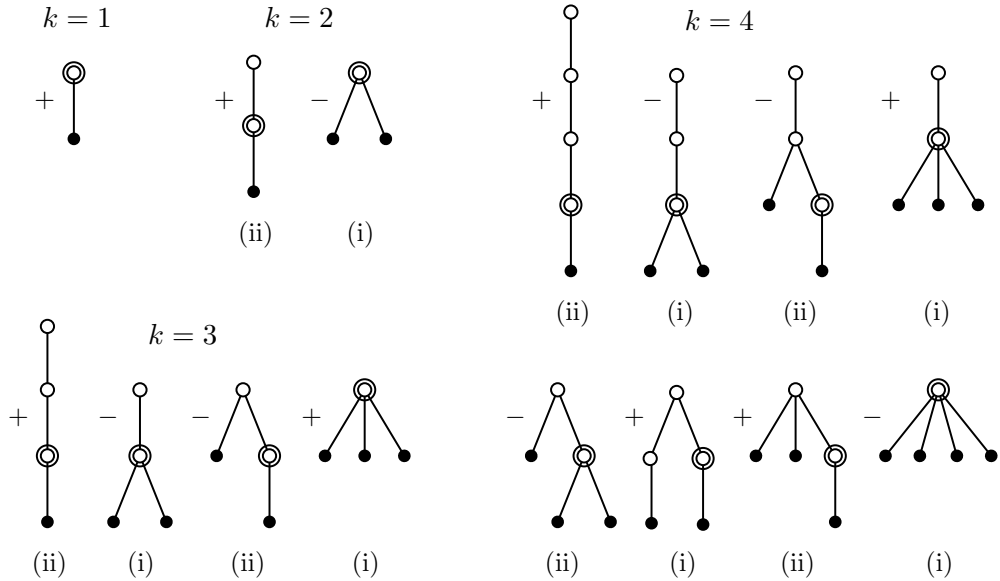}
    \caption{Diagrams of the trees in $\cM_k$ for $k=1,2,3,4$, where the root is doubly circled, and white and black vertices correspond respectively to 0 and 1 marked vertices. For each marked tree $(T,\mu)$, the $+/-$ sign corresponds to its corresponding $(-1)^{n_1(T,\mu)+1}$ term from Lemma~\ref{lemma:si_higher_derivative}, and the (i)/(ii) corresponds to the operation via which it was obtained.}
    \label{fig:enter-label}
\end{figure}

The following lemma details some useful properties of marked trees. 

\begin{lem}
\label{lemma:marked_tree_properties}
For all $k \in \N$, the following properties hold:
\begin{enumerate}
    \item For all $(T, \mu) \in \cM_k$, the root has at least one 1-marked neighbor.
    \item For all $(T, \mu) \in \cM_k$ and all $u \in T$, if $\mu(u) = 1$ then $u$ is a leaf.
    \item Let $k \ge 2$. Each marked tree in $\cM_{k}$ can be constructed from exactly one marked tree in $\cM_{k -1}$.
    \item $|\cM_k| = 2^{k - 1}$.
\end{enumerate}
\end{lem}
\begin{proof}
To prove the first property, note that it holds for $(T_1, \mu_1) \in \cM_1$, and for the inductive step, neither operation (i) nor (ii) can create any non-leaf 1-marked node. Similarly, the second property is satisfied for $(T_1, \mu_1)$, and operations (i) and (ii) ensure that the root has at least one 1-marked neighbor.

The third property is proved by way of contradiction. Assume that there are two distinct marked trees $(T, \mu), (T', \mu') \in \cM_{k - 1}$ that generate the same marked tree $(T'', \mu'')$ on $k$ edges. Since the operations (i) and (ii) are reversible, we must have that $(T'', \mu'')$ is obtained by applying different operations to $(T, \mu)$ and $(T', \mu')$. However, applying (i) creates a marked tree with at least two 1-marked neighbors of the root, whereas operation (ii) creates a marked tree with exactly one 1-marked neighbor of the root. Thus operations (i) and (ii) cannot lead to the same tree, which is a contradiction. This proves the third property.

We now prove the fourth property. In light of the second property, all potential marked trees created by operation (ii) are isomorphic, so each tree in $\cM_k$ creates two distinct marked trees via operations (i) and (ii), up to isomorphism. Furthermore, the third property implies that each sequence of operations leads to unique trees. It immediately follows that $| \cM_k | = 2^{k - 1}$, as each marked tree in $\cM_k$ is formed from $k - 1$ operations.
\end{proof}

\begin{defn}[Marked subtree counts]
Let $t \ge 0$, and let $v \in \partial \cI(t) := \{ v \in V \setminus \cascade(t)$ and $v$ has a neighbor in $\cascade(t)\}$. For a marked tree $(T, \mu)$, let $F_t(T, \mu, v)$ be the number of functions $f : V(T) \to V$ such that the following hold:
\begin{enumerate}
    \item $f$ is a homomorphism; that is $(f(u), f(w)) \in E \iff (u,w) \in E(T)$.
    \item If $\mu(u) = 0$ then $\deg(f(u)) \le d$.
    \item $\mu(u) = 1 \iff f(u) \in \cascade(t)$.
    \item The restriction $f:\{ u \in V(T) : \mu(u) = 0 \} \to V$ is injective (but it is possible that $f(u) = f(w)$ for distinct $u,w$ with $\mu(u) = \mu(w) = 1$).  
    \item $f(\treeroot(T)) = v$. 
\end{enumerate}
We also define $F_t(T, \mu) : = \sum_{v \in \partial \cascade(t):\, \deg(v) \le d} F_t (T, \mu, v)$.
\end{defn}

\begin{SCfigure}[.8]
    \centering
    \input{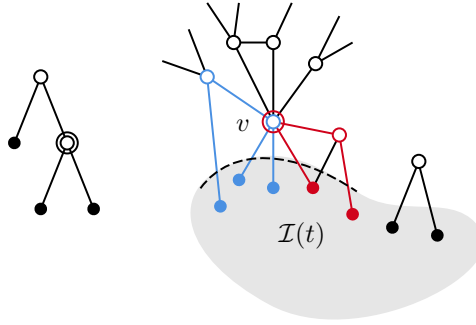}
    \vspace{-.25\baselineskip}
    \caption{In red and blue, two possible embeddings $f: V(T) \to V$ of the marked tree $(T,\mu)$ on the left hand side to some graph $(V,E)$, with $f$ satisfying the conditions of Definition~\ref{def:marked-trees}.}
    \label{fig:marked-subtrees-function}
\end{SCfigure}

In words, $F_t(T, \mu, v)$ can be viewed as the number of copies of $T$ rooted at $v$ where nodes marked with 1 correspond to infected vertices, and nodes marked with 0 correspond to uninfected vertices that have low degree. The quantity $F_t(T, \mu)$ represents the \emph{total} number of copies of $T$, without the restriction of being rooted at $v$.
 
Let the number of $1$-marked nodes in $(T,\mu)$ be $n_1(T,\mu)$.
We have the following result.

\begin{lem}
\label{lemma:si_higher_derivative}
Let $t \ge 0$ and let $k \ge 1$ be an integer.
It holds that
\begin{equation}\label{eq:si_higher_derivative}
    \lambda^{(k)}(t) = \sum_{(T, \mu) \in \cM_k} (-1)^{n_1(T,\mu) + 1} F_t (T,\mu).
\end{equation}
\end{lem}
\begin{proof}
We prove the claim by induction on $k$.
First consider $k=1$. Let $(T_1,\mu_1)$ be the single marked tree in $\cM_1$. Let $I(t) := |\cI(t)|$. From \eqref{eq:SI-evolution}, we have that 
\begin{align*}
\lambda^{(1)}(t) & = \lim_{\epsilon \to 0} \frac{\E [ ( I(t + \epsilon) - I(t)) \mathbf{1}( \cT \cap (t, t + \epsilon) = \emptyset) \vert \cF_t ]}{\epsilon} \\
& = \sum_{v \in \partial \cascade(t) :\, \deg(v) \le d} | \cN(v) \cap \cascade(t) | \\
& = F_t(T_1, \mu_1),
\end{align*}
as required for the base case. Note that $|\cN(v) \cap \cascade(t)| = \sum_{u \in \cN(v) \cap \cascade(t)} 1$ is the number of edges from $\cascade(t)$ to $v$. This contribution corresponds to computing the marked subtree count of the tree resulting from adding a $1$-marked neighbor to the root $v$. 

Now, assume \eqref{eq:si_higher_derivative} holds for $\lambda^{(k-1)}(t)$, and consider the change in each term (corresponding to a marked tree $(T,\mu)$) in the sum over $\cM_{k-1}$ as a vertex (the root $v$) becomes infected. 

On one hand, as $v$ becomes infected, the existing contribution to this term from $(T,\mu)$ is removed, and as in the base case, the sum $\sum_{v \not \in \cI(t)} |\cN(v) \cap \cI(t)|$ corresponds to adding a $1$-marked neighbor to the root. This is exactly the tree with $k$ edges obtained from $(T,\mu)$ by operation (i) in Definition~\ref{def:marked-trees}, say $(T',\mu')$. The resulting contribution to $\lambda^{(k)}(t)$ is $-(-1)^{n_1(T,\mu)+1} F_t(T',\mu') = (-1)^{n_1(T',\mu')+1} F_t(T',\mu')$. 

On the other hand, considering $v$ as infected, we now add the contribution of the ``new'' tree of the same shape as $(T,\mu)$ but where $v$ is thought of as one of the original $1$-marked neighbors of the root. Once again adding a $1$-marked neighbor to $v$ from the sum, this corresponds to the tree $(T'',\mu'')$ obtained from $(T,\mu)$ via (ii). Under (ii), the number of $1$-marked nodes is unchanged, so the contribution to $\lambda^{(k)}(t)$ is $(-1)^{n_1(T'',\mu'') + 1} F_t(T'',\mu'')$. 
\end{proof}

The next lemma lists a few properties of the conditional derivatives regarding their smoothness and jumps, which we then use to check the conditions needed to apply our main detection result.

\begin{lem}\label{lemma:SI-properties}
The following statements hold:
\begin{enumerate}
    \item \label{item:SI_smoothness}
    For any integer $k \ge 1$, $\sup_{t \ge 0}|\lambda^{(k)}(t)| \le n (2kd^k)^k$. 
    \item \label{item:SI_highdeg_jump}
    For every $v \in V$ satisfying $\deg(v) \ge D$, 
    \[
    \lambda^{(1)}(t_v) - \lambda^{(1)}(t_v^-) = \deg(v) - | \cN(v) \cap \cascade(t_v)|.
    \]
    \item \label{item:SI_generic_jump}
    For any integer $k \ge 1$,
    \[
    \max_{v \in V} \frac{| \lambda^{(k)}(t_v) - \lambda^{(k)}(t_v^-) |}{\deg(v)} \le (2kd^k)^{k+1} . 
    \]
\end{enumerate}
\end{lem}

\begin{proof}
We start by proving (\ref{item:SI_smoothness}. Fix a vertex $v \in \partial \cascade(s)$ satisfying $\deg(v) \le d$. Let $(V_k(v), E_k(v)) \subset (V,E)$ be the subgraph consisting of vertices in the $k$-hop neighborhood of $v$, where any high-degree vertex must be a leaf of the subgraph. Formally, $(u,w) \in E_k(v)$ if and only if the following two conditions hold:
\begin{itemize}
    \item $\dist(u,v), \dist(w,v) \le k$,
    \item and if $\deg(u) \geq D$, there exists no $w' \in V_k(v)$, $w' \neq w$, such that $(u,w') \in E_k(v)$;
\end{itemize}
let $V_k(v)$ be the vertex set associated to $E_k(v)$. We can bound $|E_k(v)|$ and $|V_k(v)|$ as follows. For each edge $(u,w) \in E_k(v)$, there exists a path from $v$ to either $u$ or $w$ (the last edge of the path is $(u,w)$) of length at most $k$ which passes through low-degree vertices, possibly except for the endpoint. The number of such paths is at most $\sum_{i = 1}^k d^i \le k d^k$, which bounds $|E_k(v)|$ from above. We also have $|V_k(v)| \leq 2|E_k(v)|$. 

For each $(T,\mu) \in \cM_k$, the functions counted by $F_t(T,\mu,v)$ must map $V(T)$ to some subset of $V_k(v)$. Therefore,
\begin{equation}
\label{eq:sum_Ft_bound}
\sum_{(T, \mu) \in \cM_k} F_t(T, \mu,v)  \le 2^{k-1} |E_k(v)|^k  \le (2kd^k)^k.
\end{equation}
Summing over $v \in \partial \cI(t) :\, \deg(v) \leq d$, we obtain (\ref{item:SI_smoothness}.

Next, we characterize the jump sizes. We have from the definition \eqref{eq:def-lambda-1} that 
\begin{align*}
\lambda^{(1)}(t) - \lambda^{(1)}(t^-) & = \sum_{v \in \partial \cascade(t) :\, \deg(v) \le d } | \cN(v) \cap \cascade(t) | \ - \sum_{v \in \partial \cascade(t^-) :\, \deg(v) \le d} | \cN(v) \cap \cascade(t^-) |.
\end{align*}
Suppose that $t_u = t$ for a high-degree vertex $u$. For $v \in \cN(u) \setminus \cascade(t)$, it holds that $|\cN(v) \cap \cascade(t) | - | \cN(v) \cap \cascade(t^-) | = 1$; for $v \notin \cascade(t) \cup \{ u \}$, $| \cN(v) \cap \cascade(t) | = | \cN(v) \cap \cascade(t^-) |$. The contribution of $| \cN(u) \cap \cascade(t)|$ is not relevant here, since $u$ is high-degree. As a result, we have that
\[
\lambda^{(1)}(t) - \lambda^{(1)}(t^-) = | \cN(u) \setminus \cascade(t) | = \deg(u) - | \cN(u) \cap \cascade(t) |,
\]
which proves (\ref{item:SI_highdeg_jump}. 

Next, we bound the jump sizes for the general case $k \ge 1$. Suppose that $t_u = t$. We can bound
\begin{align*}
\left| \lambda^{(k)}(t) - \lambda^{(k)}(t^-) \right| &  \le \sum_{(T, \mu) \in \cM_k} \left| F_t(T, \mu) - F_{t^-}(T, \mu) \right| \\
& = \sum_{(T, \mu) \in \cM_k} \sum_{\substack{v \in \partial \cascade (t) :\, \deg(v) \le d \\ u \in V_k(v)}} \left| F_t(T, \mu, v) - F_{t^-}(T, \mu, v) \right| \\
& \le \sum_{\substack{v \in \partial \cascade (t) :\, \deg(v) \le d \\ u \in V_k(v)}} \sum_{(T, \mu) \in \cM_k} \left( F_t(T, \mu, v) + F_{t^-}(T, \mu, v) \right) 
\end{align*}
where we used the triangle inequality, and the important fact that for any $(T,\mu) \in \cM_k$, if $u \not \in V_k(v)$, then $F_t(T,\mu,v) = F_{t^-}(T,\mu,v)$.

Now, 
note that in both cases $\deg(u) \leq d$ and $\deg(u) \geq D$, we have that
there are at most $\deg(u)kd^k$ low-degree vertices $v$ such that $u \in V_k(v)$. Therefore, again using \eqref{eq:sum_Ft_bound}, we have
\begin{equation*}
\left| \lambda^{(k)}(t) - \lambda^{(k)}(t^-) \right| 
\le 2 \deg(u) kd^k (2k)^k d^{k^2} = \deg(u) (2kd^k)^{k+1},
\end{equation*}
which proves (\ref{item:SI_generic_jump}.
\end{proof}

\subsection{Detecting high-degree vertices}
\label{subsec:SI-detection-proofs}
We can now prove that the infection times of high-degree vertices can be detected with high probability.

\begin{proof}[Proof of Theorem~\ref{thm:detecting_high_degree}]
Fix $\ell \geq 1$ as in the theorem statement, which satisfies $\alpha > (\ell+1)/(2\ell+1)$. 
We wish to apply the framework of Theorem~\ref{thm:derivative_thresholding}. Using Lemma~\ref{lemma:SI-properties}(\ref{item:SI_smoothness} and (\ref{item:SI_generic_jump}, it can be checked that Assumptions~\ref{as:regularity} and \ref{as:lambda_k} are satisfied are satisfied when we set $R$, $\mathbf S$, $\mathbf J$, $\mathbf L$ as follows: \[ R := n ~,~ S_k := n (2kd^k)^k ~,~ J_k := (2kd^k)^{k+1} ~\text{ and }~ L_k = S_{k+1}.\] 

Now, we must ensure that the event $\Eabrupt$ (c.f.\ Definition~\ref{def:Eabrupt}) holds with probability $1-o(1)$. First, we have no external events so $\cD = \emptyset$. 
Next, we want to show that $\kappa \geq 3\ell\delta$ with probability $1-o(1)$. (Note that in Definition~\ref{def:Eabrupt}, we required $1/\kappa \le S_{\leq \ell+1}^{o(1)}$ for simplicity; actually, we can see from \eqref{eq:delta_conditions} that $\kappa \geq 3\ell\delta$ is sufficient.) We set $\delta = S_{\leq \ell+1}^{-1/(2\ell+1)} = \Theta( n^{-1/(2\ell+1)})$. Let $u,v \in \highdeg(G)$. Then, as any path from $u$ to $v$ must contain at least $r > (2\ell+1)$ edges,
\[
\p\{ |t_u - t_v| < 3\ell\delta \} \leq (1+o(1)) n d^{r-1} (3\ell\delta)^{r} \leq (1+o(1)) d^{r-1} (3\ell)^r  n^{1 -r/(2\ell+1)} = o(1).
\] 
Union bounding over all pairs of $u, v \in \highdeg(G)$ gives us that $\kappa \geq 3\ell\delta$ with high probability. 

Finally, from \cite[Lemma~3.7]{mossel2024finding}, we have that with probability $1-o(1)$, $|\cN(v) \cap \cI(t_v)| \leq 2d\sqrt{n}\log^2 n$ for all $v \in V$. Therefore, using this and Lemma~\ref{lemma:SI-properties}(\ref{item:SI_highdeg_jump}, for a fixed $\ell \geq 1$,  
\begin{equation}\label{eq:SI-A-rho-settings}
A := D - 2d\sqrt{n}\log^2 n 
~\text{ and }~ \rho := 2(2\ell d^\ell)^{\ell+1},     
\end{equation}
then $\Eabrupt$ holds with probability $1-o(1)$. 

Finally, the time taken to infect all vertices can be bounded above by 
$\p\{ \max_{v \in V} t_v \geq n^2 \} \leq n^{-2} \E\{ \max_{v \in V} t_v \} = o(1)$.
Using this, we can check that \eqref{eq:thm_parameter_conditions} holds for the above settings of parameters with high probability.
Thus, using again that $\delta = \Theta( n^{-1/(2\ell+1)})$, we have that if $D = \Omega(n^\alpha)$ such that $A \geq S_{\ell+1}^\alpha$, then, Theorem~\ref{thm:derivative_thresholding} yields that with probability $1 - o(1)$, 
\begin{equation}\label{eq:cor-SI-inclusions}
    \cT \subseteq \left\{ t \in [0,\max_{v\in V} t_v]: \left| \frac{1}{\delta} \Delta_{\delta}^{(\ell+1)} \cI(t) \right| \geq \frac{A}{2} \right\} \subseteq \left\{ t \in [0,\max_{v\in V} t_v] : \exists t'\in\cT \text{ with } |t-t'| \leq (\ell + 1) \delta \right\},
\end{equation}
as required.
\end{proof}

We conclude this section by proving the corresponding result on impossibility of detection for $\alpha < 1/2$. 

\begin{proof}[Proof of Theorem \ref{thm:high_degree_detection_impossibility}]
As a shorthand, we write $\cG$ instead of $\cG(n,m,p,d,D)$. 
It was shown in \cite{mossel2024finding} that there exist two distributions over $n$-vertex graphs in $\cG$, denoted by $\mu_\emptyset$ and $\mu_v$ (where $v$ is a designated vertex in $G$), such that the following properties hold (see \cite[Appendix D]{mossel2024finding}):
\begin{itemize}
    \item For all graphs in the support of $\mu_\emptyset$, all vertices have degree at most $d$, i.e., there are no high-degree vertices.
    \item For all graphs in the support of $\mu_v$, vertex $v$ has degree at least $D$, and all others have degree at most $d$, i.e., $v$ is the unique high-degree vertex.
    \item Letting $P_\emptyset$ and $P_v$ be the induced distributions over infection times from a particular initial vertex $v_0$, it holds that $\mathrm{TV}(P_\emptyset, P_v) = o(1)$ as $n \to \infty$; see \cite[Lemma 29]{mossel2024finding}.
\end{itemize}
The last property is equivalent to the claimed result, which completes the proof. 
\end{proof}

\subsection{Estimating high-degree vertices}
In this section, we use Theorem~\ref{thm:detecting_high_degree} to recover the high-degree vertices from cascade traces, proving Theorem~\ref{thm:estimating_high_degrees}.
First, for each $k \in [K]$, denote the infection time of vertex $v$ in cascade $k$ as $t_v^{(k)}$. 
The algorithm we consider is as follows. 

\begin{alg}
Let $A$ be as in \eqref{eq:SI-A-rho-settings}, and for each $k \in [K]$, let 
\[
 \Hat{\cT}^{(k)} := \left\{ t \in [0,\max_{v\in V} t_v]: \left| \frac{1}{\delta} \Delta_{\delta}^{(\ell+1)} \cI(t) \right| \geq \frac{A}{2} \right\} \subset \R.
\]
be the set of abrupt change times detected from the $(\ell+1)$-th order discrete derivative estimator, and let $V^{(k)} := \{v \in V: t_v^{(k)} \in \Hat{\cT}^{(k)}\}$ be the corresponding infected vertices. Return $\bigcap_{k=1}^K V^{(k)}$.
\end{alg}

\begin{proof}[Proof of Theorem~\ref{thm:estimating_high_degrees}
]
We wish to show that $\highdeg(G) = \bigcap_{k=1}^K V^{(k)}$ with probability $1-o(1)$.

Theorem~\ref{thm:detecting_high_degree} above holds for each of the $K$ independent cascades. Therefore, from the first inclusion in \eqref{eq:cor-SI-inclusions}, we have that $\highdeg(G) \subseteq \bigcap_{k=1}^K V^{(k)}$ with high probability, as required. 

We must now show that $\bigcap_{k=1}^K V^{(k)}$ contains no false positives with high probability. That is, there exist no low degree vertices that are infected within $(\ell+1)\delta$ of a high-degree vertex in each of the $K$ cascades. Let $u \in V$ low degree and $v \in \highdeg(G)$. Then, in one cascade $k \in [K]$,
\begin{equation}\label{eq:prob-false-positive}
\p\{ |t_u - t_v| \leq (\ell+1) \delta 
\} \leq 3(\ell+1)\delta d.  
\end{equation}
Indeed, write $u \to v$ if $v$ is a descendant of $u$ in the tree of first infections (formed by the set of $n-1$ edges whose Poisson firings infect vertices). Then, on the event $u \to v$, we have that one of the edges neighboring $u$ must fire in $[t_u, t_u + (\ell+1)\delta]$, i.e., 
\[\p\{ ( |t_u - t_v| \leq (\ell+1)\delta ) \cap (u \to v) \}  \leq 
d(\ell+1)\delta.
\] 
On the other hand, if $u \not\to v$, then we can consider a natural coupling of the cascade on $G$ with the cascade on $G \setminus \{u\}$ (with all edges incident to $u$ removed), where we keep the same Poisson processes on all other edges. 
If $u \not\to v$ then $t_v$ is the same on $G$ as on the coupled $G \setminus \{u\}$. Therefore, conditioning on the process on $G \setminus \{u\}$, in the coupled process on $G$, the probability that any edge adjacent to $u$ fires in any fixed interval of length $2(\ell+1)\delta$ is $\leq d2(\ell+1)\delta$. So 
\[ \p \{ (|t_u - t_v| \leq (\ell+1)\delta) \cap (u \not\to v)\} \leq  2(\ell+1)d\delta.\] 
This proves \eqref{eq:prob-false-positive}. Union bounding over low degree vertices $u \in V$, we have
\[
\p\{ \exists u \in V:\ \forall k \in [K],\ \exists v \in \highdeg(G) \text{ such that } |t_u^{(k)} - t_v^{(k)}| \leq  
(\ell+1)\delta\}
\leq n \left(m\cdot {3(\ell+1)d\delta}\right)^K,
\]
which is $o(1)$ if $\delta^K = o(n^{-1})$, i.e., for $K > 2\ell+1 > (2\alpha+1)^{-1}$. 
\end{proof}

\section*{Acknowledgements}
We thank Omer Angel and Tselil Schramm for helpful discussions. 
EM was supported by ONR-MURI-N000142412742, Vannevar Bush Faculty Fellowship ONR-N00014-20-1-2826, 
NSF CCF 1918421, and Simons Investigator award (622132). AS was partially supported by ONR-N00014-20-1-2826 and NSF CCF 1918421. AB was supported supported by NSF GRFP 2141064 and in part by Simons Investigator award (622132).

\bibliographystyle{abbrv}
\bibliography{references}

\end{document}